\documentclass[10pt]{amsart}
\usepackage{amsmath}
\usepackage{amsthm}
\usepackage{amsfonts}
\usepackage{mathrsfs}
\usepackage{amssymb}
\usepackage{amscd}
\usepackage{pgf}
\usepackage{tikz}
\usetikzlibrary{cd}
\usepackage{amscd}

\usepackage{ytableau}

\expandafter\let\csname ver@amsthm.sty\endcsname\relax
\let\theoremstyle\relax

\usepackage[utf8]{inputenc}
\usepackage[
    breaklinks,
    colorlinks,
    citecolor=teal,
    linkcolor=teal,
    urlcolor=teal,
    pagebackref=true,
    hyperindex
    ]{hyperref}
\usepackage{fancyhdr}
\usepackage[
    hscale=0.7,
    vscale=0.75,
    headheight=13pt,
    centering,
    ]{geometry}
\usepackage{amsmath}
\usepackage{amsthm}
\usepackage{amssymb}
\usepackage{mathtools}
\usepackage{stmaryrd}
\usepackage{mathdots}
\usepackage[all,cmtip]{xy}
\usepackage{xcolor}
\usepackage{framed}
\usepackage{pgffor}
\usepackage{fnpct}
\usepackage[capitalize]{cleveref}

\theoremstyle{plain}

\newtheorem{theorem}{Theorem}[section]
\newtheorem{proposition}[theorem]{Proposition}
\newtheorem{lemma}[theorem]{Lemma}
\newtheorem{corollary}[theorem]{Corollary}

\newtheorem{conjecture}[theorem]{Conjecture}

\theoremstyle{definition}

 \newtheorem{Def}[theorem]{Definition}
    \newtheorem{Exa}[theorem]{Example}
    \newtheorem{Rem}[theorem]{Remark}

		\newtheorem{Question}[theorem]{Question}

\numberwithin{figure}{section}

\numberwithin{equation}{section}

\makeatletter

\def\@myMR[#1 #2]{\relax\ifhmode\unskip\spacefactor3000 \space\fi
  \MRhref{#1}{MR\,#1}}
\renewcommand\MR[1]{\@myMR[#1 ]}
\renewcommand{\MRhref}[2]{{\tiny%
  \href{http://www.ams.org/mathscinet-getitem?mr=#1}{#2}}%
}
\renewcommand*{\backref}[1]{}
\renewcommand*{\backrefalt}[4]{%
    \tiny%
    ({%\tiny%
    \ifcase #1 not cited%
          \or cit.\ on p.~#2%
          \else cit.\ on pp.~#2%
    \fi%
    })\\[-.6em]}

\def\maketitle{\par
  \@topnum\z@ % this prevents figures from falling at the top of page 1
  \@setcopyright
  \thispagestyle{empty}% this sets first page specifications
  \ifx\@empty\shortauthors \let\shortauthors\shorttitle
  \else \andify\shortauthors
  \fi
  \@maketitle@hook
  \begingroup
  \@maketitle
  \toks@\@xp{\shortauthors}\@temptokena\@xp{\shorttitle}%
  \toks4{\def\\{ \ignorespaces}}% defend against questionable usage
  \edef\@tempa{%
    \@nx\markboth{\the\toks4
      \@nx\MakeUppercase{\the\toks@}}{\the\@temptokena}}%
  \@tempa
  \endgroup
  \c@footnote\z@
    \renewcommand{\footnoterule}{%
      \kern -3pt
      \hrule width \textwidth height .5pt
      \kern 2pt
    }
  {
    \renewcommand\thefootnote{}
    \vspace{-2em}
    \footnote{
      \par\vspace{-1.2em}\noindent
      \def\@footnotetext##1{\noindent{\footnotesize##1}\par}%
      \let\@makefnmark\relax  \let\@thefnmark\relax
      \ifx\@empty\@date\else \@footnotetext{\@setdate}\fi
      \ifx\@empty\@subjclass\else \@footnotetext{\@setsubjclass}\fi
      \ifx\@empty\@keywords\else \@footnotetext{\@setkeywords}\fi
      \ifx\@empty\thankses\else \@footnotetext{%
        \def\par{\let\par\@par}\@setthanks}%
      \fi
    }
    \addtocounter{footnote}{-1}
  }
  \@cleartopmattertags
}

\def\@adminfootnotes{\@empty}

\def\@settitle{\begin{center}%
  \baselineskip14\p@\relax
    \bfseries
\Large
  \@title
  \end{center}%
}

\def\@setauthors{%
  \begingroup
  \def\thanks{\protect\thanks@warning}%
  \trivlist
  \centering\footnotesize \@topsep30\p@\relax
  \advance\@topsep by -\baselineskip
  \item\relax
  \author@andify\authors
  \def\\{\protect\linebreak}%
  \large{\authors}%
  \ifx\@empty\contribs
  \else
    ,\penalty-3 \space \@setcontribs
    \@closetoccontribs
  \fi
  \endtrivlist
  \endgroup
}

\def\@setaddresses{\par
  \nobreak \begingroup
\footnotesize
  \def\author##1{\end{minipage}\hskip 1sp \begin{minipage}{.5\textwidth}\raggedright%
    ~\\[2em]{\bf##1}\\[.5em]%
  }%
  \interlinepenalty\@M
  \def\address##1##2{\begingroup
    {\ignorespaces##2}\endgroup\\[.5em]}%
  \def\curraddr##1##2{\begingroup
    \@ifnotempty{##2}{\nobreak\indent\curraddrname
      \@ifnotempty{##1}{, \ignorespaces##1\unskip}\/:\space
      ##2\par}\endgroup}%
  \def\email##1##2{\begingroup
    \@ifnotempty{##2}{\nobreak\indent%\emailaddrname
      \@ifnotempty{##1}{, \ignorespaces##1\unskip}%\/:\space
      \ttfamily##2\par}\endgroup}%
  \def\urladdr##1##2{\begingroup
    \def~{\char`\~}%
    \@ifnotempty{##2}{\nobreak\indent\urladdrname
      \@ifnotempty{##1}{, \ignorespaces##1\unskip}\/:\space
      \ttfamily##2\par}\endgroup}%
  \setlength{\parindent}{0pt}%
  \vfill%
  {
  \begin{minipage}{0mm}
  \addresses
  \end{minipage}
  }
  \endgroup
}

\renewcommand{\author}[2][]{%
  \ifx\@empty\authors
    \gdef\authors{#2}%
    \g@addto@macro\addresses{\author{#2}}%
  \else
    \g@addto@macro\authors{\and#2}%
    \g@addto@macro\addresses{\author{#2}}%
  \fi
  \@ifnotempty{#1}{%
    \ifx\@empty\shortauthors
      \gdef\shortauthors{#1}%
    \else
      \g@addto@macro\shortauthors{\and#1}%
    \fi
  }%
}
\edef\author{\@nx\@dblarg
  \@xp\@nx\csname\string\author\endcsname}

\def\@secnumfont{\@empty}

\def\section{\@startsection{section}{1}%
  \z@{.7\linespacing\@plus\linespacing}{.5\linespacing}%
  {\large\bfseries\centering}}

\makeatother

\usepackage{enumerate}
\usepackage{hyperref}

\title{A geometric approach to characters of Hecke algebras}
\author{Alex Abreu and Antonio Nigro}
\date{February 2020}

\newcommand{\dyckpathc}[3]{
\draw[line width=2pt, color=#3] (#1) foreach \dir in {#2}{ -- ++(\dir*90:1)};
}

\usepackage{color}
\definecolor{forestgreen}{rgb}{0.13, 0.55, 0.13}

\newcommand{\col}{\colon}

\newcommand{\X}{\mathcal X}

\newcommand{\F}{\mathcal F}
\newcommand{\Y}{\mathcal Y}
\newcommand{\U}{\mathcal U}

\newcommand{\B}{\mathcal B}
\newcommand{\x}{\mathbf x}
\newcommand{\y}{\mathbf y}
\newcommand{\mS}{\mathcal S}
\newcommand{\He}{\mathcal H}
\newcommand{\T}{\mathcal T}

\newcommand{\mP}{\mathcal{P}}

\newcommand{\D}{\mathcal{D}}

\newcommand{\g}{\mathfrak{g}}
\newcommand{\bl}{\mathfrak{b}}
\newcommand{\ul}{\mathfrak{u}}
\newcommand{\hl}{\mathfrak{h}}

\newcommand{\m}{\mathbf{m}}
\newcommand{\ind}{i}

\DeclareMathOperator{\ch}{ch}

\DeclareMathOperator{\ad}{Ad}
\DeclareMathOperator{\hta}{ht}
\DeclareMathOperator{\wt}{wt}
\DeclareMathOperator{\Bin}{Bin}
\DeclareMathOperator{\Good}{Good}
\DeclareMathOperator{\Shf}{Shv}
\DeclareMathOperator{\Supp}{Supp}

\DeclareMathOperator{\spec}{Spec}
\DeclareMathOperator{\LLT}{LLT}
\DeclareMathOperator{\Fr}{Fr}

\DeclareMathOperator{\sign}{sign}
\DeclareMathOperator{\perm}{perm}
\DeclareMathOperator{\Imm}{Imm}
\DeclareMathOperator{\csf}{csf}
\DeclareMathOperator{\asc}{asc}
\DeclareMathOperator{\Poin}{Poin}
\DeclareMathOperator{\Gr}{Gr}

\newcommand{\bi}{\underline b} %binary word
\newcommand{\si}{{\underline s}} 
\newcommand{\flag}{\mathcal{B}} %flag variety
\newcommand{\hb}{\mathcal{Y}}  % Hessenberg bundle XV_i\subset F_i
\newcommand{\bs}{\mathcal{BS}}  % bott samelson Variety
\newcommand{\h}{\mathcal{Y}}     %hessenberg Variety
\newcommand{\hbs}{\mathcal{Y}}  %hessenberg bott samelson
\newcommand{\Tr}{\Upsilon} %Set of transpositions 
\newcommand{\dw}{\dot{w}} %representative of w\in W 
\newcommand{\ds}{\dot{s}} %representative of s\in S
\newcommand{\tang}{\mathbf{T}}

\begin{document}

\maketitle

\begin{abstract}
    To any element of a connected, simply connected, semisimple complex algebraic group $G$ and a choice of an element of the corresponding Weyl group there is an associated Lusztig variety. 
    When the element of $G$ is regular semisimple, the corresponding variety  carries an action of the Weyl group on its (equivariant) intersection cohomology. From this action, we  recover the induced characters of an element of the  Kazhdan-Lusztig basis of the corresponding Hecke algebra.  In type $A$, we prove a more precise statement: that the Frobenius character of this action is precisely the symmetric function given by the characters of a Kazhdan-Lusztig basis element. The main idea is to find celular decompositions of desingularizations of these varieties and apply the Brosnan-Chow palindromicity criterion for determining when the local invariant cycle map is an isomorphism. 
    This recovers some results of Lusztig about character sheaves and gives a generalization of the Brosnan-Chow \cite{BrosnanChow} solution to the Sharesian-Wachs \cite{ShareshianWachs} conjecture to non-codominant permutations, where singularities are involved. We also review the connections between Immanants, Hecke algebras, and Chromatic quasisymmetric functions of indifference graphs.
    
\end{abstract}

\setcounter{tocdepth}{2}

\tableofcontents

\section{Introduction}

The first part of this introduction is an exposition for a general audience on several adjacent areas relating to characters of Hecke algebras and the chromatic quasisymmetric function. This includes immanants of Jacobi-Trudi matrices and how they relate to Kazhdan-Lusztig elements, and the equivalence of Stembridge's conjecture on Schur positivity of such immanants with the Stanley-Stembridge conjecture on $e$-positivity of chromatic symmetric functions on indifference graphs. We then review the relation between the Hecke algebra and the flag variety (including Springer's proof of the Kazhdan-Lusztig conjecture) and how the chromatic symmetric function can be recovered from the cohomology of Hessenberg varieties. Some historical remarks are also included but the topics are not presented in chronological order, simply in the order that the authors believe will make the material most accessible to readers of different backgrounds. Of course, specialists who find this tedious may proceed directly to Section \ref{sec:results} to see our main results.

  \subsection{Symmetric functions}

     The algebra of symmetric functions $\Lambda$ is defined as:
    \[
    \Lambda:=\lim_{\leftarrow}\mathbb{C}[x_1,\ldots, x_n]^{S_n}.
    \]
    There are three important sets of generators of $\Lambda$, called the \emph{elementary}, \emph{power sum}, and \emph{complete} symmetric functions:
     \begin{align*}
         e_n:=\sum_{i_1<i_2<\ldots <i_n}x_{i_1}x_{i_2}\ldots x_{i_n},\quad\quad         p_n:=\sum_i x_i^n,\quad\quad         h_n:=\sum_{i_1\leq i_2\leq\ldots \leq i_n}x_{i_1}x_{i_2}\ldots x_{i_n}.
    \end{align*}
    Every element of $\Lambda$ can be written as a  polynomial in one of the sets of generators above. For a partition $\lambda=(\lambda_1,\ldots, \lambda_{\ell(\lambda)})$ we write
    \begin{align*}
         e_\lambda:=\prod_{i=1}^{\ell(\lambda)}e_{\lambda_i},\quad\quad         p_\lambda:=\prod_{i=1}^{\ell(\lambda)}p_{\lambda_i},\quad\quad         h_\lambda:=\prod_{i=1}^{\ell(\lambda)}h_{\lambda_i}.
    \end{align*}
    If $\Lambda_n$ is the subspace of $\Lambda$ given by the homogeneous symmetric functions of degree $n$, then $\Lambda_n$ is generated as a $\mathbb{C}$-vector space by either one of $\{e_{\lambda}\}_{\lambda\vdash n}$, $\{p_{\lambda}\}_{\lambda\vdash n}$, $\{h_{\lambda}\}_{\lambda\vdash n}$.   Moreover, these 3 sets of generators are related by the Girard-Newton identities:
      \begin{equation}
          \label{eq:enhnpmu}
          \begin{aligned}
      n!e_n=&\sum_{\mu\vdash n}c(\mu)\sign(\mu)p_\mu,    \\
      n!h_n=&\sum_{\mu\vdash n}c(\mu)p_\mu,
      \end{aligned}
      \end{equation}
    where the sum runs through all partitions $\mu=(\mu_1,\ldots, \mu_{\ell(\mu)})$ of $n$, $c(\mu)$ is the number of permutations in $S_n$ with cycle type $\mu$, and $\sign(\mu)=(-1)^{n-\ell(\mu)}$. These identities resemble the definitions of the determinant and of the permanent of a matrix, which we now recall.\par
        For a square $n\times n$ matrix $A=(a_{ij})$ 
        its \emph{determinant} is defined as:
      \[
      \det(A)=\sum_{w\in S_n} \sign(w)\cdot a_{1w(1)}a_{2w(2)}\ldots a_{nw(n)},
      \]
      where $\sign(w)$ is equal to $1$ if $w$ can be written as the product of an even number of transpositions and $-1$ otherwise. Notice that if $\lambda(w)$ is the cycle type of $w$, then $\sign(w)=\sign(\lambda(w))$. If we remove the sign function in the summation above, we obtain the \emph{permanent} of $A$:
      \[
      \perm(A)=\sum_{w\in S_n} a_{1w(1)}a_{2w(2)}\ldots a_{nw(n)}.
      \]
      
    If we consider the matrix
       \[
     Z_n=\left(\begin{array}{cccccc}
    p_{1}& p_{2} &p_{3}&\cdots&p_{n-1} &p_{n}\\
    1& p_{1} &p_{2}&\cdots&p_{n-2} &p_{n-1}\\
    0 &2& p_{1} &\cdots&p_{n-3}&p_{n-2}\\
    0 & 0 & 3 &\cdots&p_{n-4}&p_{n-3}\\
    \vdots&\vdots &\vdots& \ddots & \vdots&\vdots\\
    0& 0 &0& \cdots & n-1 &p_{1}
        \end{array}
    \right),
    \]
    then the Girard-Newton identities can be written as
      \begin{equation}
          \label{eq:enhnzn}
          \begin{aligned}
      n!e_n=&\det(Z_n),    \\
      n!h_n=&\perm(Z_n).
      \end{aligned}
      \end{equation}

     Another fundamental basis of the symmetric algebra $\Lambda$ is given by the \emph{Schur symmetric functions} $s_{\lambda}$. For each partition $\lambda=(\lambda_1,\ldots, \lambda_{\ell(\lambda)})$ of $n$, we define $s_{\lambda}$ as the limit of the Schur polynomials $s_{\lambda}(x_1,\ldots, x_m)$, which is given by Cauchy's bialternant formula (\cite[Page 111]{Cauchy})
    \[
    s_{\lambda}(x_1,\ldots, x_m):=\frac{\det((x_i^{\lambda_{j}+m-j})_{i,j=1,\ldots,m})}{\det((x_i^{k-j})_{i,j=1,\ldots,m})},
    \]
    where we set $\lambda_j=0$ for every $j>\ell(\lambda)$.   As an example
    \[
    s_{2,1}(x_1,x_2)=\det\left(\begin{array}{cc}
    x_1^3 & x_1^1 \\
    x_2^3 & x_2^1
    \end{array}
    \right)
    \bigg/ \det\left(\begin{array}{cc}
    x_1^1 & x_1^0\\
    x_2^1 & x_2^0
    \end{array}
    \right)=x_1^2x_2+x_1x_2^2.
    \]
    We note that $s_{1,\ldots, 1}=e_n$ and $s_n=h_n$.\par
    
    A century after Cauchy, Schur related  these symmetric functions with the representation theory of the symmetric group.\par
     A \emph{representation} of $S_n$ (over $\mathbb{C}$) is a group homomorphism $\rho\col S_n\to GL(V)$ for some (finite) $\mathbb{C}$-vector space $V$ and we say that $(V,\rho)$, or simply $V$,  is an \emph{$S_n$-module}. The representation is called \emph{irreducible} if $V$, does not contain a proper non-trivial invariant subspace. The irreducible representations of $S_n$ are indexed by partitions $\lambda$ of $n$ and we write $\rho_\lambda\col S_n\to GL(V_{\lambda})$ to denote these irreducible representations. When $\lambda=n$, we have that $\rho_{n}\col S_n\to GL(\mathbb{C})=\mathbb{C}^*$ is the trivial representation $\rho_n(w)=1$. On the other hand,  when $\lambda=(1,1,\ldots,1)$, we have that $\rho_{1,1,\ldots,1}\col S_n\to GL(\mathbb{C})=\mathbb{C}^*$ is the sign representation $\rho_{1,1,\ldots,1}(w)=\sign(w)$. Every representation of $S_n$ can be written uniquely as a direct sum of irreducible representations. \par 

     To each representation $\rho\col S_n\to GL(V)$ there is an associated \emph{character} $\chi\col S_n\to \mathbb{C}$, defined as $\chi(w)=Tr(\rho(w))$. In particular, $\chi$ is a \emph{class function} of $S_n$, that is, a function $\chi\col S_n\to \mathbb{C}$ satisfying $\chi(uwu^{-1})=\chi(w)$ for every $u,w\in S_n$.  We write $\chi^{\lambda}$ for the irreducible characters of $S_n$, that is, those associated to irreducible representations. Notice that $\chi^{n}(w)=1$ and $\chi^{1,1,\ldots,1}(w)=\sign(w)$ for every $w\in S_n$. Every character is a positive integer linear combination of irreducible characters, while every class function is a $\mathbb{C}$-linear combination of irreducible characters. Since a class function $\chi\col S_n\to \mathbb{C}$ only depends on the class of each element $w\in S_n$, we define, for each partition $\mu$ of $n$,  $\chi(\mu):=\chi(w)$ for any (equivalently, every) $w\in S_n$ such that $\lambda(w)=\mu$. \par

    If $\chi^\lambda$ is the irreducible character of $S_n$ associated to a partition $\lambda$ of $n$,  Schur proved that (see \cite[Equation 49]{Schurthesis}) 
    \begin{align}
    \label{eq:slambdachi} n!s_{\lambda}&=\sum c(\mu)\chi^\lambda(\mu)p_\mu,\\
    \label{eq:pmuchi} p_{\mu}&=\sum_{\lambda\vdash n}\chi^\lambda(\mu) s_{\lambda}.    
    \end{align}
    
   As we have seen above, the characters $\chi^{1,\ldots,1}$ and $\chi^{n}$ are the sign and trivial characters, respectively, so that Equation \eqref{eq:enhnpmu} is a particular instance of equation \eqref{eq:slambdachi}.\par
     
       There is also an isomorphism between class functions of $S_n$ and $\Lambda_n$, given by the \emph{Frobenius character map}
      \begin{align*}
      \ch \col\{\text{Class functions of }S_n\}&\to \Lambda_n\\
           f&\mapsto \frac{1}{n!}\sum_{w\in S_n} f(w)p_{\lambda(w)}
      \end{align*}
      By Equation \ref{eq:slambdachi}, we have that $\ch(\chi^{\lambda})=s_{\lambda}$. In particular, if $V$ is an $S_n$-module given by a representation $\rho\col S_n\to GL(V)$ such that $V$ decomposes into irreducible representations as $V=\bigoplus V_\lambda^{\oplus a_\lambda}$, then 
      \[
      \ch(V):=\ch(\chi^{\rho})=\sum_{\lambda\vdash n}a_\lambda s_{\lambda}.
      \]
      
       If $S_{\lambda}:= S_{\lambda_1}\times S_{\lambda_2}\times\ldots\times S_{\lambda_{\ell(\lambda)}}\subset S_n$ is a \emph{Young subgroup} of $S_n$, we consider the representation $\ind_{S_{\lambda}}^{S_n}$ which is the induced representation of the trivial one, defined as
       \begin{align*}
       \ind_{S_{\lambda}}^{S_n}\col S_n&\to GL\big(\bigoplus_{\overline{w}\in S_n/S_\lambda} \mathbb{C}\cdot \overline{w}\big),
       \end{align*}
        where $S_n$ acts on the basis $S_n/S_\lambda$ via left multiplication.
        If $\chi^{\ind_{S_{\lambda}}^{S_n}}$ is the associated character, then it is well-known that $\ch(\chi^{\ind_{S_{\lambda}}^{S_n}})=h_{\lambda}$.\par

      Every class function $f\col S_n\to \mathbb{C}$ can be extended linearly to a complex function on the group algebra $\mathbb{C}[S_n]$, also called a class function.   The algebra $\mathbb{C}[S_n]$ has  basis $\{T_w\}_{w\in S_n}$ as a $\mathbb{C}$-vector space and its multiplication is given by $T_wT_{w'}=T_{ww'}$.\par
        Dual to the construction above, for every element $a\in \mathbb{C}[S_n]$, we define its \emph{(dual) Frobenius character} as\footnote{In MacDonald's book \cite{Macdonald}, he writes $\Psi(a)$ for what we write as $\ch(a)$.}
      \[
      \ch(a):=\sum \chi^\lambda(a)s_{\lambda}.
      \]
      By Equation \eqref{eq:pmuchi}, we have that $\ch(T_w)=p_{\lambda(w)}$ for every $w\in S_n$.    Via Hall's inner product on $\Lambda$, defined as 
      \[
      \langle s_{\lambda},s_{\mu}\rangle =\begin{cases}
      1& \text{ if }\lambda=\mu\\
      0& \text{ otherwise},
      \end{cases}
      \]
      we have that if $f\col \mathbb{C}[S_n]\to \mathbb{C}$ is a class function, then
      \[
      f(a)=\langle \ch(f),\ch(a)\rangle.
      \]

     \subsection{Immanants}
     \label{sec:immanants}
     In Equation \eqref{eq:enhnzn}, we expressed $e_n$ and $h_n$ in terms of the matrix $Z_n$. In fact, the same can be done for $s_{\lambda}$, by using the matrix function called \emph{immanant} by Littlewood-Richardson. The immanant is a generalization of the determinant and of the permanent. For each matrix $A$ and  class function $\chi$ of $S_n$, the $\chi$-immanant of $A$ is given by
      \[
      \Imm_\chi(A):=\sum_{w\in S_n} \chi(w)\cdot a_{1w(1)}a_{2w(2)}\ldots a_{nw(n)}.
      \]
     With this definition, Equation \eqref{eq:slambdachi} can be rewritten as 
     \[
    n!s_{\lambda}=\Imm_{\chi^\lambda}(Z_n),
    \]
    see \cite{LR34}. \par
    
    Schur himself proved one of the first results about immanants,  the so-called Schur dominance theorem which generalizes Hadamard's inequality. It relates immanants of positive-definite matrices with its determinant.
    \begin{theorem}[Schur]
    \label{thm:schur}
    If $A$ is a Hermitian positive-definite matrix and $\chi$ is a character of $S_n$, then
    \[
    \Imm_{\chi}(A)\geq \chi(1)\cdot \det(A).
    \]
     \end{theorem}
    In \cite{Stem91}, Stembridge proved the same result for \emph{totally positive} matrices, that is, real matrices such that each minor is non-negative. 
    \begin{theorem}[Stembridge]
    \label{thm:stemtotally}
    If $A$ is a totally positive matrix and $\chi$ is a character of $S_n$, then
    \[
    \Imm_{\chi}(A)\geq \chi(1)\cdot \det(A).
    \]
    \end{theorem}
    The approach used by Stembridge  to compute the immanants of a matrix $A$ is to consider the following element of the group algebra $\mathbb{C}[S_n]$
     \[
     I_A=\sum_{w\in S_n} T_w\cdot a_{1w(1)}a_{2w(2)}\ldots a_{nw(n)}.
     \]
     Then we have that  
     \[
     \Imm_{\chi}(A)=\chi(I_A).
     \]
     With this point of view, Stembridge  proved (see \cite[Theorem 2.1]{Stem91}) that if $A$ is a totally positive real matrix, then $I_A$ is a non-negative linear combination of monomials in the elements\footnote{Denoted by $x_{[i,j]}$ in \cite{Stem91}.}
    \begin{equation}
    \label{eq:C'ij}
    C'_{[i,j]}=\sum_{w\in S_{[i,j]}}T_w,
    \end{equation}
    where 
    \[
    S_{[i,j]}=\{w\in S_n; w(k)=k,\text{ for every $k\notin[i,j]$}\}.
    \]
    By results of Greene \cite{Greene}, we have that $\chi(C'_{[i_1,j_1]}\cdots C'_{[i_m,j_m]})$ is positive for every character $\chi$. This proves Theorem \ref{thm:stemtotally}.\par
      
    One way to rephrase Theorems \ref{thm:schur} and \ref{thm:stemtotally}, is to consider the Frobenius character 
      \[
      \ch(I_{A})=\sum_{\lambda\vdash n}\Imm_{\chi^\lambda}(A)s_{\lambda}.
      \]
     Then Schur's dominance theorem and Stembridge's theorem say, respectively, that if a matrix $A$ is Hermitian definite positive or totally positive, then $\ch(I_A)$ is \emph{Schur-positive}, that is, it has positive coefficients when written in the Schur basis of $\Lambda$.\par

    Going back to symmetric functions, we can consider the (infinite) Jacobi-Trudi matrix $H=(h_{j-i})_{i,j}$,
    \[
    H=\left(\begin{array}{ccccc}
    1& h_1 &h_2 &h_3& \cdots\\
    0 &1& h_1 &h_2 & \cdots\\
    0&0&1& h_1 & \cdots\\
    0&0&0&1&  \cdots\\
    \vdots &\vdots &\vdots &\vdots&\ddots
    \end{array}\right).
    \]
     For a partition $\lambda$, we consider the Jacobi-Trudi matrix $(h_{\lambda_{i}+j-i})_{i,j=1,\ldots, k}$, which is a minor of the infinite matrix $H$. Then we have the following determinantal formula for the Schur symmetric function, known as the Jacobi-Trudi formula,
    \[
    s_{\lambda}=\det((h_{\lambda_{i}+j-i})_{i,j=1,\ldots, \ell(\lambda)}).
    \]  
    On the other hand, taking a general minor of $H$, we have a \emph{skew} Schur symmetric function
    \[
    s_{\lambda/\mu}:=\det((h_{\lambda_{i}-\mu_j+j-i})_{i,j=1,\ldots, \max\{\ell(\lambda),\ell(\mu)\}}).
    \]
    
    We have that $s_{\lambda/\mu}$ is Schur-positive.\footnote{This follows from the positivity of the Littlewood-Richardson coefficients $c_{\nu\mu}^\lambda$ and the fact that $s_{\lambda/\mu}=\sum c_{\nu\mu}^{\lambda}s_{\nu}$, see \cite{LR34}.}  This means that the matrix $H$ is totally Schur positive, in the sense that every minor of $H$ is a non-negative linear combination of Schur symmetric functions. In particular, every minor $H_{\lambda/\mu}$ is also totally Schur positive. From these, Stembridge made the following conjecture, proved by Haiman shortly after.\footnote{It is also noteworthy to mention that Goulden-Jackson made a similar, but weaker, conjecture in \cite{GouldenJackson}.}
    \begin{theorem}[Conjectured by Stembridge, proved by Haiman]
    \label{thm:Jacobitrudishurpos}
       For every skew partition $\lambda/\mu$ and every character $\chi$ of the symmetric group the following symmetric function
       \[
       \Imm_{\chi}(H_{\lambda/\mu})
       \]
       is Schur-positive.
    \end{theorem}
    More than that, Stembridge conjectured that the \emph{monomial immanants} of $H_{\lambda/\mu}$ are Schur-positive. To give a precise definition, let $n:=\max\{\ell(\lambda),\ell(\mu)\}$ (actually, we can take any $n\geq\max\{\ell(\lambda),\ell(\mu)\}$) and $N:=|\lambda|-|\mu|$.  Consider the following element of the algebra $\mathbb{C}[S_n]\otimes_{\mathbb{C}} \Lambda$
    \[
    I_{\lambda/\mu}=I_{H_{\lambda/\mu}}:=\sum_{w\in S_n}T_w\cdot \prod_{1\leq i \leq n} h_{\lambda_i-\mu_{w(i)}+w(i)-i}.
    \]
    Its Frobenius character is an element of $\Lambda\otimes_{\mathbb{C}} \Lambda$. To avoid confusion we denote by $\x$ and $\y$, respectively, the two sets of variables of $\Lambda\otimes_{\mathbb{C}}\Lambda$. More precisely (keeping the notations of \cite{StanStem}), we have that
    \begin{align*}
    \ch(I_{\lambda/\mu})=F_{\lambda/\mu}(\x,\y):=& \sum_{w\in S_n} p_{\lambda(w)}(\y) \prod_{1\leq i \leq n} h_{\lambda_i-\mu_{w(i)}+w(i)-i}(\x) \\
      =&\sum_{\nu\vdash n} s_{\nu}(\y)\sum_{w\in S_n}\chi^{\nu}(w)\prod_{1\leq i \leq n} h_{\lambda_i-\mu_{w(j)}+w(j)-i}(\x)\\
                                                =&\sum_{\nu\vdash n} s_{\nu}(\y)\Imm_{\chi^{\nu}}(H_{\lambda/\mu}).
    \end{align*}
    
    Writing $F_{\lambda/\mu}$ in the complete symmetric basis $\{h_{\nu}(\y)\}_{\nu\vdash n}$, we have
    \[
        F_{\lambda/\mu}(\x,\y)=\sum_{\nu\vdash n} h_{\nu}(\y)\Imm_{\phi^{\nu}}(H_{\lambda/\mu}),
    \]
    which can be taken as a definition for the monomial immanant $\Imm_{\phi^{\nu}}(H_{\lambda/\mu})$.\par
    \begin{conjecture}[Stembrigde]
       For every skew partition $\lambda/\mu$ and every monomial virtual character $\phi^\nu$, the symmetric function
       \[
       \Imm_{\phi^\nu}(H_{\lambda/\mu})
       \]
       is Schur-positive.
    \end{conjecture}
     The statement that $\Imm_{\phi^{\nu}}(H_{\lambda/\mu})$ is Schur positive for every $\nu$ is equivalent to saying that the coefficient of $h_{\nu}(\y)s_{\nu'}(\x)$ in $F_{\lambda/\mu}(\x,\y)$ is non-negative for every $\nu\vdash n$ and every $\nu'\vdash N$. Expanding $F_{\lambda/\mu}(\x,\y)$ in the Schur basis $\{s_{\nu'}(\x)\}_{\nu'\vdash N}$, we can write
     \begin{equation}
         \label{eq:FEnu}
     F_{\lambda/\mu}=\sum_{\nu'\vdash N}E_{\lambda/\mu}^{\nu'}(\y)s_{\nu'}(\x).
    \end{equation}
     Then the Stembridge conjecture about monomial immanants is equivalent to the following conjecture.
     \begin{conjecture}[Stanley-Stembridge]
     \label{conj:StanStem}
     The symmetric functions $E^{\nu'}_{\lambda/\mu}$ are $h$-positive for every $\lambda$, $\mu$ and $\nu'$.
     \end{conjecture}
     
     The easiest coefficient $E^{\nu'}_{\lambda/\mu}$ to compute is when $\nu'=N$. We know that
     \begin{align*}
         F_{\lambda/\mu}(\x,\y)=&\sum_{w \in S_n}p_{\lambda(w)}(\y)\prod_{1\leq i \leq n} h_{\lambda_i-\mu_{w(i)}+w(i)-i}(\x).
     \end{align*}
    Since the coefficient of $s_N(\x)$ in the Schur expansion of $\prod_{1\leq i \leq n} h_{\lambda_i-\mu_{w(i)}+w(i)-i}(\x)$ is $1$ if $\lambda_i-\mu_{w(i)}+w(i)-i\geq0$ for every $i$, and $0$ otherwise, we have that
    \begin{equation}
        \label{eq:Enup}
    E_{\lambda/\mu}^N(\y)=\sum_{w\in S_{\lambda/\mu}}p_{\lambda(w)}(\y).
    \end{equation}
    where $S_{\lambda/\mu}$ is the subset of $S_n$ consisting of permutations $w$ such that $\lambda_i-\mu_{w(i)}+w(i)-i\geq 0$ for every $i$. Alternatively, if we define the function $\m\col[n]\to [n]$ by
    \begin{equation}
        \label{eq:mHess}
    \m(j):=\max\{i;i-\lambda_{i}+\mu_j-j\leq 0\},
    \end{equation}
    then 
    \[
    S_{\lambda/\mu}=S_{\m}:=\{w\in S_n; w^{-1}(j)\leq \m(j)\text{ for every }j\in [n]\}.
    \]
    Moreover, if $S_{\lambda/\mu}$ is non-empty, then $\m\col[n]\to [n]$ is a \emph{Hessenberg function}, that is $\m(j)\geq j$ and $\m(j+1)\geq \m(j)$ for every $j\in[n]$. In fact, $E_{\lambda/\mu}^N$ only depends on the Hessenberg function $\m$.\par
        Following the notation in \cite{StanStem}, we can actually define a partition $\sigma=(\sigma_1,\ldots, \sigma_{\ell(\sigma)})$ by $\sigma_k=\max\{i; \m(i)\leq n-k\}$, and let $B_\sigma$ be the board obtained from an $n\times n$ board by removing the Young diagram of $\sigma$ from the upper right corner. Then the set $S_{\lambda/\mu}$ is precisely the set of permutations whose associated rook placements lay in $B_\sigma$. Stanley-Stembridge considered $Z[B_\sigma]:=\sum p_{\lambda(w)}$, the \emph{cycle indicator} of $B_\sigma$, and noted that $E_{\lambda/\mu}^N=Z[B_\sigma]$.
    
    \begin{figure}[htb]
            \begin{tikzpicture}
\begin{scope}[scale=0.4]
\draw[fill=black, fill opacity=0.1, draw opacity=0] (0,2) -- (0,6) -- (4,6) -- (4,5) -- (3,5) -- (3,4) -- (1,4) -- (1,2) -- (0,2);
\draw[help lines] (0,0) grid +(6,6);
\dyckpathc{0,0}{1,1,0,1,1,0,0,1,0,1,0,0}{blue}
\node at  (0.5,1.5) [shape=circle, fill=black, inner sep=1pt] {};
\node at  (1.5,0.5) [shape=circle, fill=black, inner sep=1pt] {};
\node at  (2.5,3.5) [shape=circle, fill=black, inner sep=1pt] {};
\node at  (3.5,4.5) [shape=circle, fill=black, inner sep=1pt] {};
\node at  (4.5,5.5) [shape=circle, fill=black, inner sep=1pt] {};
\node at  (5.5,2.5) [shape=circle, fill=black, inner sep=1pt] {};
\end{scope}
\end{tikzpicture}
\caption{The graph of $\m=(2,4,4,5,6)$ is depicted in blue. The Young diagram of $\sigma=(4,3,1,1)$ is depicted as the shaded area. The bullets correspond to the permutation $214563$. The board $B_\sigma$ is the set of  boxes below the blue line.}
    \end{figure}
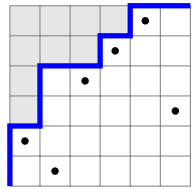

      Alternatively, since $\ch(T_w)=\ch(T_{w^{-1}})$ for every $w\in S_n$ and $\ch(T_w)=p_{\lambda(w)}$, we observe that
     \[
     E_{\lambda/\mu}^N=\ch(C'_\m),
     \]
    where $C'_\m:=\sum_{w \in S_\m} T_w$. \par
    
    Stanley-Stembridge conjectured that every symmetric function $E_{\lambda/\mu}^{\nu}$ is a positive integer linear combination of cycle-indicators.
    \begin{conjecture}[Stanley-Stembridge]
    \label{conj:StanStemsum}
      For every triple $\lambda,\mu,\nu'$ of partitions, there exist partitions $\sigma_1,\ldots, \sigma_k$ such that
      \[
      E_{\lambda/\mu}^{\nu`}=\sum_{1\leq j\leq k} Z[B_{\sigma_j}].
      \]
      Equivalently, there exist Hessenberg functions $\m_1,\ldots, \m_k$ such that
      \[
      E_{\lambda/\mu}^{\nu'}=\sum_{1\leq j\leq k} \ch(C'_{\m_j}).
      \]
    \end{conjecture}
    If Conjecture \ref{conj:StanStemsum} holds, then Conjecture \ref{conj:StanStem} is reduced to the case $\nu'=N$.

     \subsection{Immanants and Hecke algebras} 
      The elements $C'_{[i,j]}$ and $C'_{\m}$ appearing in the previous section are actually particular instances of a more general construction due to Kazhdan-Lusztig (see \cite{KL}). Since every $w\in S_n$ can be written as a product of simple transpositions $(i,i+1)$, we have that the group algebra $\mathbb{C}[S_n]$ can also be described as the $\mathbb{C}$-algebra generated by $\{T_{s}\}$, where $s$ runs through all simple transpositions, with the following relations
      \begin{align*}
          T_s^2=&1&&\text{ for every simple transposition $s$,}\\
          T_sT_{s'}=&T_{s'}T_s&&\text{ for every $s=(i,i+1)$ and $s'=(j,j+1)$ such that $|i-j|>1$,}\\
          T_sT_{s'}T_s=&T_{s'}T_sT_{s'}&&\text{ for every  $s=(i,i+1)$ and $s'=(j,j+1)$ such that $|i-j|=1$.}
      \end{align*}
      The group algebra $\mathbb{C}[S_n]$  admits a $q$-deformation called the \emph{Hecke algebra} and denoted by $H_n$, constructed as follows. Since $H_n$ will have the same generators as $\mathbb{C}[S_n]$ but with slightly different relations, we will abuse the notation and still write $T_s$ for the generators of $H_n$ and write $T_s(1)$ for the generators of $\mathbb{C}[S_n]$. More precisely, the algebra $H_n$ is a $\mathbb{C}(q^{\frac{1}{2}})$-algebra\footnote{Usually the definition is over $\mathbb{Z}[q^{\frac{1}{2}},q^{-\frac{1}{2}}]$.} generated by $T_s$, with the following relations
    \begin{align*}
          T_s^2=&(q-1)T_s+q&&\text{ for every simple transposition $s$,}\\
          T_sT_{s'}=&T_{s'}T_s&&\text{ for every $s=(i,i+1)$ and $s'=(j,j+1)$ such that $|i-j|>1$,}\\
          T_sT_{s'}T_s=&T_{s'}T_sT_{s'}&&\text{ for every  $s=(i,i+1)$ and $s'=(j,j+1)$ such that $|i-j|=1$.}
      \end{align*}
      When we set $q=1$, we recover the group algebra $\mathbb{C}[S_n]$. For elements $a\in H_n$, we write $a(1)$ for the image of $a$ via the specialization $H_n\to\mathbb{C}[S_n]$ given by sending $q$ to $1$ and $T_s$ to $T_s(1)$.\par
       We know that each $w\in S_n$ has a (non-unique) reduced expression $w=s_1s_2\ldots s_{\ell(w)}$ as product of simple transpositions. We define 
      \[
      T_w:=T_{s_1}T_{s_2}\ldots T_{s_{\ell(w)}},
      \]
      which does not depend on the choice of the reduced expression of $w$.
      As a $\mathbb{C}(q^{\frac{1}{2}})$-vector space, we have that $\{T_w\}_{w\in S_n}$ is a basis of $H_n$. We say that $\ell(w)$ is the \emph{length} of $w$, and it is equal to the number of inversions of $w$.\par

        Before we introduce the Kazhdan-Lusztig basis, we need to define the \emph{Bruhat order} of $S_n$. Given $z,w\in S_n$, we say that $z\leq w$ if for some (equivalently, for every) reduced expression $w=s_1\ldots s_{\ell(w)}$ there exist $1\leq i_1<i_2<\ldots< i_k\leq \ell(w)$ such that $z=s_{i_1}\ldots s_{i_k}$. \par

      The Hecke algebra $H_n$ has an involution $\iota$ given by
      \begin{align*}
      \iota\col H_n&\to H_n\\
      q^{\frac{1}{2}}&\mapsto q^{-\frac{1}{2}}\\
      T_w&\mapsto T_{w^{-1}}^{-1}.
      \end{align*}
       The Kazhdan-Lusztig basis $\{C'_w\}_{w\in S_n}$ of $H_n$ is defined by the properties
      \begin{equation}
          \label{eq:C'wdef}
            \begin{aligned}
      \iota(C'_w)&=C'_w,\\
      q^{\frac{\ell(w)}{2}}C'_w&=\sum_{z\leq w}P_{z,w}(q)T_z,    
      \end{aligned}
      \end{equation}
      where $P_{z,w}(q)\in \mathbb{Z}[q]$, $P_{w,w}(q)=1$, and $\deg(P_{z,w})<\frac{\ell(w)-\ell(z)}{2}$ for every $z\neq w$. The existence of such a basis is established in \cite{KL} and the polynomials $P_{z,w}(q)$ are called \emph{Kazhdan-Lusztig polynomials}.\par
       We have that the elements considered in Equation \eqref{eq:C'ij} satisfy
       \[
       C'_{[i,j]}=C'_{w_{[i,j]}}(1)
       \]
      where 
      \[
      w_{[i,j]}(k)=\begin{cases}
       k& \text{ if } k\notin [i,j]\\
       j-k+i& \text{ if } k\in [i,j].
      \end{cases}
      \]
     and the elements $C'_\m=\sum_{w\leq\m} T_w$ satisfy
     \[
     C'_\m=C'_{w_\m}(1)
     \]
     where $w_\m$ is  the lexicographically greatest permutation satisfying $w_\m(i)\leq \m(i)$ for all $i\in [n]$. 
     \begin{Def}
     \label{def:codominant}
     The permutations of the form $w_\m$ for some Hessenberg function $\m$ are called \emph{codominant} permutations.
     \end{Def}
     
     Recall that in Section \ref{sec:immanants} we have defined the element $I_{\lambda/\mu}$ of $\mathbb{C}[S_n]\otimes \Lambda$ associated to the Jacobi-Trudi matrix $H_{\lambda/\mu}$.    Haiman proved the following in \cite{Haiman}.

    \begin{theorem}[Haiman]
    \label{thm:haiman}
    The element $I_{\lambda/\mu}$ is a non-negative integer linear combination of the basis elements $C'_{w}(1)\otimes s_{\nu'}(x)$, where $\{C'_w(q)\}$ is the Kazhdan-Lusztig basis of the Hecke algebra $H_n$.
    \end{theorem}
    Each irreducible $\mathbb{C}$-representation of $S_n$ lifts to an irreducible $\mathbb{C}(q^{\frac{1}{2}})$-representation of $H_n$ (see \cite[Theorem 8.1.7]{GeckPf}). Hence if $\chi$ is an irreducible character of $S_n$ and, abusing notation, $\chi$ is the corresponding character of $H_n$, we have that 
    \[
    \chi(a(1))=\chi(a)(1),
    \]
    for every $a\in H_n$. We define the \emph{(dual) Frobenius character} of an element $a\in H_n$ by
    \[
    \ch(a)=\sum \chi^{\lambda}(a)s_{\lambda}(x)\in \mathbb{C}(q^{\frac{1}{2}})\otimes \Lambda.
    \]
    In \cite[Lemma 1.1]{Haiman} Haiman proved that $\chi^{\lambda}(q^{\frac{\ell(w)}{2}}C'_w)$ is a symmetric unimodal  polynomial in $q$ with non-negative integer coefficients. This result, together with Theorem \ref{thm:haiman}, gives a proof of Theorem \ref{thm:Jacobitrudishurpos}. We note that \cite[Lemma 1.1]{Haiman} says that $\ch(q^{\frac{\ell(w)}{2}}C'_w)$ is Schur-positive, in the sense that its coefficients in the Schur-basis are polynomials in $q$ with non-negative integer coefficients.

    Haiman also made some conjectures. The first regarding positivity of monomial characters of Kazdhan-Lusztig elements.
     
    \begin{conjecture}[Haiman]
    \label{conj:haimanhpos}
     For every $w\in S_n$, the Frobenius character of the Kazhdan-Lusztig elements $\ch(q^{\frac{\ell(w)}{2}}C'_w)$ is $h$-positive.
    \end{conjecture}
     Conjecture \ref{conj:haimanhpos} imply Conjecture \ref{conj:StanStem} via Theorem \ref{thm:haiman}. Before we state the others, we need a definition. 
    \begin{Def}
    \label{def:smoothperm}
    A permutation $w\in S_n$ is \emph{smooth} if $P_{e,w}(q)=1$, or equivalently, if $w$ avoids the patterns $3412$ and $4231$. A  permutation that is not smooth will be called \emph{singular}.
    \end{Def}

    \begin{conjecture}[Haiman]
    \label{conj:haimansmooth}
    If $w$ is a smooth permutation, then there exists a codominant permutation $w'$ such that 
    \[
    \ch(C'_w)=\ch(C'_{w'}).
    \]
    \end{conjecture}
    
    \begin{conjecture}[Haiman]
    \label{conj:haimansingular}
    Let $w$ be a singular permutation, then there exists codominant permutations $w_1,\ldots, w_k$ such that
    \[
    \ch(C'_w)=\ch(C'_{w_1})+\ch(C'_{w_2})+\dots, \ch(C'_{w_k})
    \]
    and\footnote{The condition on the Kazhdan-Lusztig polynomials is a consequence of the character equality.}
    \[
    P_{e,w}(q)=\sum_{1\leq i\leq k}q^{\frac{\ell(w)-\ell(w_i)}{2}}.
    \]
    \end{conjecture}
 
 Together with Theorem \ref{thm:haiman}, Conjectures \ref{conj:haimansmooth} and \ref{conj:haimansingular} imply the Stanley-Stembridge conjecture \ref{conj:StanStemsum} by taking $q=1$ and recalling that if $w$ is a codominant permutation then $(C'_w)(1)=C'_{\m}$ for some Hessenberg function $\m$. Conjecture \ref{conj:haimansmooth} was first proved combinatorially by Clearman-Hyatt-Shelton-Skandera in \cite{CHSS}. In an upcoming work we provide a geometric proof of the same result and a counter-example to Conjecture \ref{conj:haimansingular}.

    \subsection{Chromatic symmetric functions}
    Another interpretation of the symmetric functions $E_{\lambda/\mu}^N$ (recall equations \eqref{eq:FEnu} and \eqref{eq:Enup}) can be given in terms of colorings of particular graphs.\par
    A graph is called an \emph{indifference graph} if its vertex set is $[n]:=\{1,2,\ldots, n\}$ for some natural number $n$ and such that if $\{i,j\}$ is an edge with $i<j$, then $\{i,k\}$ and $\{k,j\}$ are also edges for every $k$ such that $i<k<j$.  Indifference graphs are naturally associated with Hessenberg functions. The graph $G_{\m}$ induced by a Hessenberg function $\m$ is the graph with vertex set $[n]$ and edge set $\{(i,j); i<j\leq \m(i)\}$. Every indifference graph arises in this way.\par
     
     The chromatic polynomial of a graph admits a symmetric function generalization introduced by Stanley in \cite{Stan95}, called the \emph{chromatic symmetric function}. Given a graph $G$, it is defined as
     \[
     \csf(G):=\sum_{\kappa}x_{\kappa},
     \]
     where $x_{\kappa}:=\prod_{v\in V(G)} x_{\kappa(v)}$ and the sum runs through all proper colorings of the vertices  $\kappa\col V(G)\to \mathbb{N}$. A coloring $\kappa$ is proper if $\kappa(v)\neq\kappa(v')$ whenever $v$ and $v'$ are adjacent. We have that $\csf(G)$ lies in $\Lambda$.\par
     For $\m$ defined as in Equation \eqref{eq:mHess}, the chromatic symmetric function of the graph $G_{\m}$ is related to the symmetric function $E_{\lambda/\mu}^N$ via the following equality (see \cite[Proposition 4.1 and Section 5]{StanStem})
      \[
      E_{\lambda/\mu}^N=\omega(\csf(G_{\m})),
      \]
      where $\omega\col \Lambda\to \Lambda$ is the involution defined by $\omega(e_n)=h_n$.\footnote{Equivalently, we could define $\omega$ by $\omega(p_n)=(-1)^{n-1}p_n$, or alternatively, by $\omega(s_{\lambda})=s_{\lambda^{t}}$, where $\lambda^t$ is the transposed partition of $\lambda$.}
     The special case of the Stanley-Stembridge conjecture (Conjecture \ref{conj:StanStem}) corresponding to $\nu'=N$ can hence be restated as:
     \begin{conjecture}
     For every Hessenberg function $\m$, we have that $\csf(G_{\m})$ is $e$-positive.
     \end{conjecture}
     The conjecture above was also extended (see \cite{Stan95}) to more general graphs, namely  incomparability graphs of $3+1$-free posets, but it was reduced to the case of indifference graphs by Guay-Paquet in \cite{GPmodular}.\par
     A $q$-analogue of the function $\csf$ is the \emph{chromatic quasisymmetric function} introduced in \cite{ShareshianWachs} by Shareshian-Wachs. For a graph $G$ with vertex set $[n]$, the chromatic quasisymmetric function $\csf_q(G)$ is defined as
    \[
    \csf_q(G):=\sum_{\kappa}q^{\asc_G(k)} x_\kappa,
    \]
    where the sum runs through all proper colorings of $G$  and
    \[
    \asc_G(\kappa):=|\{(i,j);i<j, \kappa(i)<\kappa(j), \{i,j\}\in E(G)\}|
    \]
    is the number of ascents of the coloring $\kappa$. For a general graph the quasisymmetric chromatic function is usually  not  symmetric, however when $G$ is an indifference graph, then $\csf_q(G)$ is symmetric. By results of Clearman-Hyatt-Shelton-Skandera \cite{CHSS}, we have that if $w_{\m}\in S_n$ is the codominant permutation associated to the Hessenberg function $\m$, then
    \begin{equation}
        \label{eq:CHSS}
    \ch(q^{\frac{\ell(w)}{2}}C'_{w_\m})=\omega(\csf_q(G_{\m})).
    \end{equation}
    In a subsequent work we will give a proof of this equality using the modular law for the chromatic quasisymmetric function.

\subsection{The Hecke algebra and the flag variety}
\label{subsec:heckeflag}
        Although in the previous section we defined the Hecke algebra as a $q$-deformation of the group algebra $\mathbb{C}[S_n]$, historically the Hecke algebra was constructed as follows.\par
        
        Let $G$ be any finite group and $B$ a subgroup of $G$. We define $\He(G,B)$ as the set of $B$-bi-invariant functions $\phi\col G\to\mathbb{C}$, that is
        \[
        \He(G,B)=\{\phi\col G\to \mathbb{C}; \phi(hgh')=\phi(g)\text{ for every }g\in G,h,h'\in B\}.
        \]
        We equip $\He(G,B)$ with a multiplication given by the convolution
        \[
        (\phi_1\cdot \phi_2)(g)=\frac{1}{|G|}\sum_{a\in G}\phi_1(ga^{-1})\phi_2(a).
        \]
        As $\mathbb{C}$-vector space, we can find a basis for $\He(G,B)$ given by the characteristic functions of the double cosets $B\backslash G/ B$. That is, for each $g\in G$, we consider the set $BgB$, and define
        \[
        \phi_g(g')=\begin{cases}
          1&\text{ if }g'\in BgB\\
          0& \text{ otherwise}.
        \end{cases}
        \]
        Of course, $\phi_g$ only depends on the double coset $BgB$. \par
        
         Iwahori described the Hecke algebra when $G=SL_n(\mathbb{F}_q)$, for some finite field $\mathbb{F}_q$ and $B$ the Borel subgroup of upper-triangular matrices, although with a slightly different normalization for the convolution given by
       \[
        (\phi_1\cdot \phi_2)(g)=\frac{1}{|B|}\sum_{a\in G}\phi_1(ga^{-1})\phi_2(a).
        \]

        For each permutation $w\in S_n$, we choose a matrix $M_w\in SL_n(\mathbb{F}_q)$ with all entries $0$ except for the entries $(M_w)_{w(j)j}$ (the matrix $M_w$ is essentially the permutation matrix associated to $w$, but modified in such a way to satisfy $\det(M_w)=1$). Then the set $\{M_w\}_{w\in S_n}$ is a complete set of representatives for $B\backslash G/ B$ (this is the \emph{Bruhat decomposition} of $G$). This means that for each $w\in S_n$, we have an associated characteristic function $\phi_w:=\phi_{M_w}$. Iwahori proved that $\He(G,B)$ has the following multiplication rules (see \cite{Iwahori}).
        \begin{align*}
          \phi_s^2=&(q-1)\phi_s+q&&\text{ for every simple transposition $s$,}\\
          \phi_s\phi_{s'}=&\phi_{s'}\phi_s&&\text{ for every $s=(i,i+1)$ and $s'=(j,j+1)$ such that $|i-j|>1$,}\\
          \phi_s\phi_{s'}\phi_s=&\phi_{s'}\phi_s\phi_{s'}&&\text{ for every  $s=(i,i+1)$ and $s'=(j,j+1)$ such that $|i-j|=1$,}\\
          \phi_w=&\phi_{s_1}\phi_{s_2}\ldots \phi_{s_{\ell(w)}}&&\text{ where }w=s_1s_2\ldots s_{\ell(w)}\text{ is a reduced expression for $w$.}
      \end{align*}
        We just point out that $q$ in the equation above is a power of a prime number instead of being a variable as in the previous section. \par
        
        Unsurprisingly, the Hecke algebra $H_n$ is related to the general linear group $G=GL_n(\mathbb{C})$ and its Borel subgroup $B$ of upper triangular matrices. The remainder of this section is devoted to describe this relation.\par
        We begin by defining the \emph{flag variety} $\B:=G/B$. It can be identified as the projective variety parametrizing flags of vector subspaces, that is
       \[
       \B=\{V_1\subset V_2\subset\ldots\subset V_n=\mathbb{C}^n; \dim_{\mathbb{C}}(V_i)=i\},
       \]
       often writing $V_\bullet$ instead of $V_1\subset\ldots \subset V_n$.   This identification is constructed as follows: For each matrix $g\in G$ with columns $(v_1\; v_2\;\ldots\; v_n)$, we associate the flag $V_\bullet$ given by $V_i=\langle v_1,\ldots, v_i\rangle$. It is clear that every matrix in $gB$ gives rise to the same flag.
       
       If $F_\bullet$ is the flag given by the identity matrix, that is $F_i=\langle e_1,\ldots, e_i\rangle$ and $V_\bullet^{w}$ is the flag given by the permutation matrix $M_w$, then 
       \[
       \dim V_i^w\cap F_j= |\{k;k\leq i,w(k)\leq j\}|=:r_{i,j}(w).
       \]
       Moreover, the same equality holds for every flag associated to a matrix $g\in BM_wB$.  For a permutation $w$ and a fixed flag $F_\bullet$, the \emph{Schubert variety} associated to $w$ and $F_\bullet$ is defined as
       \begin{align*}
       \Omega_{w,F_\bullet}&:=\{V_\bullet; \dim V_i\cap F_j\geq r_{i,j}(w)\}.
       \end{align*}
       We also set
       \begin{align*}
       \Omega^\circ_{w,F_\bullet}&:=\{V_\bullet; \dim V_i\cap F_j= r_{i,j}(w)\}=BM_wB/B.
       \end{align*}
       We have that $\Omega_{w,F_\bullet}=\overline{\Omega_{w,F_\bullet}^\circ}$ and $\Omega_{w,F_\bullet}=\bigcup_{z\leq w}\Omega_{z,F_\bullet}^\circ$. In particular,  $\{\Omega_{w,F_\bullet}^\circ\}_{w\in S_n}$ stratify $\B$.
       
       There are also relative versions, which we call the \emph{relative Schubert varieties}:
       \begin{align*}
       \Omega_w&:=\{(F_\bullet,V_\bullet); \dim V_i\cap F_j\geq r_{i,j}(w)\}\subset \B\times \B,\\
       \Omega^\circ_w&:=\{(F_\bullet,V_\bullet); \dim V_i\cap F_j= r_{i,j}(w)\}\subset \B\times \B.
       \end{align*}
         As before, we have that $\Omega_{w}=\overline{\Omega_{w}^\circ}$, $\Omega_{w,}=\bigcup_{y\leq w}\Omega_{y,}^\circ$, and  $\{\Omega_{w}^\circ\}_{w\in S_n}$ stratify $\B\times \B$. \par
        
       \begin{Exa}
       \label{exa:schub}
       We note that the inequalities defining Schubert varieties are somewhat redundant. For example, for $n=4$, we have that
       \begin{align*}
       \Omega_{3412}&=\{(F_\bullet,V_\bullet); V_1\subset F_3, F_1\subset V_3\},\\
       \Omega_{2341}&=\{(F_\bullet,V_\bullet); V_1\subset F_2, V_2\subset F_3\},\\
       \Omega_{4231}&=\{(F_\bullet,V_\bullet); \dim V_2\cap F_2\geq 1\}.
       \end{align*}
       \end{Exa}
        In Definition \ref{def:smoothperm} we defined smooth permutations, terminology derived from the fact that the Schubert variety $\Omega_w$ is smooth if and only if the permutation $w$ is smooth. More generally, the Kazhdan-Lusztig polynomial $P_{z,w}(q)$ is a ``measure'' of how singular $\Omega_{w}$ is at each point in $\Omega_{z}^\circ$.    This can be made precise with the following construction described in Springer's article \cite{Springer}.
        
        We recall that if $\X$ is a complex algebraic variety, its singular cohomology (with complex coefficients) is the same as the (\v Cech) cohomology of the trivial sheaf $\mathbb{C}_\X$, that is
        \[
        H^i(\X,\mathbb{C})=H^i(\mathbb{C}_\X).
        \]
        However, singular cohomology is not well behaved for singular varieties, for instance Poincaré duality does not hold in general. In \cite{GoreskyMacphersonI} and  \cite{GoreskyMacphersonII}, Goresky-Macpherson introduced the notion of intersection cohomology, which agrees with singular cohomology for smooth varieties and  is better behaved in presence of singularities. Sheaf theoretically, this corresponds to replacing the trivial sheaf $\mathbb{C}_\X$ with the so-called \emph{intersection cohomology complex} $IC_\X$.  The intersection cohomology group $IH^*(\X,\mathbb{C})$ is then defined as the hypercohomology $\mathcal{H}^*(IC_\X)$ of the intersection cohomology sheaf. We refer to Section \ref{sec:decomp} for a brief review of bounded derived categories, perverse sheaves, and intersection cohomology. 
        
        \par
        
        We consider the bounded derived category $\D=D^{b}_{cons}(\B\times \B)$ of sheaves on $\B\times \B$ of $\mathbb{C}$-vector spaces cohomologically constructible with respect to the stratification given by relative Schubert varieties. For each relative Schubert variety $\Omega_w$ its intersection cohomology complex $IC_{\Omega_w}$ is an object of $\D$ (after pushing forward to $\B\times \B$).
        
         We have a map
         \[
         h\col \mathbf{Objects}(\D)\to H_n
         \]
         given by
         \[
         h(A_\bullet)=\sum_{w\in S_n}\sum_{i\in\mathbb{Z}}\dim(H^i(A_\bullet)_{(V^{e}_\bullet,V^w_{\bullet})})q^{\frac{i}{2}}T_w,
         \]
         where $V^{w}_\bullet$ is the flag associated to the permutation matrix $M_w$ and $H^i(A_\bullet)_{(V_\bullet^{e},V_\bullet^w)}$ is the stalk at the point $(V_\bullet^{e},V_\bullet^w)\in \Omega_w^\circ\subset  \flag\times\flag$ of the $i$-th cohomology sheaf $H^i(A_\bullet)$ of the complex $A_\bullet$.\par
         It is worth noticing that if 
         \begin{equation}
         \label{eq:AC}
         A_\bullet= \ldots 0 \to \mathbb{C}_{\Omega_w}\to 0 \to \ldots
         \end{equation}
         is the complex associated to the trivial sheaf $\mathbb{C}_{\Omega_w}$ supported in degree 0, then 
         \[
         \dim H^i(A_\bullet)_{(V^{e}_\bullet, V^{z}_{\bullet})}=\begin{cases}
         1& \text{ if $i=0$ and $z\leq w$,}\\
         0& \text{ otherwise}.
         \end{cases}
         \]
         Consequently,
         \[
         h(A_\bullet)=\sum_{z\leq w} T_z.
         \]
        The right-hand side above is precisely $q^{\frac{\ell(w)}{2}}C'_w$ when $w$ is smooth. In general, we have the following theorem.
        \begin{theorem}[Springer]
        \label{thm:springer}
        If $IC_{\Omega_w}$ is the intersection cohomology sheaf of the relative Schubert variety $\Omega_w$, then
        \begin{equation}
            \label{eq:springerthm}
        h(IC_{\Omega_w})=q^{-\frac{\dim \B}{2}}C'_w.
        \end{equation}
        In particular, the Kazhdan-Lusztig polynomials $P_{z,w}(q)$ are equal (up to multiplication by a power of $q$\footnote{This power of $q$ and the  one appearing in Equation \eqref{eq:springerthm} vary depending on how we index the degrees on which perverse sheaves are supported. See \cite[Section 2.3]{GoreskyMacphersonII} for a list of possibilities on  how to  index  complexes of sheaves and see Section \ref{sec:decomp} for our choice of indexing scheme.}) to $\sum_{i\in\mathbb{Z}}\dim(H^i(IC_{\Omega_w})_{(V^{e}_\bullet,V^z_{\bullet})})q^{\frac{i}{2}}$.
        \end{theorem}

        Since the argument in the proof of Springer's theorem plays a crucial role in our results, we will give an idea of the proof. As usual when working with singular varieties, we will consider a desingularization, which in this case is the Bott-Samelson resolution of $\Omega_w$. It is constructed as follows. Let $\si=(s_1,\ldots, s_{\ell(w)})$ be a reduced word for $w$, that is $s_i=(k_i,k_i+1)$ are simple transpositions and $w=s_1\ldots s_{\ell(w)}$. Define
        \[
        \bs_{\si}:=\{(V_{\bullet}^0,V_{\bullet}^1,\ldots, V_{\bullet}^{\ell(w)}); V_j^{i-1}=V_j^{i} \text{ for every }i,j \text{ with }j\neq k_i\}\subset \B\times \B\times\ldots \times \B.
        \]
      
      We have a natural map $\alpha_\si\col \bs_\si\to \flag\times\flag$ given by projecting to the first and to the last factors whose image is precisely $\Omega_w$. We will also denote by $\alpha_\si$ the induced morphism $\alpha_\si\col \bs_\si\to \Omega_w$. 
      \begin{Exa}
      If $w=2341\in S_4$, then we can choose $\si=((1,2),(2,3),(3,4))$, to which it corresponds the resolution
      \[
      \bs_\si=\left\{ \begin{array}{c}
                        F_1\subset F_2\subset F_3,\\
                        V_1\subset F_2\subset F_3,\\
                        V_1\subset V_2\subset F_3,\\
                        V_1\subset V_2\subset V_3
                        \end{array}\right\}.
      \]
      From this description we can see that the image $\alpha_\si(\bs_\si)\subset \B\times \B$ is precisely the set of pairs $(F_\bullet, V_\bullet)$ such that $V_1\subset F_2$ and $V_2\subset F_3$, that is, $\alpha_\si(\bs_\si)=\Omega_{2341}$.
      \end{Exa}

      The map $\alpha_\si$ has a Whitney stratification given by $\Omega_z^\circ$ for $z\leq w$. Since the varieties $\Omega_z^\circ$ are simply connected, we have that the decomposition theorem (see Theorem \ref{thm:decomp}) reads
   \begin{equation}
 \label{eq:decompCBS}
   \alpha_{\si*}(\mathbb{C}_{\bs_\si}[\dim \bs_\si])=\bigoplus_{z\leq w} IC_{\Omega_z}\otimes L_z,
   \end{equation}
   where $\mathbb{C}_{\bs_\si}[\dim \bs_\si]$ is the trivial complex (as in Equation \eqref{eq:AC}) shifted $\dim \bs_\si$ times to the left, and $L_z=\bigoplus_{i\in \mathbb{Z}}L_{z,i}$ are graded vector spaces. Let
   \[
   \Poin(L):=\sum_{i\in\mathbb{Z}} \dim(L_i)q^{\frac{i}{2}}
   \]
   be the \emph{Poincaré-Hilbert series} of a graded vector space $L=\bigoplus_{i\in \mathbb{Z}} L_i$ and define
   \[
   Q_{z,\si}(q^{\frac{1}{2}}):=\Poin(L_z).
   \]
   Applying the map $h$ to Equation \eqref{eq:decompCBS}, we have
   \begin{equation}
   \label{eq:hCBS}
   h(\alpha_{\si*}(\mathbb{C}_{\bs_\si}[\dim \bs_\si]))=\sum_{z\leq w} h(IC_{\X_z}) Q_{z,\si}(q^{\frac{1}{2}}).
   \end{equation}
   Since $\alpha_\si$ is proper and birational, we have that $Q_{w,\si}(q^{\frac{1}{2}})=1$ and $Q_{z,\si}(q^{-\frac{1}{2}})= Q_{z,\si}(q^{\frac{1}{2}})$ (see \cite{Springer} for more details). Moreover, the left-hand side of Equation \eqref{eq:hCBS}  can be explicitly computed as
    \[
    h(\alpha_{\si*}(\mathbb{C}_{\bs_\si}[\ell(w)]))=q^{-\frac{\ell(w)}{2}}\prod_{1\leq i\leq \ell(w)}(1+T_{s_i}).
    \]
    Substituting in Equation \eqref{eq:hCBS}, we have that
    \[
    q^{-\frac{\ell(w)}{2}}\prod_{1\leq i\leq n}(1+T_{s_i})=q^{\frac{\dim \B}{2}}\big(h(IC_{\Omega_w})+\sum_{z<w}h(IC_{\Omega_z})Q_{z,\si}(q^{\frac{1}{2}})\big).
    \]
    Proceeding by induction, we can assume that $q^{\frac{\dim \B}{2}}h(IC_{\Omega_z})=C'_z$ for $z<w$, which gives us that 
    \[
    q^{\frac{\dim \B}{2}}h(IC_{\Omega_w})=q^{-\frac{\ell(w)}{2}}\prod_{1\leq i\leq \ell(w)}(1+T_{s_i})-\sum C'_zQ_{z,\si}(q^{\frac{1}{2}}).
    \]
    This proves that $q^{\frac{\dim \B}{2}}h(IC_{\Omega_w})$ satisfies the conditions in Equation \eqref{eq:C'wdef}, from which we conclude that $q^{\frac{\dim \B}{2}}h(IC_{\Omega_w})=C'_w$.\par
    
     Theorem \ref{thm:springer} provides a deep connection between Schubert varieties and the Kazhdan-Lusztig elements $C'_w$. A natural question would be:
     \begin{Question}
     What is the relation between the Schubert varieties and the characters $\ch(q^{\frac{\ell(w)}{2}}C'_w)$?
     \end{Question} 
     Haiman observed that the coefficient of $s_{n}$ in $\ch(q^{\frac{\ell(w)}{2}}C'_w)$, which is the evaluation at $q^{\frac{\ell(w)}{2}}C'_w$ of the trivial character $\chi^n$,  satisfies
    \begin{equation}
    \label{eq:chPoin}
    \chi^n(q^{\frac{\ell(w)}{2}}C'_w)=\Poin(IH^*(\Omega_{w,F_{\bullet}})).
    \end{equation}
     In the next section  we will introduce Lusztig varieties, whose geometry completely characterizes $\ch(q^{\frac{\ell(w)}{2}}C'_w)$.

    \subsection{Subvarieties of the flag varieties}
    \label{sec:HesseDL}
 
     We begin with an example of how the chromatic symmetric function can be recovered from the geometry of certain varieties.
     
          \begin{Exa}
          \label{exa:losevmanin}
          Let $n$ be a positive integer and consider $\h_n$  the toric variety given by the Weyl chambers of $S_n$, the Weyl group of type $A_{n-1}$ (see  \cite{Procesi}). Another description of $\h_n$ is to begin with $\mathbb{P}^{n-1}=\{(x_1:\ldots:x_n)\}$ and perform a sequence of blowups in the following order.
          \begin{enumerate}
              \item Blow up $\mathbb{P}^{n-1}$ along the points $P_i:=(0:\ldots:0:1:0:\ldots:0)=V(x_1,\ldots, \widehat{x_i},\ldots, x_n)$ for $i=1,\ldots, n$.
              \item Then blow up the resulting variety along the strict transforms of the lines 
              \[
              \ell_{ij}:=V(x_1,\ldots, \widehat{x_i},\ldots, \widehat{x_j},\ldots, x_n)
              \]
              for $1\leq i<j\leq n$.
              \item  Now blow up the strict transforms of planes and so on.
          \end{enumerate}

         Since $S_n$ acts on the Weyl chambers of $S_n$, we have that $S_n$ acts on the variety $\h_n$ and hence on its cohomology ring as well. More explicitly, $S_n$ acts on $\mathbb{P}^{n-1}$ by permuting the coordinates, and all the loci we are blowing up are invariant under this action, so this action lifts to an action of $S_n$ on $\h_{\m}(X)$.\par

        Stanley in \cite[Proposition 12]{Stanley86}, based on a recurrence of Procesi in \cite{Procesi},
        observed that the Frobenius character of the action of $S_n$ on $H^*(\h_n)$ coincides with $\omega(\csf(G_\m))$, where $\m=(2,3,4,\ldots, n-1,n,n)$ is a Hessenberg function (in this case $G_\m$ is the path graph with $n$ vertices). Moreover Shareshian-Wachs noted that if we take the grading of $H^*(\h_n)$ into account we have that
       \[
       \omega(\csf_q(G_\m))=\sum_{ i} \ch\big(H^{2i}(\h_n)\big)q^i.
       \]
       
       If $n=4$, then $\h_n$ is obtained by blowing up $\mathbb{P}^3$ along the four points $(1:0:0:0)$, $(0:1:0:0)$, $(0:0:1:0)$, $(0:0:0:1)$, and then along the strict transforms of the six lines $(*:*:0:0)$, $(*:0:*:0)$, $(*:0:0:*)$, $(0:*:*:0)$, $(0:*:0:*)$, and $(0:0:*:*)$. The cohomology of $\h_{n}$ has basis described in Table \ref{tab:Hgen} below.
        \begin{table}[h]
  \begin{tabular}{c|c|c}
    Codimension   &  Generators & Frob. Char.\\
    \hline
    0    &   $\h_{n}$ & $h_4$, \\
    2  & $H$, $E_{ij}$, $E_i$ & $h_4+h_{2,2}+h_{3,1}$\\
    4  & $L$, $L_{ij}$, $L_i$ & $h_4+h_{2,2}+h_{3,1}$\\
    6 & $pt$ & $h_4$.
  \end{tabular}
  \caption{Generators of $H^*(\h_n)$ and the Frobenius character of the action of $S_n$ in each codimension. Here $H$ and $L$ are, respectively, the pullback of the hyperplane and line classes of $\mathbb{P}^3$. The classes $E_{ij}$ and $E_i$ are, respectively, the classes of the exceptional divisors over the lines $\ell_{ij}$ and over the points $P_i$. The classes $L_{ij}$ and $L_i$ are, respectively, the classes of lines inside  $E_{ij}$ and $E_i$.}
  \label{tab:Hgen}
  \end{table}
       
       The group $S_4$ acts on $H^*(\h_n)$ by fixing $\h_{n}(X)$, $H$, $L$ and $pt$ and permuting the indices of $E_{ij}$, $E_{i}$, $L_{ij}$ and $L_i$. We can decompose this action as a sum of representations $\ind_{S_{\lambda}}^{S_n}$, from which we get
    \[
    \sum_i \ch(H^{i}(\h_\m(X)))q^{\frac{i}{2}}=(1+q+q^2+q^3)h_4+(q+q^2)h_{3,1}+(q+q^2)h_{2,2},
    \]
    while 
    \[
    \csf_q(G_{(2,3,4,4)})=(1+q+q^2+q^3)e_4+(q+q^2)e_{3,1}+(q+q^2)e_{2,2}.
    \]
    \end{Exa}

    The variety $\h_n$ in Example \ref{exa:losevmanin} can be seen as a subvariety of the flag variety. Indeed, if $X$ is any regular semisimple $n\times n$ matrix, then $\h_n$ is isomorphic to the variety
    \[
    \{V_\bullet; XV_j\subset V_{j+1}\text{ for }j=1,\ldots, n-1\}.
    \]
    
    Generalizing the definition above, a \emph{Hessenberg variety} associated to the Hessenberg function $\m$ and the $n\times n$ matrix $X$ is the subvariety $\h_\m(X)$ of the flag variety $\flag$ defined by
     \[
     \h_\m(X):=\{V_{\bullet}; XV_j\subset V_{\m(j)}\}.
     \]

     Hessenberg varieties where introduced by De Mari-Prochesi-Shayman in \cite{MPS}. When $X$ is regular semisimple (that is, when $X$ has distinct eigenvalues), we have that $\h_\m(X)$ is smooth. 
    
        To avoid cumbersome notations in the sequel, we will define
        \begin{equation}
            \label{eq:chgraded}
        \ch(L):=\sum \ch(L_i)q^{\frac{i}{2}},
        \end{equation}
        where $L=\bigoplus_{i\in \mathbb{Z}} L_i$ is a graded $S_n$-module (equivalently, a graded representation).\par

     Given a more general Hessenberg function, we still have a natural action of $S_n$ on $H^*(\h_\m(X))$ for $X$ a regular semisimple matrix. This action was described in \cite{Tym08} by Tymozcko and  is called the \emph{dot action}. The Shareshian-Wachs conjecture, now proved by Brosnan-Chow in \cite{BrosnanChow} and by Guay-Paquet in \cite{GP}, states that:
     \begin{theorem}
     \label{thm:brosnanchow}
      For a Hessenberg function $\m$ and a regular semisimple $n\times n$ matrix $X$, we have that
      \[
      \omega(\csf_q(G_{\m}))=\ch(H^*(\h_{\m}(X))).
      \]
     \end{theorem}        
     Before giving a brief idea of Brosnan-Chow's proof, we need to introduce yet another basis of $\Lambda$. The monomial symmetric functions $m_{\lambda}$ are defined as the sum of all distinct monomials whose sequence of exponents is a permutation of $\lambda$. For example
     \[
     m_{211}=x_1^2x_2x_3+x_1x_2^2x_3+x_1x_2x_3^2+x_1^2x_2x_4+\ldots.
     \]
     This basis relates to the Frobenius character map  and to its dual as follows: If $L$ is a graded $S_n$-module and $L^{S_{\lambda}}$ denotes the subspace of $L$ invariant by the action of $S_{\lambda}$, then
     \begin{equation}
         \label{eq:chL}
     \ch(L)=\sum_{\lambda\vdash n}\Poin(L^{S_{\lambda}})m_{\lambda}.
     \end{equation}
     Moreover, if $a\in H_n$, then 
     \begin{equation}
         \label{eq:cha}
     \ch(a)=\sum_{\lambda\vdash n} \chi^{\ind_{S_{\lambda}}^{S_n}}(a)m_{\lambda}.
    \end{equation}
     Brosnan-Chow's proof of the Shareshian-Wachs conjecture has two steps. First, they prove that (see \cite[Theorem 35]{BrosnanChow}) 
     \[
     \omega(\csf_q(G_\m))=\sum_{\lambda\vdash n}\Poin(H^*(\h_{\m}(X_{\lambda})))m_{\lambda},
     \]
     where $X_{\lambda}$ is a regular matrix with Jordan block decomposition $\lambda$, that is, each eigenvalue has a single Jordan block, and the sizes of the blocks make up the partition $\lambda$. To do that, they give a combinatorial description of the coefficients of $m_{\lambda}$ in $\omega(\csf_q(G_{\m})))$ and check that they agree with the Poincaré polynomial of $\h_{\m}(X_{\lambda})$, which was computed by Tymoczko  in \cite{Tym06}.\par
     For the second step, they prove a deep palindromicity result about invariant cycles (see Theorem \ref{thm:invmappalin}), which implies the isomorphism
     \[
     H^*(\h_{\m}(X_{\lambda}))\cong H^*(\h_{\m}(X))^{S_{\lambda}}.
     \]

       Via equation \eqref{eq:CHSS}, Theorem \ref{thm:brosnanchow} gives a geometric description of $\ch(q^{\frac{\ell(w)}{2}}C'_{w_\m})$. This can be extended to every permutation $w$. Let $X\in GL_n(\mathbb{C})$\footnote{Here it is important that $X$ is an invertible matrix, while in the definition of Hessenberg varieties we could consider any matrix $X$. In general, this means that Lusztig varieties are defined for elements of the algebraic group, while Hessenberg varieties are defined for elements of the Lie algebra.}, $w\in S_n$ be an invertible matrix and a permutation, then the \emph{Lusztig variety} associated to $X$ and $w$ is defined as
    \begin{align*}
    \h_w(X)^\circ&:=\{V_\bullet; \dim XV_i\cap V_j=r_{i,j}(w)\},\\
    \h_w(X)&:=\overline{\h_w(X)^\circ}.
    \end{align*}
    When $w=w_{\m}$ is the codominant permutation associated to $\m$, then the Lusztig variety $\h_w(X)$ coincides with the Hessenberg variety $\h_\m(X)$.

     The varieties $\h_w(X)$ are usually singular (even when $X$ is regular semisimple), so it is natural to look for an analogue of the Bott-Samelson resolution. Let $X\in GL_n(\mathbb{C})$ and let $\si=(s_1,\ldots, s_{\ell(w)})$ be a reduced word for $w$ with $s_i=(k_i,k_i+1)$, and define
     \begin{equation}
         \hbs_\si(X):=\{(V_\bullet^0,\ldots, V_\bullet^{\ell(w)}); V_j^{i-1}=V_j^{i}  \text{ for every }i,j \text{ with }j\neq k_i, XV_\bullet^0=V_\bullet^n\}.
     \end{equation}

     Lusztig actually considered the relative versions 
      \[
      \h_w:=\overline{\{(X,V_\bullet);\dim XV_i\cap V_j=r_{i,j}(w)\}},
      \]
      which comes with a forgetful map
      \begin{align*}
      f\col \h_w &\to GL_n(\mathbb{C})\\
             (X,V_\bullet)&\mapsto X.
      \end{align*}
      and studied how $f_{*}(IC_{\h_w})$  decomposes into simple perverse sheaves, which are called \emph{character sheaves} (see Section \ref{sec:chashv}).

   These varieties are characteristic 0 analogues of the so-called Deligne-Lusztig varieties. If instead of working over the complex numbers we were over an algebraically closed field $K$ of positive characteristic $p$, and if instead of $X$ we would have considered the so-called \emph{Frobenius morphism} $\Fr$, then the variety $\h_w(\Fr)^\circ$ is a \emph{Deligne-Lusztig variety} introduced by Deligne-Lusztig in \cite{DeligneLusztig} to study representations of $GL_n(\mathbb{F}_q)$. We briefly recall their construction. \par

      Let $q=p^r$ be a power of $p$, then the Frobenius homomorphism $\Fr\col K\to K$ is defined by $\Fr(x)=x^q$. The fixed locus of $\Fr$ is precisely $\mathbb{F}_q$, the finite field with $q$ elements. We also denote by $\Fr$ the map 
      \begin{align*}
          K^n&\to K^n\\
          (x_1,\ldots, x_n)&\mapsto (x_1^q,\ldots, ,x_n^q).
      \end{align*} 
        If we consider the flag variety $\flag_n(K)=\{V_1\subset V_2\subset\ldots\subset V_n=K^n\}$, we have an induced Frobenius morphism $ \flag_n(K)\to \flag_n(K)$, that takes a flag $V_\bullet$ to the flag $\Fr(V_\bullet)$ obtained by applying $\Fr$ to each vector space $V_i$. \par
        For a permutation $w\in S_n$, the Deligne-Lusztig variety associated to $w$ is defined as
        \[
        DL_w:=\{V_\bullet; \dim \Fr(V_i)\cap V_j=r_{i,j}(w)\}.
        \]
    It is a locally closed subvariety of $\flag_n(K)$.  Since for each element $X\in GL_n(\mathbb{F}_q)$, we have that $\Fr(XV)=X\Fr(V)$ for every subspace $V\subset K^n$, we have an action of $GL_n(F_q)$ on $DL_w$. This action lifts to an action of $GL_n(\mathbb{F}_q)$ on the $l$-adic cohomology with compact supports of $DL_w$, which produces virtual representations of $GL_n(\mathbb{F}_q)$ (see \cite[Definition 1.5]{DeligneLusztig}).
     \begin{Exa}
       As in Example \ref{exa:losevmanin}, if $w=234\ldots n1$ is the codominant permutation associated to the Hessenberg function $\m=(2,3,4,\ldots, n-1,n,n)$, then $DL_w$ is isomorphic to the Drinfeld's halfspace $\mathbb{P}^{n-1}_K\setminus \bigcup H$, where the union is over all $\mathbb{F}_q$-rational hyperplanes, see \cite[Section 2.2]{DeligneLusztig}. Actually, the Deligne-Lusztig construction is indeed a generalization of a construction of Drinfeld  for $SL_2$ (see \cite{Drinfeld}).
       
       Going back to Example \ref{exa:losevmanin}, we can see that the open locus $\h_w(X)^\circ$ is isomorphic to $\mathbb{P}^{n-1}_{\mathbb{C}}\setminus\bigcup_{i=1}^n H_i$, where $H_i=V(x_i)$.
     \end{Exa}
    
      Although the varieties $\h_w(X)^\circ$ and $DL_w$ are similarly defined, there is not a direct translation of results. For instance, one of the tools in the study of Deligne-Lusztig varieties is Lang's theorem, which states that the map
       \begin{align*}
       GL_n(K)&\to GL_n(K)\\
         g&\mapsto g^{-1}F(g)
       \end{align*}
      is finite and surjective. On the other hand, the image of the map
       \begin{align*}
       GL_n(\mathbb{C})&\to GL_n(\mathbb{C})\\
         g&\mapsto g^{-1}Xg
       \end{align*}
     is the conjugacy class of $X$ and its fibers are isomorphic to the centralizer $Z_G(X)$.

    Also noteworthy is that when $e\in S_n$ is the identity and $X$ unipotent (that is, $X-I$ is nilpotent), then the variety $\h_{e}(X)$ is the so-called \emph{Springer fiber} at $X$.

\subsection{Results}
\label{sec:results}

   The purpose of this paper is to use the techniques in Brosnan-Chow's proof of the Shareshian-Wachs conjecture to recover the relation between the Frobenius characters of the Kazhdan-Lusztig basis elements of the Hecke algebra and the geometry of Lusztig varieties.

    Since some of the results hold more generally, we make the following definition.
    \begin{Def}
    \label{def:twisted}
        Let $G$ be a connected, simply-connected,  semisimple algebraic group, $B$ its  Borel subgroup, $W$ its associated Weyl group $W$  and $S$ the set of simple reflections of $W$. Denote by $\flag:=G/B$ the flag variety of $G$. For each $X\in G$, $w\in W$, and $\si=(s_1,\ldots, s_n)$ a word in $S$, define the \emph{Lusztig varieties} and their Bott-Samelson resolutions by\footnote{When $X$ is regular we have that $\h_w(X)=\overline{\h_w(X)^\circ}.$}
          \begin{align*}
          \h_w(X)^\circ&:=\{gB; g^{-1}Xg\in B\dw B\} \subset \flag\\
          \h_w(X)&:=\{gB; g^{-1}Xg\in \overline{B\dw B}\}\\
          \hbs_\si(X)&:=\{(g_0B,g_1B,\ldots, g_nB); g_0^{-1}Xg_n\in B, g_{i+1}^{-1}g_i\in B\cup B\ds_i B\}
          \end{align*}
         where $\dw, \ds_i\in G$ are representatives of $w$ and $s_i$ respectively.\par
    \end{Def}

         \begin{theorem}
         \label{thm:mainsimply}
         Let $G$ be a connected, simply connected, semisimple algebraic group over $\mathbb{C}$, $W$ its associated Weyl group, and $S$ the set of simple reflections in $W$. Given $X$ a regular semisimple element of $G$ and $w\in W$, there exists a natural monodromy action of $W$ on the intersection cohomology group $IH^*(\h_w(X))$. Moreover we have
         \[
         \chi^{\ind_{W_J}^W}(q^{\frac{\ell(w)}{2}}C'_w)=\Poin(IH^*(\h_w(X))^{W_J})
         \]
         for every $J\subset S$.  In particular, when $G=SL_n(\mathbb{C})$ and $W=S_n$, we have
         \[
         \ch(q^{\frac{\ell(w)}{2}}C'_w)=\ch(IH^*(\h_w(X)))
         \]
         for every $w\in S_n$.
         \end{theorem}
         Note that the second part of Theorem \ref{thm:mainsimply} follows directly from the first part and from Equations \eqref{eq:chL} and \eqref{eq:cha}. When $w = w_\m$ is codominant, we recover Theorem \ref{thm:brosnanchow}. 
        
        \begin{Exa}
        \label{exa:coxeter}
         Let $w\in S_n$ be a Coxeter element of $S_n$, that is, $w$ is the product of all simple transpositions of $S_n$ in some order, and let $X$ be a regular semisimple matrix. Then $\h_w(X)$ is isomorphic to the toric variety of Weyl chambers of $S_n$ as in Example \ref{exa:losevmanin}. In particular, for each Coxeter element, we have that $\h_w(X)$ is isomorphic to $\h_{\m}(X)$ for $\m=(2,3,\ldots, n-1,n,n)$.\par
           For instance, assume that $X$ is a diagonal matrix and consider $w=3142=(2,3)\cdot (1,2)\cdot (3,4)$, then
          \[
          \h_{w}(X)=\{V_\bullet;V_1\subset XV_2, XV_2\subset V_3\}.
          \]
          We consider the varieties $\widehat{\h_w}(X)=\{(V_1,V_2); V_1\subset V_2, V_1\subset XV_2\}$, $\mathbb{P}^3=\text{Gr}(1,4)=\{V_1\subset \mathbb{C}^4\}$ and the forgetful map
          \begin{align*}
          f_1\col \widehat{\h_w}(X)&\to \mathbb{P}^3\\
              (V_1,V_2)&\mapsto V_1.
          \end{align*}
          We have that $f_1^{-1}(V_1)=\{(V_1,V_2); V_2\supseteq V_1+X^{-1}V_1\}$, and $f_1$ is an isomorphism on the locus where $X^{-1}V_1\neq V_1$, since in this case $\dim(V_1+X^{-1}V_1)=2$. If $X^{-1}V_1=V_1$, then $f_1^{-1}(V_1)$ is isomorphic to $\mathbb{P}^2=\{V_2; V_1\subset V_2\subset \mathbb{C}^4\}$. The locus in $\mathbb{P}^3$ where $X^{-1}V_1=V_1$ is precisely the set of points $\{(1:0:0:0),(0:1:0:0),(0:0:1:0),(0:0:0:1)\}$ corresponding to the eigenvectors of $X$. This means that $\widehat{\h_w}(X)$ is isomorphic to the blow-up of $\mathbb{P}^3$ along these points.\par
            Now consider the forgetful map
          \begin{align*}
          f_2\col \h_w(X)&\to \widehat{\h_w}(X)\\
              (V_1,V_2,V_3)&\mapsto (V_1,V_2).
          \end{align*}
        As before, we have $f_2^{-1}(V_1,V_2)=\{(V_1,V_2,V_3);V_3\supset V_2+XV_2\}$ and $f_2$ is an isomorphism on the locus where $XV_2\neq V_2$. Notice that $V_1\subset XV_2\cap V_2$ for every $(V_1,V_2)\in \widehat{\h_w}(X)$, and hence $\dim(V_2+XV_2)\leq 3$ with equality if and only if $XV_2\neq V_2$. Moreover, if $XV_2=V_2$, then $f_2^{-1}(V_1,V_2)$ is isomorphic to $\mathbb{P}^1=\{V_3; V_2\subset V_3\subset \mathbb{C}^4\}$. The condition $XV_2=V_2$ implies that $V_2=\langle e_i,e_j\rangle $ for some distinct $i,j\in \{1,2,3,4\}$. This means that $V_1\in \mathbb{P}^3$ belongs to one of the lines 
        \[
        (*:*:0:0), (*:0:*:0), (*:0:0:*), (0:*:*:0), (0:*:0:*), (0:0:*:*)\subset \mathbb{P}^3.
        \]
        Moreover, the locus $\{(V_1,V_2);XV_2=V_2\}$ will be precisely the union of the strict transforms of these lines. This means that $f_2$ is the blow-up of $\widehat{\h_w}(X)$ at these strict transforms.\par
         For Coxeter elements $w\in S_n$, one can compute $\ch(q^{\frac{\ell(w)}{2}}C'_w)$ as in \cite[Proposition 4.2]{Haiman}. It does not depend on the Coxeter element and agrees with $\omega(\csf_q(G_{\m}))$ for $\m=(2,3,\ldots, n-1,n,n)$.
        \end{Exa}

        \begin{Exa}
        \label{exa:twisted3412}
       Let us make an example for a singular permutation. Let $w=3412$ be a permutation in $S_4$ and let  $X$ be a  regular semisimple diagonal matrix. Let us describe the variety $\h_w(X)$. In this case,  we have that (see Example \ref{exa:schub})
        \[
        \h_w(X)=\{V_\bullet; XV_1\subset V_3, V_1\subset XV_3\}.
        \]
       Consider the image $\widehat{\h_w}(X)$ of $\h_w(X)$ in the partial flag variety $\{(V_1,V_3); V_1\subset V_3\subset \mathbb{C}^4\}$, that is
       \[
        \widehat{\h_w}(X)=\{(V_1,V_3);  XV_1\subset V_3, V_1\subset XV_3\}.
       \]
      Then the map $\h_w(X)\to\widehat{\h_w}(X)$ is a $\mathbb{P}^1$-bundle.\par
      We take a similar approach as in example \ref{exa:coxeter} and consider the forgetful map
     \begin{align*}
     f_1\col \widehat{\h_w}(X)&\to \text{Gr}(1,4)=\mathbb{P}^3=\{(x_1:x_2:x_3:x_4)\}\\
     (V_1,V_3)&\mapsto V_1.
      \end{align*}
    The fibers of $f_1$ over a point $V_1\in \text{Gr}(1,4)$ can be described as
      \[
     f_1^{-1}(V_1)=\{(V_1,V_3); V_3\supset X^{-1}V_1+V_1+XV_1\}.
     \]
       However, we have that $\dim(X^{-1}V_1+V_1+XV_1)\leq 2$ if and only if $V_1\subset \langle e_i,e_j\rangle$ for some distinct $i,j\in\{1,2,3, 4\}$. The condition  $V_1 \subset \langle e_i,e_j \rangle$ means that $V_1$ lies inside one of the lines  $\ell_{ij}$
       \[
        (*:*:0:0), (*:0:*:0), (*:0:0:*), (0:*:*:0), (0:*:0:*), (0:0:*:*)\subset \mathbb{P}^3.
        \]
        This proves that $f_1$ is an isomorphism outside of the lines $\ell_{ij}$.
     We leave to the reader to see that $f_1$ is the blowup of $\mathbb{P}^3$ along the union of all these lines.  We note that the natural action of $S_4$ on $\mathbb{P}^3$ lifts to an action of $S_4$ on $\widehat{\h_w}(X)$, since we are blowing-up an invariant locus.\par
       If $w'=3142$ as in Example \ref{exa:coxeter}, then we have seen that
       \[
       \h_{w'}(X)=\{V_\bullet;V_1\subset XV_2; XV_2\subset V_3\}.
       \]
      We can consider the forgetful morphism
      \begin{align*}
          f_2\col \h_{w'}(X)&\to \{(V_1,V_3); V_1\subset V_3\subset \mathbb{C}^4\}.
      \end{align*}
      Then $f_2(\h_{w'}(X))=\widehat{\h_w}(X)$, because $X^{-1}V_1\subset V_2\subset V_3$ and $XV_1\subset XV_2\subset V_3$. The fibers of $f_2$ can be described as 
      \[
      f_2^{-1}(V_1,V_3)=\{(V_1,V_2,V_3); V_2\supset V_1+X^{-1}V_1,\, V_2\subset V_3\cap X^{-1}V_3\},
      \]
     In particular, $f_2$ is a birational map, and it is an isomorphism outside of the locus $\{(V_1,V_3); X^{-1}V_1=V_1,X^{-1}V_3=V_3\}$. This locus is a finite set of points and the fibers of $f_2$ over these points are isomorphic to $\mathbb{P}^1$. This means that $f_2$ is a small resolution, in the sense of \cite[Section 1.1]{BorhoMacpherson}.\par

    Since $f_2$ is a small resolution, we have that $IH^*(\widehat{\h_w}(X))=H^*(\h_{w'}(X))$, and the latter cohomology group was described in Example \ref{exa:coxeter} and \ref{exa:losevmanin}. Then
     \[
     \ch(IH^*(\widehat{\h_w}(X))=\ch(H^*(\h_{w'}(X)))=(1+q+q^2+q^3)h_4+(q+q^2)h_{3,1}+(q+q^2)h_{2,2}
     \]
  Since $\h_w(X)\to \widehat{\h_w}(X)$ is a $\mathbb{P}^1$-bundle, we have that
  \begin{align*}
  \ch(IH^*(\h_w(X)))&=(1+q)\ch(IH^*(\widehat{\h_w}(X)))\\
                    &=(1+2q+2q^2+2q^3+q^4)h_4+(q+2q^2+q^3)h_{3,1}+(q+2q^2+q^3)h_{2,2}
  \end{align*}
  which agrees with $\ch(q^{2}C'_w)$ (see for instance \cite[Table 1]{Haiman}).
  \end{Exa}

        To prove Theorem \ref{thm:mainsimply}, we would like to follow Brosnan-Chow's proof of Theorem \ref{thm:brosnanchow}. However, several problems appear. For example, even when $X$ is regular semi-simple we have that $\h_w(X)$ is singular whenever $w$ is singular. This drives us to substitute usual cohomology for intersection cohomology. Even so, we would have to find affine pavings for $\h_w(X)$, which is simply not possible when $w$ is singular.
        
        To overcome these difficulties, the idea is to work with the  variety $\hbs_\si(X)$ (which is a desingularization of $\h_w(X)$ when $X$ is regular semisimple).  Via arguments similar to Springer's proof of the Kazhdan-Lusztig conjecture (see Sections \ref{subsec:heckeflag} and \ref{subsec:flag}), Theorem \ref{thm:mainsimply} becomes a consequence of the following theorem.
        \begin{theorem}
        \label{thm:hbsall}
        Let $G$ be a connected, simply connected, semisimple algebraic group over $\mathbb{C}$, $W$ its associated Weyl group, and $S$ the set of simple reflections in $W$. Let $X$ be a regular element of $G$ and $\si=(s_1,\ldots, s_\ell)$ a reduced word of simple reflections of $W$. The following hold:
        \begin{enumerate}
            \item The  variety $\hbs_\si(X)$ is paved by affines. In particular $H^i(\hbs_\si(X))=0$ for every odd $i$.
            \item If the semisimple  part $X_s$ of $X$ has Levi subgroup $G_J$ for some $J\subset S$, then 
            \[
        \chi^{\ind_{W_J}^W}(\prod_{i=1}^\ell (1+T_{s_i}))=\Poin (H^*(\hbs_\si(X))).
        \]
        In particular, the  variety $\hbs_\si(X)$ has palindromic Betti numbers.
        \item If $X$ is regular semisimple, then there exists a natural monodromy action of $W$ on the cohomology group $H^*(\hbs_\si(X))$ and
        \[
        \chi^{\ind_{W_J}^W}(\prod_{i=1}^{\ell} (1+T_{s_i}))=\Poin (H^*(\hbs_\si(X))^{W_J}).
        \]
        In particular, if $G=SL_n(\mathbb{C})$ and $W=S_n$, then
        \[
        \ch(\prod_{i=1}^\ell(1+T_{s_i}))=\ch (H^*(\hbs_\si(X))).
        \]
        \end{enumerate}
         \end{theorem}
     
          For Hessenberg varieties,   item (1) of Theorem \ref{thm:hbsall}  is done in \cite{Tym06} for type $A$ and \cite{Precup13} for general algebraic groups (the palindromicity of Betti numbers is also proved in the latter work). \par

      To prove Theorem \ref{thm:hbsall}, we proceed as follows.  In Theorem \ref{thm:TrTr} we give a combinatorial interpretation of $\chi^{\ind_{W_J}^W}(\prod (1+T_{s_i}))$ as the generating function of pairs $(w,\bi)\in W\times \{0,1\}^\ell$ satisfying ``good'' properties (see Definition \ref{def:good}). In Theorem \ref{thm:hbspaving} we prove that the varieties $\hbs_\si(X)$ have a paving by affines where each affine is indexed by a ``good'' element $(w,\bi)$, and in particular, its Poincaré polynomial coincides with $\chi^{\ind_{W_J}^W}(\prod (1+T_{s_i}))$. This proves items (1) and (2) of Theorem \ref{thm:hbsall}.\par
        To prove item (3) of Theorem \ref{thm:hbsall}, we apply the same arguments in Brosnan-Chow proof of Theorem \ref{thm:brosnanchow}. Via the local invariant cycle theorem (see Theorem  \ref{thm:invmap}) and Brosnan-Chow's palindromicity result (see Theorem \ref{thm:invmappalin}), we have that $H^*(\hbs_\si(X))^{W_J}\cong H^*(\hbs_\si(X_{J}))$, where $X_J$ is a regular element of $G$ such that its semisimple part has Levi subgroup $G_J$. This is done in Section \ref{sec:proofs}. The last statement of item (3) of Theorem \ref{thm:hbsall} follows from Equations \eqref{eq:chL} and \eqref{eq:cha}.\par

    \subsection{Outline} This paper is organized as follows. In Section \ref{sec:decomp} we recall some  results about perverse sheaves and intersection cohomology. Among other things we recall the Beilinson-Bernstein-Deligne-Graber decomposition theorem (Theorem \ref{thm:decomp}), their local invariant cycle map (Theorem \ref{thm:invmap}), and the Brosnan-Chow palindromicity theorem (Theorem \ref{thm:invmappalin}). Section \ref{sec:alggroups} is devoted to the theory of algebraic groups, including discussions about Grothendieck-Springer simultaneous resolution, torus actions, the Baker-Campbell-Hausdorff formula, and flag varieties.
    In section \ref{sec:chashv} we briefly review the work of Lusztig on character sheaves.
    In Section \ref{sec:hecke} we recall some  preliminaries about Coxeter groups and Hecke algebras, and give a combinatorial interpretation for the induced characters $\chi^{\ind_{W_J}^W}((1+T_{s_1})\ldots(1+T_{s_n}))$. Section \ref{sec:pavingbs} is devoted to describe an affine paving of the Bott-Samelson variety that closely resembles the combinatorics in Section \ref{sec:hecke}. In Section \ref{sec:twisted} we introduce the varieties $\h_w(X)$ and $\hbs_{\si}(X)$, and establish some of their basic properties. In section \ref{sec:hbspaving} we construct a paving by affines for the  varieties $\hbs_{\si}(X)$ and prove that their Poincaré polynomial agrees with the induced characters. Section \ref{sec:fibrations} is devoted to the study of the monodromy of $\h_w(X)$ and $\hbs_{\si}(X)$, including a proof that the Weyl group acts on the intersection cohomology of such varieties.  In Section \ref{sec:proofs} we use all the previous results to prove our main Theorems, \ref{thm:mainsimply} and \ref{thm:hbsall}. Finally, Section \ref{sec:questions} is devoted to further directions.

\section{Decomposition theorem}
\label{sec:decomp}
In this section we recall the basic definitions and results concerning intersection homology, perverse sheaves, and the Decomposition Theorem of Bernstein-Beilinson-Deligne-Gabber. We refer to \cite{KirwanWoolf}, \cite{BBD}, \cite{GoreskyMacphersonII} for more details.\par
We will restrict ourselves to algebraic varieties over $\mathbb{C}$. Let $\X$ be a scheme of finite type over $\mathbb{C}$ and denote by $\X(\mathbb{C})$ the topological space of its geometric points with the analytic topology. We denote by $\Shf(\X)$ the abelian category of sheaves of $\mathbb{C}$-vector spaces over $\X(\mathbb{C})$. An object $F$ in $\Shf(\X)$ will simply be called a sheaf over $\X$. As usual we will denote by $F_x$ the stalk of $F$ at the point $x\in \X(\mathbb{C})$ and the \emph{support} of $F$ is defined as $\Supp(F)=\{x\in \X(\mathbb{C}); F_x\neq\{0\}\}$. \par
 As usual, if $f\col \X\to \Y$ is a morphism we have left exact functors $f_*, f_!\col \Shf(\X)\to \Shf(\Y)$, called the \emph{direct image} and \emph{direct image with compact support},  and an exact functor $f^*\col \Shf(\Y)\to \Shf(\X)$, called the \emph{inverse image}. Also, if $f\col \Y\to \X$ is an immersion, we write $F|_\Y:=f^*F$ for $F\in \Shf(\X)$.\par
  If $L$ is a vector space, the constant sheaf $L_\X$ is the sheaf defined by $L_\X(\emptyset)=\{0\}$ and $L_\X(\U)=L$ for every non-empty connected open set $\U\subset \X(\mathbb{C})$. A sheaf $F$ is locally constant if there exists a covering $(\U_i)_{i\in I}$ of $\X(\mathbb{C})$ such that $F|_{\U_i}$ is isomorphic to a constant sheaf. If $\X$ is connected, we call a locally constant sheaf $F$ a local system. If $\X$ is connected, $x\in \X(\mathbb{C})$ and $L$ is a vector space, we have a bijection
  \begin{equation}
  \label{eq:localsystem}
  \{\text{Local systems $F$ with $F_x\cong L$}\}\leftrightarrow \{\text{homomorphisms $\pi_1(\X,x)\to GL(L)$}\}.
  \end{equation}
  In particular, if $\X$ is simply connected every local system is a constant sheaf.\par
    A sheaf $F$ on $\X$ is constructible if there exists a stratification $\X=\bigsqcup \X_{i}$ with each $\X_i$ a smooth locally closed subscheme such that $F|_{\X_i}$ is a local system on $\X_i$ for every $i$.\par
  As usual, we denote by $D^b_c(\X)$ the bounded derived category of sheaves of complex vector spaces on $\X$ with  constructible cohomology. Its objects $\mathcal{F}$ are bounded complex of sheaves of vector spaces, that is
  \[
  \mathcal{F}=\ldots 0\to 0\to F_{n}\to F_{n+1}\to\ldots\to  F_m\to 0\to 0\ldots
  \]
  for some integers $n<m$. If $F$ is a sheaf on $\X$, we view $F$ as the complex $0\to F\to 0$. As usual, we write $\mathcal{F}[n]$ for the shift of $\mathcal{F}$ to the left, $H^i(\mathcal{F})$ for its $i$-th cohomology sheaf, and $\mathcal{H}^i(F)$ for its $i$-th hypercohomology vector space. Moreover, if $L$ is a vector space, we write $\mathcal{F}\otimes L$ for the complex
  \[
  \ldots 0\to 0\to F_{n}\otimes L\to F_{n+1}\otimes L\to\ldots\to  F_m\otimes L\to 0\to 0\ldots.
  \]
  If $L=\bigoplus {L^n}$ is a graded vector space, we define
\[
\mathcal{F}\otimes L:=\bigoplus F\otimes L^n [-n].
\]
This is the same as $\mathcal{F}\otimes \mathcal{L}$, where 
\[
\mathcal{L}=\ldots \to L^{j}_\X\to L^{j+1}_\X\to \ldots
\]
is the complex of constant sheaves with every differential map being the zero map.  \par

   Let $f\col \X\to \Y$ be a morphism. Below we list the usual functors between the categories $D^b_c(\X)$ and $D^b_c(\Y)$. The first three are the derived functors of the usual functors between $\Shf(\X)$ and $\Shf(\Y)$.
   \begin{enumerate}
       \item The (derived) direct image $f_*\col D^b_c(\X)\to D^b_c(\Y)$.
       \item The (derived) direct image with proper supports $f_!\col D^b_c(\X)\to D^b_c(\Y)$. There exists a natural transformation $f_!\to f_*$ that is an isomorphism if $f$ is proper.
       \item The (derived) inverse image $f^*\col D^b_c(\Y)\to D^b_c(\X)$. 
       \item The extraordinary inverse image $f^!\col D^b_c(\Y)\to D^b_c(\X)$. The functor $f^!$ is the right adjoint of $f_!$.
       \item The Verdier duality functor $\mathcal{D}_\X\col D^b_c(\X)\to D^b_c(\X)$.
   \end{enumerate}
For simplicity we will write $f_*$ instead of the usual notation $Rf_*$ for the derived functor, and use $R^0f_*\col \Shf(\X)\to \Shf(\Y)$ for the direct image functor in the category of sheaves. Here, $R^jf_*(\mathcal{F}):=H^j(f_*\mathcal{F})$ for $\mathcal{F}\in D^b_c(\X)$, and similarly for $f_!$.

Below we list some properties of these functors:
\begin{enumerate}
    \item $H^j(\mathcal{F}[m])=H^{j+m}(\mathcal{F})$
    \item $R^jf_*(\mathcal{F}[m])=R^{j+m}f_*(\mathcal{F})$
    \item (\cite[Section 3.6]{KirwanWoolf}) If $f\col \X\to \spec(\mathbb{C})$, we have that $\mathcal{H}^j(\mathcal{F})=R^jf_*(\mathcal{F})$.
    \item (\cite[Page 92 item 13]{GoreskyMacphersonII})  If we have a fiber diagram
    \[
    \begin{tikzcd}
    \X_2\ar[d,"f^2"]\ar[r,"g_\X"] & \X_1\ar[d,"f^1"]\\
    \Y_2\ar[r,"g_\Y"] & \Y_1
    \end{tikzcd}
    \]
    then $g_\Y^*f^1_*\mathcal{F}\cong f^{2}_*g_\X^*\mathcal{F}$ for any $F\in D^b_c(\X)$.
\end{enumerate}

If $\X$ is smooth of complex dimension $n$, we have that its singular homology can be recovered as the hypercohomology of the trivial sheaf $\mathbb{C}_\X[n]$:
\[
H^i(\X,\mathbb{C})=\mathcal{H}^{i-n}(\mathbb{C}_\X[n]).
\]
Since usual singular homology does not behave well for singular varieties, we resort to perverse sheaves and intersection homology.\par

A complex $\mathcal{F}$ in $D^b_c(\X)$ is a perverse sheaf if
\begin{align*}
    \dim_{\mathbb{C}} \Supp(H^{-i}(\mathcal{F}))&\leq i\\
    \dim_{\mathbb{C}} \Supp(H^{-i}(\mathcal{D}_\X(\mathcal{F})))&\leq i.
\end{align*}
Here we follow Beilinson-Bernstein-Deligne-Gabber indexing scheme.  The category of perverse sheaves over $\X$ is a full subcategory of $D^b_c(\X)$ and, moreover, an abelian category. We note that if $\X$ is smooth of dimension $n$ then $\mathbb{C}_\X[n]$ is a perverse sheaf.
If $\U\subset \X$ is a smooth open dense subvariety of $\X$, then there exists a unique simple perverse sheaf $IC_\X$ such that $IC_\X|_\U=\mathbb{C}_\U[n]$. More generally, for any local system $L$ on any smooth open dense subvariety $\U\subset \X$, there exists a unique simple perverse sheaf $IC_\X(L)$ such that $IC_\X(L)|_\U=L[n]$. This is called the \emph{intermediate extension} of $L$ (see \cite[Definition 1.4.22 and Corollaries 1.4.24 and 1.4.25]{BBD}). Alternatively, if we write $j\col \U\to \X$ for the inclusion, we have that $IC_\X(L)$ is the image of the natural morphism
\[
j_!(L)\to j_*(L).
\]
 Moreover, we have that if $\U\subset \X$ is an open immersion then $IC_\X|_\U=IC_\U$.\par 
 If $\X$ is compact, the \emph{intersection cohomology of $\X$ with coefficients in $\mathbb{C}$}, denoted by $IH^i(\X)$, is the $(i-n)$-th hypercohomology $\mathcal{H}^{i-n}(IC_\X)$. Moreover, we let $IH^*(\X)=\bigoplus_j IH^j(\X)$.

  If $i\col \Y\hookrightarrow \X$ is a closed immersion, $\U\subset \Y$ an open dense subvariety, and $L$ a local system on $\Y$, we will abuse notation and write $IC_\Y(L)$ for $i_*(IC_\Y(L))$. We can now enunciate the decomposition theorem.

\begin{theorem}[BBDG]
\label{thm:decomp}
Let $\X$ be an smooth algebraic variety of dimension $n$ and $f\col \X\to \Y$ a proper map to a possibly singular algebraic variety $\Y$. Then there exists a stratification of $\Y=\bigcup \Y_i$, a local system $\mathcal{L}_{\Y_i,b}$ on each stratum $\Y_i$, and an integer $b$ such that there exists an isomorphism
\[
f_*(\mathbb{C}_\X[n])\cong\bigoplus_{\Y_i,b}IC_{\overline{\Y_i}}(\mathcal{L}_{S,b})[-b].
\]
\end{theorem}
\begin{Rem}
If the all the strata $\Y_i\subset \Y$ are simply connected, then the local systems are constant and we can write
\[
f_*(\mathbb{C}_\X[n])\cong\bigoplus_{\Y_i}IC_{\overline{\Y}_i}\otimes V_{\Y_i}.
\]
where $V_{\Y_i}$ is a graded vector space.
\end{Rem}
\begin{Rem}
The strata $\Y_i$ can be chosen as a stratification where the map $f$ is a topological fibration. See \cite[Theorem 2.1.1 (c)]{CatMig05}.
\end{Rem}
\begin{Rem}
\label{rem:locsys}
When $f\colon \X\to \mS$ is a smooth projective map of relative dimension $d$ with $\mS$ and $\X$ smooth, then 
\[
f_*(\mathbb{C}_\X)\simeq \bigoplus R^jf_*(\mathbb{C}_\X)[-j]
\]
and $R^jf_*(\mathbb{C}_\X)$ are semisimple local systems on $\mS$. This version of the decomposition theorem was proven by Deligne in \cite{DLef69} and \cite{Dhodge2}.\par

More generally if $f\col \X\to \mS$ is a proper topological fibration with $\mS$ smooth, we have that
\[
f_*(IC_\X)\cong\bigoplus R^jf_*(IC_\X)[-j]
\]
and $R^jf_*(IC_\X)$ are local systems on $\mS$. Indeed, since the claim is local on $\mS$, we can assume that $\X=\mS\times \Y$ and $f$ is the first projection. We consider the fiber diagram
\[
\begin{tikzcd}
\X=\mS\times \Y\ar[r,"p_2"]\ar[d,"f"]& \Y \ar[d, "g"]\\
\mS\ar[r,"h"] & \spec(\mathbb{C})
\end{tikzcd}
\]
By \cite[Theorem 5.4.2]{GoreskyMacphersonII} we have $IC_\X=p_2^*(IC_\Y)[-\dim \mS]$, so that
\begin{align*}
  f_*(IC_\X)&\cong f_*p_2^*(IC_\Y)[-\dim \mS]    \\
           &\cong h^*g_*(IC_\Y)[-\dim \mS]\\
           &\cong h^*(\bigoplus R^jg_*(IC_\Y)[-j-\dim \mS])\\
           &\cong \bigoplus h^*( R^jg_*(IC_\Y)[-j-\dim \mS])\\
           &\cong \bigoplus R^jf_*(p_2^*(IC_\Y)[-\dim \mS])[-j]\\
           &\cong \bigoplus R^jf_*(IC_\X)[-j]
\end{align*}
and from $R^jf_*(IC_\X)\cong h^*(R^jg_*(IC_\Y)[-\dim \mS])$ we have that $R^jf_*(IC_\X)$ is a constant sheaf. We used the fact that $f_*p_2^*=h^*g_*$, and that for each complex $\mathcal{F}$ over a point we have $\mathcal{F}\cong \oplus H^j(\mathcal{F})[-j]$.
\end{Rem}

\begin{proposition}
\label{prop:fibrations}Let  $f\col \X\to \mS$ and $g\col \Y\to \mS$ be topological fibrations over $\mS$, with $\mS$ and $\X$ smooth and let $h\col \X\to \Y$ be an $\mS$-morphism.
  \[
  \begin{tikzcd}
    \X \ar[r, "h"] \ar[rd, "f"] & \Y \ar[d, "g"]\\
     & \mS
  \end{tikzcd}
  \]
Moreover, assume that $h$ has a stratification $\Y_j$ such that $\overline{\Y_j}\to \mS$ is a fibration. Also, suppose that the local systems appearing in the decomposition theorem for $h$ are trivial, that is
  \[
  h_*(\mathbb{C}_\X)\cong \bigoplus IC_{\overline{\Y}_j}\otimes L_j,
  \]
  where $L_j$ are graded vector spaces. Then for a point $s\in \mS$, we have
  \[
  H_*(\X_s)\cong \bigoplus IH_*(\Y_{i,s})\otimes (L_i[-n+\dim \Y_i]).
  \]
 Moreover the action of $\pi_{\mS,s}$ on both sides is compatible with this isomorphism.
 \end{proposition}
  \begin{proof}
  From
  \[
  h_*(\mathbb{C}_\X)\cong \bigoplus IC_{\overline{\Y}_j}\otimes L_j,
  \]
  we see that
    \[
  f_*(\mathbb{C}_\X)\cong g_*(\bigoplus IC_{\Y_j}\otimes L_j).
  \]
  Taking cohomology, we have
  \begin{equation}
      \label{eq:RjC}
    \begin{aligned}
  R^jf_*(\mathbb{C}_\X)=H^j(f_*(\mathbb{C}_\X))&\cong H^j(g_*(\bigoplus_i IC_{\Y_i}\otimes L_i))\\
   &\cong \bigoplus_i H^j(g_*(IC_{\Y_i})\otimes L_i)\\
   &\cong \bigoplus_i \bigoplus_k H^j(g_*(IC_{\Y_i})\otimes L_i^k[-k])\\
   &\cong \bigoplus_i \bigoplus_k H^{j-k}(g_*(IC_{\Y_i}))\otimes L_i^k\\
   &\cong \bigoplus_i \bigoplus_k R^{j-k}g_*(IC_{\Y_i})\otimes L_i^k
  \end{aligned}
  \end{equation}
  Where both sides are local systems by Remark \ref{rem:locsys}. Taking the stalks at the point $s\in \mS$, we have
  \[
  H^{n+j}(\X_s)=(R^jf_*(\mathbb{C}_\X)_s)\cong \bigoplus_i \bigoplus_k (R^{j-k}g_*(IC_{\Y_i}))_s\otimes L_i^k \cong\bigoplus_i \bigoplus_k IH^{\dim \Y_i+j-k}(\Y_{i,s})\otimes L_i^k,
  \]
  which implies
  \[
  H^*(\X_s)\cong \bigoplus IH^*(\Y_{i,s})\otimes (L_i[-n+\dim \Y_i]).
  \]
  The action of $\pi_1(\mS,s)$ on both sides is compatible with the isomorphism because Equation \eqref{eq:RjC} is an isomorphism of local systems (which are identified with representations of $\pi_1(\mS,s)$).
  
  \end{proof}

 \begin{proposition}
 \label{prop:smooth_pullback_perverse}
 If $f\col \X\to \mS$ is a smooth map of relative dimension $d$ and $\F$ a simple perverse sheaf on $\mS$, then $f^*(\F)[d]$ is a simple perverse sheaf on $\X$. In particular, a simple perverse sheaf $\F$ on $\mS$ is a summand of a complex $\F'$ if and only if $f^*(\F)[d]$ is a summand of $f^*(\F')$.
 \end{proposition} 
 \begin{proof}
 This is \cite[4.2.5 and 4.2.6]{BBD}.
 \end{proof}
  
\subsection{Local invariant cycle map}
We finish this section on perverse sheaves with a few results about invariant cycles. 

    Let $f\col \X\to \Y$ be a smooth projective morphism between smooth varieties and $y\in \Y$. The global invariant cycle map (see \cite[Theorem 4.1.1]{Dhodge2}) says that we have a surjective map
    \[
    H^*(\X,\mathbb{C})\to H^*(\X_y,\mathbb{C})^{\pi_1(\Y,y)}.
    \]
    There is a local version, called the local invariant cycle map, which we now state, see [\cite[Corollary 6.2.9]{BBD}, \cite{CatMig09} and \cite[Theorem 54]{BrosnanChow}.
    \begin{theorem}
    \label{thm:invmap}
    Let $f\col \X\to \Y$ be a proper surjective map of algebraic varieties with $\X$ smooth. Let $\U\subset \Y$ be a Zariski open set where the restriction $f^{-1}(\U)\to \U$ is a smooth morphism. Let $y\in \Y\setminus \U$ and $B_y$ be a small Euclidean ball centered at $y$ and $y_0\in B_y\cap \U$. Then $H^*(\X_y,\mathbb{C})=H^*(f^{-1}(B_y),\mathbb{C})$ and there is a surjective map
    \[
    H^*(\X_y,\mathbb{C})\to H^*(\X_{y_0},\mathbb{C})^{\pi_1(B_y\cap \U,y_0)}.
    \]
    \end{theorem}
    We have the following characterization, due to Brosnan-Chow, of the cases where the local invariant cycle map is an isomorphism.
    
\begin{theorem}[{\cite[Theorems 57 and 102]{BrosnanChow}}]
\label{thm:invmappalin}
Let $f\col \X\to \Y$ be a projective morphism of reduced schemes of finite type over $\mathbb{C}$ with $\Y$ equidimensional. Let $y\in \Y$ be a closed point and define $d:=\dim \X-\dim \Y$.
Then the local invariant cycle maps are isomorphisms for all $i$ if and only if the fiber $\X_y$ have palindromic Betti numbers, that is, $\dim H^i(\X_y)=\dim H^{2d-i}(\X_y)$ for every $i=0,\ldots, n$.
\end{theorem}

We refer to \cite[Sections 5 and 6]{BrosnanChow} for a more detailed discussion about the local invariant cycle map.

\section{Algebraic groups and Lie algebras}
\label{sec:alggroups}
In this section we review some results about algebraic groups and Lie algebras necessary to the following sections. We refer the reader to \cite{MalleTesterman} and \cite{Hall} for more details.

Let $G$ be a connected, simply-connected, semisimple algebraic group. Let $T$, $B$, and $U$ be, respectively, a maximal torus, a Borel subgroup, and a unipotent subgroup of $G$ such that $B=TU=UT$. Also, we let $U^{-}$ be the opposed unipotent subgroup of $U$. As in the introduction we let $\flag=G/B$ be the flag variety of $G$.\par

We denote by $\g$, $\hl$, $\bl$, $\ul$, and $\ul^{-}$ the Lie algebras associated with $G$, $T$, $B$, $U$, and $U^{-}$. Any subalgebra $\hl'\subset \g$ that is the Lie algebra of a maximal torus $T'\subset G$ is called a \emph{Cartan subalgebra of $\g$}. Also, we write $\g_\hl$ for $\ul\oplus \ul^{-}$ (which depends only on $\hl$ and not on the choice of $\ul$).\par

Denote by $W=N_G(T)/T$ the Weyl group associated to $T$ and let $\Phi$ and $\Phi^+$ be the set of roots and  positive roots, respectively, associated with the pair $(T,B)$. Recall that there is a natural action of $W$ on $\Phi$. Let $\Delta\subset \Phi^+$ be the set of simple roots. We recall that there exists a natural bijection between $\Delta$ and the $S$ of simple reflections of $W$. For each simple root $\alpha\in \Delta$ we write $s_{\alpha}$ for the associated simple reflection in $S$ and, conversely, $\alpha_s$ for the simple root associated to $s\in S$. \par

   For each positive root $\gamma\in \Phi^+$ we define its height as $\hta(\gamma):=\sum_{\alpha\in \Delta} a_\alpha$ where $\gamma=\sum_{\alpha\in\Delta} a_{\alpha}\alpha$. We will, once and for all, choose representatives in $N_G(T)$ for each $w\in W$ and call it $\dw$. Moreover, for $w\in W$, we let $\Phi(w):=\{\gamma\in \Phi^+;w^{-1}(\gamma)\in \Phi^-\}$.\par
    
    For each $\gamma \in \Phi\setminus\{0\}$, we denote by $\ul_\gamma$ the associated root subspace in $\g$ and by $U_\gamma$ the unipotent subgroup of $G$ with Lie algebra $\ul_\gamma$.         For every ordering of the roots in $\Phi^{+}$, we  have that
      \begin{equation}
      \label{eq:Uprod}
      U=\prod_{\gamma\in \Phi^{+}} U_{\gamma}.
      \end{equation}
       If $w\in W$, we define $\ul^w:=\bigoplus_{\gamma\in \Phi(w)} \ul_\gamma$ and $\ul_w:=\bigoplus_{\gamma\in \Phi^+\setminus \Phi(w)} \ul_\gamma$ and let $U^w$ and $U_w$ be the associated unipotent subgroups. Alternatively, we have that $U^w=U\cap \dw U^{-}\dw^{-1}$ and $U_w=U\cap \dw U \dw^{-1}$.
    
       The algebraic group $G$ has a stratification $G=\bigsqcup_{w\in W} B\dw B$ in Bruhat cells $B\dw B=U^w\dw B$, and correspondingly, the flag variety $\flag=G/B$ is stratified by the Schubert strata $U^w\dw B/B$.\par
       
       Given a simple reflection $s$ of $W$ and its associated simple root $\alpha_s\in \Delta$, we denote by $P_s$ the parabolic subgroup associated to $s$, that is the subgroup generated by $B$ and $U_{-\alpha_s}$.  It is well known that 
       \begin{equation}
           \label{eq:PsB}
           P_s=B\sqcup B\dot{s}B\quad\text{ and }\quad P_s=\dot{s}B\sqcup \dot{s}B\dot{s}B.
       \end{equation}
      More generally, if $J\subset S$, we define the parabolic subgroup $P_J$ as the subgroup generated by $B$ and $U_{-\alpha_s}$ for $s\in J$. The parabolic subgroup also has a Bruhat decomposition given by $P_J=\bigsqcup_{w\in W_J}B\dw B$. Here, $W_J$ is the subgroup of $W$ generated by $J$. Moreover, we define $\Phi_J\subset \Phi$ as the subset of roots that are linear combinations of the roots in $J$. We also write $G_J$ for the subgroup of $G$ generated by $T, U_{\pm\alpha}$, and $s_\alpha\in J$.  \par

    A element $g\in G$ is \emph{regular} if its centralizer $Z_g(G)$ has the same rank as $T$. We denote by $G^r$ the locus of regular elements, which is an open subset of $G$. Recall that every element $g\in G$ can be written uniquely as $g=g_sg_u=g_ug_s$ where $g_s$ is semisimple and $g_u$ is unipotent, in which case $g$ is regular if and only if $g_u$ is regular as an element of $Z_{g_s}(G)$.  We write $G^{rs}$ for the set of regular semisimple elements of $G$, also an open subset of $G$. Analogously we write $B^r$ and $B^{rs}$ for the regular and regular semisimple locus of $B$, and $T^r$ for the regular locus of $T$.
    \begin{Exa}
     \label{exa:Uw}
      If $G=GL_n(\mathbb{C})$, we can take $T$, $B$, $U$, and $U^{-}$ to be the subgroups of diagonal matrices, upper triangular matrices, upper triangular matrices with all entries in its diagonal equal to $1$, and lower triangular matrices with all entries in its diagonal equal to $1$, respectively. In this case $W=S_n$, and for each $w\in S_n$, $\dw$ can be chosen as the permutation matrix given by $\dw e_i=e_{w(i)}$, where $e_i$ is the canonical basis of $\mathbb{C}^n$.\par
        We have that $\g=\text{Mat}_n(\mathbb{C})$, the space of $n\times n$ matrices with complex entries, and $\hl$, $\bl$, $\ul$, and $\ul^{-}$ are the subspaces of $\g$ of diagonal matrices, upper triangular matrices, upper triangular matrices with all entries in its diagonal equal to $0$, and lower triangular matrices with all entries in its diagonal equal to $0$. The Lie bracket of $\g$ is the usual commutator $[X,Y]=XY-YX$.\par
       
      The set of roots $\Phi$ lies inside the subspace $x_1+x_2+\ldots +x_n=0$ of $\mathbb{R}^n$ and its elements are $\gamma_{i,j}=e_i-e_j$, for $i,j=1,\ldots, n$. Moreover we have that $\Phi^{+}=\{\gamma_{i,j}\in \Phi, i<j\}$ and $\Delta=\{\gamma_{i,i+1}; i=1,\ldots, n\}$, where the identification between $\Delta$ and $S$ is given by $\gamma_{i,i+1}\to (i,i+1)$. Since $\gamma_{i,j}=\gamma_{i,i+1}+\gamma_{i+1,i+2}+\ldots+\gamma_{j-1,j}$, we have that $\hta(\gamma_{i,j})=j-i$.\par
      The unipotent subgroup $U_{\gamma_{i,j}}$ is the set of matrices of the form $I+\lambda E_{i,j}$, where $E_{i,j}$ is the matrix with $1$ in the entry $(i,j)$ and $0$ everywhere else, while the nilpotent subspace $\ul_{\gamma_{i,j}}$ is the set of matrices of the form $\lambda E_{i,j}$.\par

      For $w\in W$, we have that the action of $W$ on $\Phi$ is given by $w(\gamma_{i,j})=\gamma_{w(i),w(j)}$. This means that for each $\gamma_{i,j}\in\Phi^{+}$ we have that $w^{-1}(\gamma_{i,j})\in \Phi^{-}$ if and only if $w^{-1}(i)>w^{-1}(j)$, or equivalently, $\ell(w^{-1}(j,i))<\ell(w^{-1})$.\par
      If $n=4$ and  $w=4213$, then $\Phi(w)=\{\gamma_{1,2},\gamma_{1,4},\gamma_{2,4},\gamma_{3,4}\}$, hence
       \[
       U^w=\left\{\left(\begin{array}{cccc}
       1 & * &0 & *\\
       0 & 1 &0 & *\\
       0 & 0 & 1 & *\\
       0 & 0 & 0 &1
        \end{array}\right)\right\}
        \quad\text{ and }\quad 
        U_w=\left\{\left(\begin{array}{cccc}
       1 & 0 &* &0\\
       0 & 1 &* &0\\
       0 & 0 & 1 & 0\\
       0 & 0 & 0 &1
        \end{array}\right)\right\}.
       \]
      If $s=(2,3)$, then 
      \[
       P_s=\left\{\left(\begin{array}{cccc}
       1 & * &* & *\\
       0 & 1 &* &*\\
       0 & * & 1 & *\\
       0 & 0 & 0 &1
        \end{array}\right)\right\}.
      \]
    If we take $J=\{(1,2),(3,4)\}$, then 
    \[
    G_J=\left\{\left(\begin{array}{cccc}
       1 & * &0 & 0\\
       * & 1 &0 &0\\
       0 & 0 & 1 & *\\
       0 & 0 & * &1
        \end{array}\right)\right\}.
    \]
      
    An element $g \in GL_n(\mathbb{C})$ is semisimple if it is diagonalizable, and it is unipotent if $g-I$ is nilpotent. The element $g$ is regular if each distinct eigenvalue has a single Jordan block.\par
     The Bruhat decomposition of $GL_n(\mathbb{C})$ is equivalent to the fact that every invertible matrix can be written as a product $b_1\dw b_2$, where $b_1$ and $b_2$ are upper triangular and $\dw$ is a permutation matrix. Moreover, this permutation matrix is unique.
    \end{Exa}

 We state some lemmas that will be needed later on. 
 
 \begin{lemma}
\label{lem:UJnormal}
If $J\subset S$ with associated parabolic subgroup $P_J$, we have that 
\[
U^J:=\prod_{\gamma\in\Phi^+\setminus\Phi_J} U_\gamma
\]
is a normal subgroup of $P_J$. In particular, if $U_J=L_J\cap U$, then $U=U_J\cdot U^J$ and the projection $U\to U_J$ is a homomorphism of groups. Moreover, both $U_J$ and $U^J$ are fixed under conjugation by elements $t\in T$.
\end{lemma}
\begin{proof}
The statement that $U^J$ is a normal subgroup of $P_J$ is contained in \cite[Proposition 12.6]{MalleTesterman}. The fact that $U=U_J\cdot U^J$ follows from Equation \eqref{eq:Uprod}. Since $U^J$ is normal in $P_J$ and hence in $U$, the projection $U\to U_J$ is just the quotient $U\to U/U^J$. The last statement follows from the fact that each $U_\gamma$ is fixed under conjugation by elements of $T$.
\end{proof}
    
\begin{lemma}
\label{lem:bb'}
If $\bl$ and $\bl'$ are two Borel subalgebras of a Lie algebra $\g$, then there exists a Cartan subalgebra $\hl''$ inside $\bl\cap \bl'$.
\end{lemma}
\begin{proof}
This is equivalent to the fact that two Borel subgroups $B'$ and $B$ of $G$ contain a maximal torus, which is a consequence of the Bruhat decomposition, see \cite[Section 2.4]{BT}.
%https://math.berkeley.edu/~jawolf/publications.pdf/paper_133.pdf lema 22
\end{proof}

\begin{lemma}
\label{lem:bh'}
Let $\bl$ and $\hl'$ be a Borel subalgebra and Cartan subalgebra of $\g$, respectively. If $\g=\hl'\oplus\g_{\hl'}$, then the quotient map $\g_{\hl'}\to \frac{\g}{\bl}$ is surjective.
\end{lemma}
\begin{proof}
Let $\bl'$ be an associated Borel subalgebra of $\g$ containing $\hl'$, and write $\bl'=\hl'\oplus \ul'$. By Lemma \ref{lem:bb'} there exists a Cartan subalgebra $\hl''\subset \bl\cap \bl'$. Since the elements of $\ul'$ are nilpotent and the elements of $\hl''$ are semisimple, we have that $\hl''\cap \ul'=\{0\}$. In particular, $\bl'=\hl''\oplus\ul'$. This means that 
\[
\g=\bl'\oplus\ul'^{-}=\hl''+\ul'+\ul'^{-}\subset \bl+ \g_{\hl'}.
\]
This proves that $\g=\bl+\g_{\hl'}$ and hence the quotient map $\g_{\hl'}\to \frac{\g}{\bl}$ is surjective.
\end{proof}

Let us write $\ul_k:=\oplus_{\hta(\gamma)=k} \mathfrak{g}_\gamma$.  Since $\ul=\oplus_k\ul_k$, we define $p_k\col \ul\to \ul_k$ as the projection. Also, we choose generators $E_\gamma$ of $\g_\gamma$, so that every element of $\ul$ is of the form $\sum_{\gamma\in \Phi^+}x_\gamma E_\gamma$. Moreover, we define $\lambda_{\gamma,\beta}$ via
\[
[E_\gamma,E_\beta]=\lambda_{\gamma,\beta}E_{\gamma+\beta}
\]
if $\gamma+\beta\in \Phi^+$ and $\lambda_{\gamma,\beta}=0$ if $\gamma+\beta\notin \Phi$.
\begin{lemma}
\label{lem:fsurj}
Let $X=\sum_{\alpha\in \Delta} E_\alpha$. Then the map $f\col \ul_k\to \ul_{k+1}$ given by $f(u)=p_{k+1}([X,u])$ is surjective.
\end{lemma}
\begin{proof}
The map 
\begin{align*}
c_X\col\ul&\to \ul    \\
u&\mapsto [X,u]
\end{align*}
satisfies $c_X(\ul)\subset [\ul,\ul]$. Since $E_\alpha\neq 0$ for every $\alpha\in \Delta$, we have that $X$ is regular, and hence $\dim Z_{\g}(X)=|\Delta|$. Since $\ker (c_X)\subset Z_{\g}(X)$, we have that $\dim \ker(c_X)\leq |\Delta|$. Hence, we have 
\[
\dim(\ul)=\dim \ker(c_X)+\dim (c_X(\ul))\leq |\Delta|+\dim [\ul,\ul]=|\Delta|+\dim(\ul)-|\Delta|=\dim(\ul)
\]
which proves that $c_X(\ul)=[\ul,\ul]=\oplus_{k\geq2} \ul_k$. Since $c_X(\ul_k)\subset \ul_{k+1}$ for every $k$, the result follows. 
\end{proof}

In the corollary below we consider linear polynomial in the polynomial ring $\mathbb{C}[x_{\gamma}]_{\gamma\in \Phi^+}$.

\begin{corollary}
\label{cor:LI}
Let $\gamma_1,\gamma_2,\ldots,\gamma_{m_k}$ be all positive roots of height $k$ and let $\gamma_1',\ldots ,\gamma_{m_{k+1}}'$ be all roots of height $k+1$. Consider the following linear polynomials
\[
F_{j}:=\sum_{\alpha\in \Delta}\lambda_{\alpha, \gamma_j'-\alpha}x_{\gamma_j'-\alpha}
\]
for $j=1,\ldots, m_{k+1}$, where $x_{\gamma_j'-\alpha}=0$ if $\gamma_j'-\alpha\notin \Phi$. Then $(F_j)_{1\leq j\leq m_{k+1}}$ is linearly independent.
\end{corollary}
\begin{proof}
We have that 
\begin{align*}
f(\sum_{1\leq i\leq m_k} x_{\gamma_i}E_{\gamma_i})&=\sum_{\alpha\in\Delta}\sum_{1\leq i\leq m_k}\lambda_{\alpha,\gamma_i+\alpha}x_{\gamma_i}E_{\gamma_i+\alpha}\\
   &=\sum_{1\leq j\leq m_{k+1}} F_jE_{\gamma'_j}.
\end{align*}
By the surjectivity of $f$ proved in Lemma \ref{lem:fsurj}, we have that the $F_j$ must be linearly independent.
\end{proof}

\begin{Exa}
   If $G=GL_n(\mathbb{C})$ and $E_{i,j}$, $\gamma_{i,j}$ are as in Example \ref{exa:Uw}, we have that $\lambda_{\gamma_{i,j},\gamma_{j,l}}=1$ for $i<j<l$. Moreover, set $X=\sum_1^{n-1} E_{i,i+1}$ and $f$ as in Lemma \ref{lem:fsurj}. Then
   \[
   f(\sum_{1\leq i\leq n-k-1} x_iE_{i,i+k})=\sum_{1\leq i\leq n-k-2} (x_i-x_{i+1})E_{i,i+k+1}
   \]
   which means that $F_j=x_j-x_{j+1}$. Clearly, these polynomials are linearly independent.
\end{Exa}

   The Lie algebra $\g$ is canonically identified with the tangent space $\tang_eG$. There are also identifications $\tang_eB=\bl$, $\tang_eT=\hl$, $\tang_eU=\ul$, and $\tang_eU^{-}=\ul^{-}$.  Actually, for all $g_0\in G$, if $l_{g_0}\col G\to G$ is the left multiplication by $g_0$, we have a canonical identification $\tang_{g_0}G\cong \g$ given by $dl_{g_0}\col \tang_eG\to \tang_{g_0}G$. If we consider the map $c_{g_0}\col G\to G$ given by the conjugation $c_{g_0}(g)=g_0gg_0^{-1}$, the differential
    \[
    dc_{g_0}\col \tang_eG\to \tang_eG
    \]
    induces the adjoint action $\ad(g_0)\col \g\to \g$. Similarly, the multiplication map $G\times G\to G$ induces the sum map on the Lie algebras $\g\oplus \g\to \g$ and the diagonal map $G\to G\times G$ induces the diagonal map $\g\to \g\oplus \g$.

\begin{lemma}
\label{lem:TXB}
Let $g\in B\dw B\subset G$. Via the natural identification $\tang_gG=\g$, we have that $\bl\subset \tang_gB\dw B$.
\end{lemma}
\begin{proof}
Just note that $gB\subset B\dw B$. Hence, via the left multiplication $l_g$, we have that $l_g(B)\subset B\dw B$, which means that $dl_g(\bl)\subset \tang_g B\dw B$. Since $\tang_gG$ is identified with $\g$ via $dl_g$, we have $\bl\subset \tang_g B\dw B$.
\end{proof}

      The Lie algebra $\g$ also comes with an exponential map $\exp\col \g\to G$ which is not usually algebraic. However, on the nilpotent Lie algebra $\ul$, the restriction $\exp\col \ul\to U$ is an algebraic map, and moreover an isomorphism. The Baker-Campbell-Hausdorff (BCH) formula allows us to write the multiplication of $U$ in terms of the commutator of $\mathfrak{u}$, that is, if we give $\mathfrak{u}$ the operation
\begin{equation}
    \label{eq:BCH}
X\otimes Y=X+Y+\frac{1}{2}[X,Y]+\frac{1}{12}([X,[X,Y]]+[Y,[Y,X]])+\ldots,
\end{equation}
then $(\mathfrak{u},\otimes)$ is a group and $\exp$ a group isomorphism. We note that since $\ul$ is nilpotent the sum above is actually finite. For the explicit definition of the BCH formula we refer to \cite[Chapter 5]{Hall}, but we will only need the terms appearing in Equation \eqref{eq:BCH}.

\begin{Rem}
\label{exa:bch}
In the case that $\ul$ is the subspace of $\text{Mat}_n(\mathbb{C})$ consisting of upper triangular matrices with all entries in the diagonal equal to 0, then for $X,Y\in \ul$ we have that $X\otimes Y$ is given by 
\[
\exp(X\otimes Y)=\exp(X)\exp(Y),
\]
where $\exp$ is the usual matrix exponentiation.
\end{Rem}

\subsection{Grothendieck-Springer simultaneous resolution  }
Here, we recall some facts about the quotient of $G$ by its adjoint action. In this section we follow  \cite[Chapter 4]{Slodowy}, \cite{Stein74} and \cite{Popov}.\par
It is a well known result of Steinberg that the adjoint action of $G$ on itself has a quotient that is isomorphic to $T/W$. Each element semisimple element $g_s\in G$ is conjugate to some element $t\in T$, and there is a map $\eta\col G\to T/W$ which takes an element $g=g_sg_u\in G$ to the class of $t\in T$, where $t$ is conjugate to $g_s$.\par

\begin{Exa}
 If $G=GL_n(\mathbb{C})$, then the map $\eta\col G\to T/S_n$ takes an invertible matrix $g$ to the diagonal matrix of its eigenvalues (even if $g$ is not diagonalizable). Of course, this diagonal matrix is only defined up to permutation of its entries.  One can think of $\eta$ as the map that takes each matrix to its characteristic polynomial.
\end{Exa}

Recall that $B=TU=UT$ and that for every $t\in T$, we have that $tUt^{-1}=U$. Let $p\col B\to T$ be the quotient of the right action of $U$ on $B$. Consider the $B$-action on $B$ given by conjugation and the trivial $B$-action on $T$. 

\begin{lemma}
\label{lem:pmap}
The map $p\col B\to T$ is $B$-equivariant, where $B$ acts on $B$ by conjugation and trivially on $T$. Moreover, we have that $b\in B$ is regular semisimple if and only if $p(b)$ is regular.
\end{lemma}
\begin{proof}
Indeed, we have to prove that given any $t\in T$, $u\in U$, and $b\in B$, there exists $u'\in U$ such that
\[
btub^{-1}=tu'.
\]
Writing $b=u_0t_0$, we have 
\begin{align*}
btub^{-1}&=u_0tt_0ut_0^{-1}u_0^{-1}\\
         &=u_0tu_1 && \text{ for }u_1=t_0ut_0^{-1}u_0^{-1}\in U,\\
         &=tu_2u_1 && \text{ for }u_2=t^{-1}u_0t\in U,\\
         &=tu'.
\end{align*}
In particular $B$ acts trivially on the quotient $B/U$. The second statement is well known, see \cite[Corollary 2.13]{Stein65}.
\end{proof}

\begin{Exa}
 If $G=GL_n(\mathbb{C})$, then $B$ and $T$ are the subgroups of upper triangular matrices and diagonal matrices. The map $p\col B\to T$ takes an upper triangular matrix $b$ to the diagonal matrix with same diagonal as $b$. The diagonal of an upper triangular matrix is invariant by conjugation of upper triangular matrices, and moreover an upper triangular matrix is regular semisimple if its eigenvalues are all distinct, which is the same as saying that the diagonal matrix has distinct eigenvalues.
\end{Exa}

Since $B$ acts on $G$ by right multiplication and also on itself by conjugation, we write $G\times^BB$ for the contracted product of $G$ and $B$, i.e., the quotient variety $(G\times B)/ \sim$, where $(g,b)\sim (gb', b'^{-1}bb')$ for every $b'\in B$.\par

We then have morphisms
\begin{equation}
\begin{tikzcd}[row sep=0cm]
G\times^B B\ar[r] &  G\times^BT   \ar[r,"\sim"] & \flag\times T\\
(g,b) \ar[r,maps to] & (g,p(b)) &\\
                   & (g,t) \ar[r,maps to] & (gB,t)\\
\end{tikzcd}
\end{equation}

Composing with the projection on the second factor, we have a smooth morphism (see \cite[Page 47]{Slodowy})
\[
\theta\col G\times^BB\to T.
\]

Consider the map $\phi \col G\times^BB\to G$ given by $\phi(g,b)=gbg^{-1}$. It is well defined as $\phi(gb',b'^{-1}bb')=gbg^{-1}$. We have a commutative diagram (see \cite[Page 50]{Slodowy}), that is called the Grothendieck-Springer simultaneous resolution: 
\[
\begin{tikzcd}
G\times^B B \ar[r,"\phi"] \ar[d,"\theta"] & G\ar[d,"\chi"]\\
T\ar[r,"\psi"] & T/W.
\end{tikzcd}
\]
There are other characterizations of $G\times^BB$. Consider the variety $\h_{e}\subset G\times \flag$ given by 
  \begin{equation}
      \label{eq:h1}
      \h_{e}=\{(X,gB); g^{-1}Xg\in B\}.
  \end{equation}
  This is the relative variety $\h_e$ associated to the identity element $e\in W$ (see Definition \ref{def:twisted}).  Denote by $\h_{e}^{rs}:= (G^{rs}\times \flag) \cap \h_{e}$.
 \begin{Exa}
 \label{exa:h1GLn}
 If $G=GL(\mathbb{C},n)$, then $\h_e^{rs}=\{(X,V_\bullet); XV_\bullet=V_\bullet\}$. This can be though as the set of regular semisimple matrices together with an ordering of its eigenvectors.
 \end{Exa}
 \begin{proposition}
 \label{prop:iso}
 We have the following isomorphisms
 \begin{enumerate}
     \item\label{item:2} $G\times^BB\to \h_{e}$,
     \item\label{item:3} $G\times^B B^{rs}\to \h_{e}^{rs}$,
     \item\label{item:4} $G/T\times T^{r}\to \h_{e}^{rs}$.
\end{enumerate}     
 \end{proposition}

 \begin{proof}
The isomorphism in item \eqref{item:2} is given by the embedding
\begin{align*}
    G\times^B B&\to G\times \flag\\
    (g,b)&\mapsto (gbg^{-1}, gB).
\end{align*}
It is easy to see that the image of $(gb',b'^{-1}bb')$ is the same as that of $(g,b)$, so this map is well defined, and its image is precisely $\h_e$, with inverse map
\begin{align*}
    \h_e&\to G\times^BB\\
    (X,gB) &\mapsto (g, g^{-1}Xg).
\end{align*}
The isomorphism in item \eqref{item:3} is just a restriction of the isomorphism in item \eqref{item:2}. To prove item \eqref{item:4} we will show that the map
\begin{align*}
G/T\times T^r &\to \h_e^{rs}\\
(gT,t) &\mapsto (gtg^{-1},gB)
\end{align*}
is an isomorphism. Actually, we will prove that the map
\begin{align*}
G/T\times T^r &\to G\times^BB^{rs}\\
(gT,t) &\mapsto (g,t)
\end{align*}
is an isomorphism, which is equivalent to the statement before via item \eqref{item:3}. Since $G\times^BB^{rs}$ is smooth, it suffices to prove that the map is a bijection. We begin with injectivity: If $(g,t)\sim (g',t')$ for $g,g'\in G$ and $t,t'\in T^r$, then there exists $b'\in B$ such that
\begin{align*}
    g&=g'b',\\
    t&=b'^{-1}t'b'.
\end{align*}
By Lemma \ref{lem:pmap}, we have that $t=p(t)=p(b'^{-1}t'b')=p(t')=t'$. This means that $b'\in Z_G(t)=T$ (recall that $G$ is simply connected), and hence $gT=g'b'T=g'T$. This proves injectivity.
 To prove the surjectivity, note that every $b\in B^{rs}$ is conjugated to $p(b)\in T^{r}$. Hence, every point $(g,b)\in G\times^BB^{rs}$ has a representative of the form $(g',t)$ with $t\in T^{r}$, so in particular $(g,b)$ is the image of $(g'T,t)$.
\end{proof}

If we restrict the Grothendieck-Springer simultaneous resolution to regular semisimple elements, that is, if we consider the commutative diagram
\[
\begin{tikzcd}
G\times^B B^{rs} \ar[r,"\phi"] \ar[d,"\theta"] & G^{rs}\ar[d,"\chi"]\\
T^r\ar[r,"\psi"] & T^r/W,
\end{tikzcd}
\]
then we have a fiber diagram (see \cite[Lemma 5.5]{Popov}). In particular, the map $H_{e}^{rs}\to G^{rs}$ (which is the same map as $G\times^BB^{rs}\to G^{rs}$) is a $W$-Galois cover. This means that we have a surjective morphism $\pi_1(G^{rs},g)\to W$. 
  
  We now define the local monodromy groups (following the definitions and notation in \cite[Section 7.1]{BrosnanChow}): For $X\in G^r$, let $B_X$ be a small Euclidean ball centered at $X$ and let $g\in G^{rs}\cap B_x$. Choosing a path $\xi$ from $X$ to $g$, we have a homomorphism
  \begin{align*}
  \pi_1(G^{rs}\cap B_x,X)&\to \pi_1(G^{rs},g)\\
                    \tau &\mapsto \xi\cdot \tau \cdot\xi^{-1}.
  \end{align*}
  Composing with the map $\pi_1(G^{rs},g)\to W$, we have a map $\pi_1(G^{rs}\cap B_X,X)\to W$. The image of this last map is the \emph{local monodromy group} at $X$, which we denote by $LM(X)$. Note that this is defined only up to conjugation.\par
  For an element $X=X_sX_u\in G^r$, $X_s$ is conjugate to an element $X_s'\in T$, and we denote by $W_{X_s}$ the stabilizer of $X'_s$ in the action of $W$ on $T$. Again, $W_{X_s}$ is only defined up to conjugation.
\begin{proposition}
\label{prop:conjugate}
 For $X\in G^r$, $LM(X)$ and $W_{X_s}$ are conjugate.
\end{proposition}
\begin{proof}
The proof is the same as in \cite[Proposition 113]{BrosnanChow}. We consider the Grothendieck-Springer simultaneous resolution (restricted to regular elements of $G$)
\[
\begin{tikzcd}
G\times^B B^r \ar[r,"\phi"] \ar[d,"\theta"] & G^r\ar[d,"\chi"]\\
T\ar[r,"\psi"] & T/W.
\end{tikzcd}
\]
and the Steinberg section $\epsilon\col T/W\to G^{r}$ of $\chi$  (see \cite{Stein74}). Up to conjugation, we can assume that $\epsilon(\chi(X_s'))=X$. The result follows from \cite[Propositions 50 and 106]{BrosnanChow}.
\end{proof}

\subsection{Relative torus action}
\label{sec:torusbundle}
Consider the torus fiber bundle $\T\to G^{rs}$  (see \cite[Section 8.7]{BrosnanChow} and \cite[]{balibanu}) where
\[
\T=\{(g,g')\in G^{rs}\times G^{rs}; g'\in Z_G(g)\}
\]
and projection is onto the first factor. This is indeed a torus fiber bundle as the centralizer of every regular semisimple element is a torus (recall that $G$ is simply connected, otherwise we would have to consider the connected component of $Z_G(g)$ containing the identity). In fact, $Z_G(g)$ is the only maximal torus of $G$ containing $g$.
\begin{proposition}
We have that $H_{e}^{rs}\times_{G^{rs}}\T\to H_{e}^{rs}$ is a trivial torus bundle.
\end{proposition}
\begin{proof}
By Proposition \ref{prop:iso}, we identify $H_{e}^{rs}$ with $G/T\times T^{r}$, and the natural map $H_e^{rs}\to G^{rs}$ is given by
\begin{align*}
G/T\times T^r&\to G^{rs}\\
   (gT,t)&\mapsto gtg^{-1}.
\end{align*}
 Then we have the equality
\[
(G/T\times T^r)\times_{G^{rs}}\T=\{(gT,t,g')| g'\in Z_G(gtg^{-1})\}.
\]
However, $g'\in Z_{G}(gtg^{-1})$ if and only if $g^{-1}g'g\in Z_{G}(t)=T$, hence we have an isomorphism
\begin{align*}
(G/T\times T)\times_{G^{rs}}\T &\to G/T\times T^{r}\times T\\
((gT, t),(gtg^{-1},g'))&\mapsto (gT,t, g^{-1}g'g).
\end{align*}
This proves that $H_1^{rs}\times_{G^{rs}}\T\cong H_1^{rs}\times T$.
\end{proof}
\begin{Exa}
 When $G=GL(\mathbb{C},n)$ and $X\in G^{rs}$, we have that $Z_{G}(X)$ is the set of matrices acting diagonally on the eigenvectors of $X$. We have seen in example \ref{exa:h1GLn}  that $\h_e^{rs}$ is the set of matrices together with an ordering with the eigenvectors. If $(X,(v_1,\ldots, v_n))$ is a point in $\h_e^{rs}$ then  we can identify the fiber $(\h_e^{rs}\times_{G^{rs}}\T)_{(X,(v_1,\ldots, v_n))}=Z_{G}(X)$ with $(\mathbb{C}^*)^n$ where the $i$-th coordinate $t_i$ of $(\mathbb{C}^*)^n$ acts on the $i$-th eigenvector of $X$.
\end{Exa}
\begin{proposition}
\label{prop:pi1toW}
Let $y\in T^{r}$ (in Particular $\T_y=T$). We have an action of $\pi_1(G^{rs},y)$ on $T$, and hence on $\hl$ (the Lie algebra of $T$). This action factors through the morphism $\pi_1(G^{rs},y)\to W$, where $W$ acts on $T$, and hence on $\hl$, via the adjoint action.
\end{proposition}
\begin{proof}
Let $\gamma\in \pi_1(G^{rs},y)$, and let $w\in W$ be the image of $\gamma$ via $\pi_1(G^{rs},y)\to W$. Let $y_1=(T,y)$ and $y_2=(\dw^{-1} T,\dw y\dw^{-1})$ be points in $G/T\times T^r$ over $y$. Note that $\gamma$ lifts to $G/T\times T^r$ connecting $y_1$ and $y_2$. Via the isomorphism below
\begin{align*}
\T_H:=(G/T\times T)\times_{G^{rs}}\T &\to G/T\times T^{r}\times T\\
((gT, t),(gtg^{-1}, g'))&\mapsto (gT,t, g^{-1}g'g),
\end{align*}
we can make the following identifications 
\begin{align*}
T=\T_y= \T_{H,y_1}&\to T\\
g'&\mapsto g'.
\end{align*}
On the other hand, we also have
\begin{align*}
T=\T_y= \T_{H,y_2}&\to T\\
g'&\mapsto \dw g'\dw^{-1}
\end{align*}    
and therefore an induced automorphism 
\begin{align*}
    \T_y&\to \T_y\\
    g'&\mapsto \dw g'\dw^{-1}.
\end{align*}
Since the path $\gamma$ lifts to a path on $G/T\times T^{r}$ with endpoints $y_1$ and $y_2$, and since $\T_H$ is a trivial bundle, we have our result. Taking the differential of the isomorphism above we get the action of $\pi_1(G^{rs},y)$ on $\hl$.
\end{proof}
\begin{corollary}
\label{cor:pi1T}
The action of $\pi_1(G^{rs},y)$ on $T$ and $\hl$ factor through the usual action of $W$ on $T$ and $\hl$.
\end{corollary}

\subsection{Flag varieties}
\label{subsec:flag}
In this section we recall some basic results about flag varieties, Schubert varieties, and Bott-Samelson resolutions. Unlike the introduction, here we will work with any Lie type.\par 

As in the previous section let $G$ be a connected, simply connected, semisimple algebraic group. Denote by $\flag:=G/B$ the flag variety associated to $G$, so that $G$ acts on $\flag$ by left multiplication. Recall the definition of the relative Schubert strata $\Omega_w^\circ\subset \flag\times \flag$  
\[
\Omega_w^\circ:= G \cdot (B/B, \dw B/B)
\]
and let $\Omega_w=\overline{\Omega_w^\circ}$ be the relative Schubert variety. For an element $g\in G$, we write $\Omega_{w,g}^\circ$ and $\Omega_{w,g}$ for the fibers of $\Omega_{w}^\circ$ and $\Omega_{w}$ over $gB/B$ via the first projection (of course $\Omega_{w,g}^\circ$ and $\Omega_{w,g}$ only depends on $gB/B$). In particular, if $g=e$, we have that $\Omega_{w,e}^\circ=B\dw B/B$, the usual Bruhat strata of $\flag$.

\par
   Given a simple transposition $s\in S$ and $\mP_s:=G/P_s$ the associated partial flag variety, there is a natural map $\flag\to \mP_s$.  For a word $\si=(s_1, \dots, s_n)$, we define the Bott-Samelson variety
   \[
   \bs_\si:=\flag\times_{\mP_{s_1}} \flag\times_{\mP_{s_2}}\ldots\times_{\mP_{s_n}}\flag,
   \]
   and there is a natural map $\alpha_\si\col \bs_\si\to \flag\times\flag$ given by the first and last projections.\par
   
   \begin{Exa}
   \label{exa:gVbullet}
      When $G=GL_n(\mathbb{C})$, an element $gB\in \flag=G/B$ is identified with a flag $V_\bullet$ given by $V_i=\langle v_1,\ldots, v_i\rangle$, where the $v_i$ are columns of $g=(v_1\; v_2\;\ldots\; v_n)$. It is clear that $V_\bullet$ does not depend of the representative $g\in gB$, and the Schubert varieties $\Omega_w$ are the same as those considered in the introduction.\par
       For $s=(i,i+1)\in S$, we have that $\flag\times_{\mP_s}\flag=\{(g_1B,g_2B); g_1P_s=g_2P_s\}$. The condition $g_1P_s=g_2P_s$ is equivalent (recall Equation \eqref{eq:PsB}) to
       \[
       g_1\in g_2B\quad\text{ or }\quad g_1\in g_2B\ds B.
       \]
       Since $\ds$ is the permutation matrix associated to $(i,i+1)$, we have that $V_\bullet^1$ and $V_\bullet^2$ must satisfy $V_j^1=V_j^2$ for every $j\neq i$. Therefore the Bott-Samelson varieties are also the same as those considered in the introduction.
   \end{Exa}

   If $\si$ is a reduced word for $w$, then the image of $\alpha_\si$ is precisely $\Omega_w$, and we will still denote by $\alpha_\si$ the induced morphism $\alpha_\si\col \bs_\si\to \Omega_w$. It is known (see \cite{Springer}) that $\alpha_\si$ has a Whitney stratification given by $\Omega_u^\circ$ for $u\leq w$. Since the varieties $\Omega_u^\circ$ are simply connected, the Decomposition theorem reads (see \cite[Theorem 2.8]{Springer})
   \begin{equation}
       \label{eq:Schubertdecomp}
   \alpha_{\si*}(\mathbb{C}_{\bs_\si}[\dim \bs_\si])=\bigoplus IC_{\Omega_u}\otimes V_u
\end{equation}
   for some graded vector spaces $V_u$. Let $P_{\si,u}(q^{\frac{1}{2}})=\sum q^{\frac{i}{2}} \dim V_u^i$ be the Hilbert-Poincaré series of $V_u$. \par
   As seen in Section \ref{subsec:heckeflag}, we have a direct relation between $\Omega_w$ and $C'_w$. We make explicit in the following proposition the results that will be needed in the sequel.
   \begin{proposition}[\cite{Springer}]
   \label{prop:springer}
    If $\si=(s_1,\ldots, s_{\ell(w)})$ is a reduced word for $w$, then 
    \begin{enumerate}
        \item $P_{\si,u}(q^{-\frac{1}{2}})=P_{\si,u}(q^{\frac{1}{2}})$,
        \item $P_{\si,w}=1$
        \item $q^{-\frac{\ell(w)}{2}}\prod (1+T_{s_i})=\sum_{u\leq w} P_{\si,u}(q^{\frac{1}{2}}) C'_u$.
    \end{enumerate}
   \end{proposition}

\section{Character Sheaves}
\label{sec:chashv}

   We have seen in Section \ref{sec:alggroups} that the Bruhat stratum $B\dw B$ is equal to $U_w\dw B$, and moreover every element $g\in BwB$ can be written uniquely as $u\dw b$, with $u\in U_w$ and $b\in B$. Since $B=TU$, we have a projection $p\col B\to T$, which induces a projection $p_w\col BwB\to T$ given by $p_w(u\dw b)=p(b)$. If $L$ is a local system on $T$, we can pull it back via $p_w$ to obtain a local system on $BwB$.  \par
   
We can identify $\h_w$ with $G\times^B B\dw B$ (where $B$ acts on $BwB$ by conjugation) via the isomorphism
\begin{align*}
    \h_w^\circ &\to G\times^B B\dw B\\
    (X,gB)& \to (g, g^{-1}Xg).
\end{align*}
 Since $p_w^*(L)$ is $B$-invariant (see \cite[Lemma 4.1.2]{MarsSpringer}), we have that $p_w^*(L)$ will descend to a local system $L_w$ on $\h_w$. \par
 
   Lusztig proved that $f_!(L_w)$ is a semisimple complex that is a sum of shifted simple perverse sheaves. An simple perverse sheaf that is a summand of $f_!(L_w)$ for some $w$ and some $L$ is called a \emph{character sheaf}.\par
   The work in \cite{ChaShvI,ChaShvII,ChaShvIV,ChaShvV} is a thorough study of these character sheaves, which includes their classification.
   
   As an example, when $w$ is the identity and $L$ is the trivial sheaf, then the summands in $f_!(L_w)$ correspond to the irreducible representations of $W$.  In fact, by Proposition \ref{prop:pi1toW} we have a map $\pi_1(G^{rs},y)\to W$, and each irreducible representation $\rho$ of $W$ will induce an simple local system $L_\rho$ on $G^{rs}$ (recall the bijection in Equation \eqref{eq:localsystem}), the intermediate extension $IC_G(L_\rho)$ of which will be an simple character sheaf. Moreover, we have that $f_!(L_w)$ will be induced by the regular representation of $W$. In the case that $G=GL_n$, these are the only character sheaves induced by the trivial sheaf on $T$.\par
   
   Lusztig also noted that we could change $f_!(L_w)$ to $f_*(IC_{\h_w}(L_w))$, that is, the latter sheaf is a semisimple complex which is a sum of shifted character sheaves, and each character sheaf appears as a summand of $f_*(IC_{\h_w}(L_w))$.

  The main idea behind the study of character sheaves is Grothendieck's sheaf-function correspondence. Here, we assume that we are over a field $k$ of positive characteristic $p$. If $F$ is a bounded complex of sheaves of $\overline{\mathbb{Q}_\ell}$ vector spaces over a scheme $\mathcal{X}$ such that there is an isomorphism $\varphi\col \Fr^*(F)\to F$. The characteristic function of $F$ is defined as 
   \begin{align*}
       \chi_{\mathcal{F}}\col \mathcal{X}(\mathbb{F}_q)&\to \overline{\mathbb{Q}_{\ell}}\\
                                  x &\mapsto \sum_{i}(-1)^i\text{Tr}(\varphi_x, H^i(F)_x).
   \end{align*}
   Grothendieck's correspondence is the idea that an interesting function on $\mathcal{X}$ should be a characteristic function of some complex of sheaves. When $\mathcal{X}=G$, the irreducible characters of $G(\mathbb{F}_q)$ are interesting functions, and should therefore come from complex of sheaves on $G$.\par  
    This is precisely what Lusztig proved when $G=GL_n(k)$: Every irreducible character of $GL_n(\mathbb{F}_q)$ is the characteristic function of a character sheaf. In general, we have that the characteristic functions of character sheaves form an orthonormal basis for the class functions of $G(\mathbb{F}_q)$.

   Here we are mostly concerned with the case that $L$ is trivial, and then $L_w$ is trivial as well. So we will restrict ourselves to this case.

   For each semisimple complex $F=\sum F_i[n_i]$ on $G$, with each $F_i$ a perverse sheaf, define
      \[
      \chi(F):=\sum q^{\frac{-n_i}{2}}F_i
      \]
      which is an element of $\mathbb{Z}[q^{\frac{1}{2}},q^{-\frac{1}{2}}]\otimes KG$ where $KG$ is the Grothendieck group of perverse sheaves on $G$. Lusztig considered a map
      \begin{align*}
      \tau\col H_W &\to \mathbb{Z}[q^{\frac{1}{2}},q^{-\frac{1}{2}}]\otimes KG\\
               C'_w&\mapsto \chi(f_*(IC_{\h_w}))
      \end{align*}
     and proved that $\tau$ is a generalized trace function. Moreover, he also concludes that there exists a function $c$, depending on a character sheaf $A$ and an irreducible character $\chi$ of $W$ such that, for every $a \in H_W$,
     \[
     \tau(a)=\sum c(A,\chi)\chi(a) A,
     \]
     where the sum runs through all irreducible characters $\chi$ of $W$ and all character sheaves $A$ of $G$. \par
     In the case that $G=GL_n$, we have that the character sheaves are in bijection with the irreducible character of $S_n$ and $c(A,\chi)=1$ if $A$ is induced by $\chi$ and $0$ otherwise. In particular we have that
     \[
     f_*(IC_{\h_w})=\sum_{\lambda\vdash n}\chi^{\lambda}(C'_w)A_{\lambda}
     \]
     Considering the stalk at a regular semisimple point $X$ and the appropriate shift,
     \[
     IH^*(\h_w(X))=\sum_{\lambda\vdash n}\chi^{\lambda}(q^{\frac{\ell(w)}{2}}C'_w) V^{\lambda},
     \]
     in particular, taking the Frobenius character, we have
     \begin{align*}
     \ch(IH^*(\h_w(X)))&=\sum_{\lambda\vdash n}\chi^{\lambda}(C'_w) \ch(V^{\lambda})\\
     &=\sum_{\lambda\vdash n}\chi^{\lambda}(C'_w) s_{\lambda}\\
     &=\ch(C'_w).
     \end{align*}

 \section{Induced characters of Hecke algebras}
\label{sec:hecke}
   Let $W$ be a finite Coxeter group with set of simple reflections $S$ and set of transpositions $\Tr$.   Each element $w\in W$ can be written (possibly non-uniquely) as a reduced word $w=s_1\ldots s_n$, and we define $\ell(w):=n$. The Bruhat order on $W$ is given by $w'\leq w$ if some reduced word of $w'$ is a subword of a reduced word of $w$. We also define the right and left weak orders as
   \begin{align*}
       w'\leq_Rw&\text{ if }w=w'u\text{ for some $u\in W$ with }\ell(w)=\ell(w')+\ell(u)\\ 
       w'\leq_Lw&\text{ if }w=uw'\text{ for some $u\in W$ with }\ell(w)=\ell(w')+\ell(u).
   \end{align*}
   We note that if $s\in S$ and $w\in W$, then either $sw<w$ or $sw>w$, and either $sw<_Lw$ or $sw>_Lw$. Moreover $sw<w$ if and only if $sw<_Lw$. The same holds for $\leq_R$.
   
       Define also the \emph{left and right inversion set} of $w$ by 
       \begin{align*}
           \Tr_L(w)&:=\{t \in \Tr; \ell(tw)<\ell(w)\},\\
           \Tr_R(w)&:=\{t \in \Tr; \ell(wt)<\ell(w)\}.
       \end{align*}
      Since we will mostly consider the set of left inversions, we write $\Tr(w):=\Tr_L(w)$. Similarly, the set of \emph{left and right descents} of $w$ by:
       \begin{align*}
           D_L(w)&:=\Tr_L(w)\cap S,\\
           D_R(w)&:=\Tr_R(w)\cap S.
       \end{align*}
       We also recall a simple result on Coxeter groups:
       \begin{lemma}
\label{lem:s'ws}
Let $w\in W$ and $s,s'\in S$ such that $s'w>w$ and $ws>w$. Then either 
\begin{enumerate}
    \item $s'ws=w$, or
    \item $s'ws>ws$ and $s'ws>s'w$.
\end{enumerate}
\end{lemma}
\begin{proof}
See \cite[Corollary 2.2.8]{BjornerBrenti}.
\end{proof}

   Given $J\subset S$ we denote by $W_J$ the induced parabolic subgroup, that is, the subgroup generated by $J$. We also define ${}^JW$ as the set of minimal representatives of the right cosets $W_J\backslash W$, that is, for each coset $W_jw$ we choose its element with minimal length.
   \begin{Exa}
   \label{exa:S4}
   Let $W=S_4$, then $S=\{(1,2),(2,3),(3,4)\}$. If $J=\{(1,2),(3,4)\}$, then ${}^JW$ is the set of permutations $w$ such that $w^{-1}(1)<w^{-1}(2)$ and $w^{-1}(3)<w^{-1}(4)$.
   \[
   {}^JW=\{1234, 1324, 1342, 3124, 3142, 3412\}.
   \]
   Analogously $W^J$ would be the set of permutations $W$ such that $w(1)<w(2)$ and $w(3)<w(4)$.   
   \end{Exa}
   
   By Deodhar's lemma \cite[]{BjornerBrenti}, we have that if $s\in S$ and $w\in {}^JW$ then either $ws\in {}^JW$  or $W_Jw=W_Jws$. We define the right action  $\cdot_J$ of $W$ on ${}^JW$ given by 
   \[
   w\cdot_J s=\begin{cases}
   ws& \text{ if } ws\in {}^JW\\
   w& \text{ otherwise}
   \end{cases}
   \]
   for every $w\in {}^JW$ and for every $s\in S$.  This is equivalent to the right action of $W$ on $W_J\backslash W$.
   \begin{Exa}
   \label{exa:S4dot}
     Keep the notation in Example \ref{exa:S4} and take $w=1342\in {}^JW$. Then
     \begin{align*}
          w\cdot_J(1,2)&=3142,\\
          w\cdot_J(2,3)&=1342,\\
          w\cdot_J(3,4)&=1324.
     \end{align*}
   \end{Exa}
   
   The Hecke algebra $H_W$ is the algebra over $\mathbb{C}(q^{\frac{1}{2}})$ with basis $\{T_w, w\in W\}$ satisfying the following multiplication rules
   \[
   \begin{array}{rll}
       T_u\cdot T_w=&\!\!\!\!\!T_{uw}& \text{ if }\ell(uw)=\ell(u)+\ell(w),\\
       T_s^2=&\!\!\!\!\!(q-1)T_s+q& \text{ if }s\in S.
   \end{array}
   \]
   The Hecke algebra $H_W$ has an involution 
   \begin{align*}
       i\col H_W&\to H_W\\
       q&\mapsto q^{-1}\\
       T_w&\mapsto T_{w^{-1}}^{-1}.
   \end{align*}
   The Kazhdan-Lusztig elements $C'_w$ of $H_w$ are defined by 
   \begin{align*}
       i(C'_w)=&C'_w,\\
       q^{\frac{\ell(w)}{2}}C'_w=&\sum_{u\leq w}P_{u,w}T_u,
   \end{align*}
   where $P_{u,w}\in \mathbb{Z}[q]$, $P_{w,w}(q)=1$ and $\deg(P_{u,w})<(\ell(w)-\ell(u))/2$ for $u\neq w$. The existence of such elements $C'_w$ and polynomials $P_{u,w}$ is given by \cite{KL}.\par
   \begin{Exa}
     If $W=S_4$, we have
     \begin{align*}
     q^{2}C'_{3412}=&(1+q)T_{1234}+(1+q)T_{1324}+T_{2134}+T_{1243}+T_{1342}+T_{1423}+T_{2143}\\
                    &+T_{3124}+T_{2314}+T_{1432}+T_{3142}+T_{2413}+T_{3214}+T_{3412},
     \end{align*}
     which means that $P_{1234,3412}(q)=1+q$. 
   \end{Exa}
     Given a set $J\subset S$, we can define a right action (also called $\cdot_J$) of $H_W$ on 
     \[
     V_J:=\mathbb{C}(q^{\frac{1}{2}})^{{}^JW}=\bigoplus_{w\in {}^JW}\mathbb{C}(q^{\frac{1}{2}})w
     \]
     by
     \[
     w\cdot_J T_s:=q\max(w\cdot_Js,w)+\min(w\cdot_Js,w)-w.
     \]
     \begin{Exa}
      Keep the notation in examples \ref{exa:S4} and \ref{exa:S4dot}. We have that
       \begin{align*}
          w\cdot_JT_{(1,2)}&=[3142],\\
          w\cdot_JT_{(2,3)}&=q [1342],\\
          w\cdot_JT_{(3,4)}&=q [1324]+(q-1)[1342].
     \end{align*}
     \end{Exa}
     This is indeed an action as can be seen in \cite[Page 287]{GeckPf}. More precisely, the action in loc. cit. is defined as
     \[
     w \otimes_JT_s =\begin{cases}
       qw         & \text{ if }ws\in W_Jw\\
      ws         & \text{ if } ws\in {}^JW \text{ and }\ell(ws)>\ell(w)\\
       qws+(q-1)w & \text{ if }ws\in {}^JW \text{ and }\ell(ws)<\ell(w)
     \end{cases}
     \]
     Our action can be obtained from $\otimes_J$ simply by a base change. If $\overline{w}=\frac{w}{q^{\ell(w)}}$, then
     \[
     \overline{w}\otimes_JT_s=q\overline{\max(w\cdot_Js,w)}+\overline{\min(w\cdot_Js,w)}-\overline{w}.
     \]
     The action $\otimes_J$ is the induced representation  $\ind_{H_J}^{H_W}$. In particular, we have that
     \[
     \chi^{\ind_{H_j}^{H_W}}(a)=Tr((-)\otimes_J a)=Tr( (-) \cdot_Ja).
     \]
     For simplicity, we will write $c_J(a):=Tr( (-) \cdot_Ja)$.

\begin{Rem}
\label{rem:C'wRam}
       Let $W=S_n$ be the symmetric group, and $\chi^\lambda$ the irreducible characters of $S_n$ and, abusing notation, of $H_{n}$. Recall that for every $a\in H_n$ we have defined
       \[
       \ch(a)=\sum_{\lambda \vdash n} \chi^{\lambda}(a)s_{\lambda}.
       \]
       By \cite[Theorem 3.8]{Ram}, we have that 
      \begin{equation}
      \label{eq:cha_cJ}
       \ch(a)=\sum_{\lambda \vdash n} c_\lambda(a) m_\lambda.
       \end{equation}
       More precisely, in \cite[Equation (3.2)]{Ram} there is defined the action
       \[
     w \boxtimes_JT_s =\begin{cases}
       qw         & \text{ if }ws\in W_Jw\\
       q^{\frac{1}{2}}ws+(q-1)w        & \text{ if } ws\in {}^JW \text{ and }\ell(ws)>\ell(w)\\
       q^{\frac{1}{2}}ws & \text{ if }ws\in {}^JW \text{ and }\ell(ws)<\ell(w).
     \end{cases}
     \]
     Again, we have to change basis. As usual we write $w_0$ for the maximal element of $W$ and $w_0(J)$ for the maximal element of $W_J$. Let
     \[
     \overline{w}= \frac{w_0(J) w w_0}{q^{\frac{\ell(w)}{2}}}.
     \]
     By \cite[Proposition 2.5.4]{BjornerBrenti} and the fact that $w_0s=sw_0$ for every $s\in S$, we have that
     \[
     \overline{w}\boxtimes T_s =\begin{cases}
       q\overline{w}        & \text{ if }ws\in W_Jw\\
       \overline{ws}+(q-1)\overline{w}        & \text{ if } ws\in {}^JW \text{ and }\ell(ws)<\ell(w)\\
       qws & \text{ if }ws\in {}^JW \text{ and }\ell(ws)>\ell(w),
     \end{cases}
     \]
     which is precisely the action $\cdot_J$.
\end{Rem}

We are now interested in the elements $(1+T_{s_1})(1+T_{s_2})\ldots (1+T_{s_{n}})$, which correspond to the Bott-Samelson resolution of the Schubert varieties (see Section \ref{subsec:flag}). For each sequence $\si=(s_1,s_2,\ldots, s_n)$ we define
\[
c_J(\si)=c_J((1+T_{s_1})(1+T_{s_2})\ldots (1+T_{s_{n}}))=Tr((-)\cdot_J(1+T_{s_1})(1+T_{s_2})\ldots (1+T_{s_{n}})).
\]
Note that for each $w\in {}^JW$, we have
\[
w\cdot_J (1+T_s)=q\max(w\cdot_Js,w)+\min(w\cdot_Js,w).
\]

   Alternatively we can compute $c_J(\si)$ as follows. Define $\phi_{J,s,i}\col {}^JW\to {}^JW$ for $i=0,1$ as
\begin{align*}
\phi_{J,s,1}(w)&=\max(w\cdot_J s,w)\\
\phi_{J,s,0}(w)&=\min(w\cdot_J s,w).
\end{align*}
Moreover, for each binary word  $\bi=(b_1,\ldots, b_n)$, define 
\begin{equation}
    \label{eq:PhiJLb}
    \phi_{J,\si,\bi}=\phi_{J,s_n,b_n}\circ\ldots\circ\phi_{J,s_1,b_1}.
\end{equation} Then we have 
\begin{equation}
    \label{eq:cJLW}
 c_J(\si)=\sum_{w\in {}^JW}\sum_{\bi \in \Bin_{J,\si}(w)} q^{|\bi|}
\end{equation}
where $|\bi|:=\sum b_i$ and $\Bin_{J,\si}(w):=\{\bi\in\{0,1\}^n; \phi_{J,\si,\bi}(w)=w\}$. When $J=\emptyset$,  we will write $\phi_{\si,\bi}$ instead of $\phi_{\emptyset,\si,\bi}$.\par
\begin{Exa}
\label{exa:cJs}
 Keep the notation in Examples \ref{exa:S4} and \ref{exa:S4dot}, and consider $\si=((1,2),(2,3))$. Then
 \begin{align*}
     [1234]\cdot_J(1+T_{(1,2)})(1+T_{(2,3)})=&(q+1)[1234]\cdot_J(1+T_{(2,3)})\\
                                            =&q(q+1)[1324]+(q+1)[1234]),\\
     [1324]\cdot_J(1+T_{(1,2)})(1+T_{(2,3)})=&(q[3124]+[1324])\cdot_J(1+T_{(2,3)})\\
                                            =&q(q+1)[3124]+q[1324]+[1234],\\
     [1342]\cdot_J(1+T_{(1,2)})(1+T_{(2,3)})=&(q[3142]+[1342])\cdot_J(1+T_{(2,3)})=\\
                                            =&q^2[3412]+q[3142]+(q+1)[1342],\\
     [3124]\cdot_J(1+T_{(1,2)})(1+T_{(2,3)})=&(q[3124]+[1324])\cdot_J(1+T_{(2,3)})=\\
                                            =&q(q+1)[3124]+q[1324]+[1234],\\
     [3142]\cdot_J(1+T_{(1,2)})(1+T_{(2,3)})=&(q[3142]+[1342])\cdot_J(1+T_{(2,3)})=\\
                                            =&q^2[3412]+q[3142]+(q+1)[1342],\\
     [3412]\cdot_J(1+T_{(1,2)})(1+T_{(2,3)})=&(q+1)[3412]\cdot_J(1+T_{(2,3)})\\
                                             =&q(q+1)[3412]+(q+1)[3142].\\
 \end{align*}
 Hence 
 \[
 c_J(L)=(q+1)+q+(q+1)+q(q+1)+q+q(q+1)=2q^2+6q+2.
 \]
 Alternatively, we have the following words $\bi\in\{0,1\}^2$ satisfying $\phi_{J,\si,\bi}(w)=w$ for each $w\in{}^JW$.
 \begin{align*}
     \Bin_{J,\si}(1234)=& \{(1,0), (0,0)\}\\
     \Bin_{J,\si}(1324)=& \{(0,1)\}\\
     \Bin_{J,\si}(1342)=&\{(0,0),(0,1)\}\\
     \Bin_{J,\si}(3124)=&\{(1,1),(1,0)\}\\
     \Bin_{J,\si}(3142)=&\{(1,0)\}\\
     \Bin_{J,\si}(3412)=&\{(1,1),(0,1)\}.
 \end{align*}
 which recovers $c_J(L)$.
\end{Exa}

\begin{proposition}
\label{prop:cJmonic}
   If the sequence $\si=(s_1,\ldots, s_n)$ is such that for each simple reflection $s\in S$ there exists $i=1,\ldots, n$ with $s_i=s$, then 
   \[
    c_J(\si)=q^n+[\textnormal{lower degrees}].
   \]
\end{proposition}
\begin{proof}
    By Equation \eqref{eq:cJLW} it suffices to prove that there exists exactly one element $w\in {}^JW$ such that $\phi_{J,\si,\bi}(w)=w$ for $\bi=(1,1,\ldots,1)$. Let $w_0\in {}^JW$ be such that $\phi_{J,\si,\bi}(w_0)=w_0$, and define $w_{i+1}:=\phi_{J,\si|_i,\bi|_i}(w_i)$ for $i=0,\ldots, n-1$. Since $w_{i+1}=\max(w_{i}\cdot_J s_{i+1},w_i)$ and $w_0=w_n$, we must have that $w_{i+1}=w_i$ for all $i=0,\ldots, n-1$. This means that $w\geq w\cdot_J s_i$ for every $i=1,\ldots, n$, which implies that $w\geq w\cdot_j s$ for every $s\in S$ (by the hypothesis that each element $s\in S$ appears in the sequence $\si$). This happens if and only if $w$ is the unique maximal element of ${}^JW$, which finishes the proof.
\end{proof}

\begin{Def}
\label{def:good}
Fix $J\subset S$, and define $D_{\si,J}(w):=D_{\si}(w)\cap J$. Moreover, for $\si=(s_1,\ldots, s_n)$ and $\bi=(b_1,\ldots, b_n)$ let us define $\si|_i=(s_1,\ldots, s_i)$ and $\bi|_i=(b_1,\ldots, b_i)$. We say that a word $\bi$ is $(L,J)$-\emph{good} for $w$ if the following holds:
  \begin{enumerate}
      \item $\phi_{\si,\bi}(w)=w$, and
      \item for each $s\in D_{\si,J}(w)$, there exists $i\in \{1,\ldots,n\}$ such that $\phi_{\si|_{i-1},\bi|_{i-1}}(w)\cdot s_{i}=s\cdot \phi_{\si|_{i-1},\bi|_{i-1}}(w)$. 
  \end{enumerate}
  \end{Def}
  The intuition behind this definition is Theorem \ref{thm:hbspaving}.
  
\begin{Exa}
\label{exa:good}
  Consider $W=S_4$, $w=4213$,  $\si=((2,3),(1,2),(3,4),(2,3),(1,2),(3,4))$, and $J=S$. Then $\D_\si(w)=\{(1,2),(3,4)\}$, and we claim that the word $\bi=(1,0,1,1,1,0)$ is a $(L,J)$-good word for $w$. For simplicity let $w_i:=\phi_{\si|_{i},\bi|_i}(w)$.\par
  \begin{itemize}
      \item Since $s_1=(2,3)$ and $b_1=1$, we have $ws_1=4123< w$ and $w_1=\max(ws_1,w)=w=4213$. Moreover, it holds that $(1,2)w=4123=ws_1$.
      \item Since, $s_2=(1,2)$ and $b_2=0$, we have $w_1s_2=2413<w_1$ and $w_2=\min(w_1s_2,w_1)=w_1s_2=2413$. Moreover, it holds that $(2,4)w_1=2413=w_1s_2$.
      \item Since, $s_3=(3,4)$ and $b_3=1$, we have $w_2s_3=2431>w_2$ and $w_3=\max(w_2s_3,w_2)=w_2s_3=2431$. Moreover, it holds that $(1,3)w_2=2413=w_2s_3$.
      \item Since, $s_4=(2,3)$ and $b_4=1$, we have $w_3s_4=2341<w_3$ and $w_4=\max(w_3s_4,w_3)=w_3=2431$. Moreover, it holds that $(3,4)w_3=2341=w_3s_4$.
      \item Since, $s_5=(1,2)$ and $b_5=1$, we have $w_4s_5=4231>w_4$ and $w_5=\max(w_4s_5,w_4)=w_4s_5=4231$. Moreover, it holds that $(2,4)w_4=4231=w_4s_5$.
      \item Since, $s_6=(3,4)$ and $b_6=0$, we have $w_5s_6=4213<w_5$ and $w_6=\min(w_5s_6,w_5)=w_5s_6=4213$. Moreover, it holds that $(1,3)w_5=4213=w_5s_6$.
  \end{itemize}
  Then we have that the conditions of Definition \ref{def:good} are satisfied. Indeed, we have that $w_6=w$ and for each simple reflection $s\in \{(1,2),(3,4)\}=D_{\si,J}(w)$ there exists $i\in\{1,2,3,4\}$ such that $w_{i-1}s_i=sw_i$. If $s=(1,2)$, then $i=1$, and if $s=(3,4)$, then $i=4$.
  
\end{Exa}  
  
  Let $\Good(L,J,w)$ be the set of all $(L,J)$-good words for $w$ and define 
\[
c'_J(L)=\sum_{w\in W}\sum_{b \in Good(L,J,w)} q^{|b|}.
\]
The main goal of this section is to prove the following theorem.
\begin{theorem}
\label{thm:TrTr}
Let $J$ be a subset of $S$ and $\si=(s_1,\ldots, s_n)$ a sequence of simple reflections. We have that 
\[
c'_J(L)=c_J(L).
\]
\end{theorem}
See also \cite{KLS} for combinatorial characterizations of the induced sign characters.

\begin{Exa}
\label{exa:S4good}
Keeping the notation in Examples \ref{exa:S4}, \ref{exa:S4dot} and \ref{exa:cJs}. We have that
\begin{align*}
    \Good(J,\si,1234)=&\{(0,0)\},& \Good(J,\si,1324)=&\{(0,1)\},& \Good(J,\si,1342)=&\{(0,0)\},\\
    \Good(J,\si,1432)=&\{(0,1)\},& \Good(J,\si,2134)=&\{(1,0)\},& \Good(J,\si,3124)=&\{(1,0)\},\\
    \Good(J,\si,3142)=&\{(1,0)\},& \Good(J,\si,3214)=&\{(1,1)\},& \Good(J,\si,3412)=&\{(0,1)\},\\
    \Good(J,\si,4312)=&\{(1,1)\}.
\end{align*}
While the other good sets are empty. This tells us that $c'_J(L)=2q^2+6q+2$, which agrees with $c_J(L)$.
\end{Exa}
We will need a few results before we are able to prove Theorem \ref{thm:TrTr}. We begin with the following lemma.
\begin{lemma}
\label{lem:phis}
Let $s, s'\in S$ and $b=0,1$. Then
\[
\phi_{s,b}(s'w)=\begin{cases}
\phi_{s,b}(w) & \text{if }s'w=ws\\
s'\phi_{s,b}(w)& \text{otherwise.}\\
\end{cases}
\]
Moreover, if $s'w>_\si w$, then $\phi_{s,b}(s'w)\geq _\si\phi_{s,b}(w)$. This means that $\phi_{s,b}\col W\to W$ is a weak-left order preserving map.
\end{lemma}
\begin{proof}
First we will assume that $s'w> w$ and $ws>w$ (in particular $s'w>_Lw$). By Lemma \ref{lem:s'ws}, we have two cases. In the first case we have $s'ws=w$, so that 
\[
\phi_{s,1}(s'w)=\max(s'ws,s'w)=\max(w,ws)=\phi_{s,1}(w).
\]
The same argument holds for $b=0$. In the second case, we have $s'ws>ws$ and $s'ws>s'w$, so 
 \begin{align*}
 &\phi_{s,1}(s'w)=\max(s'ws,s'w)=s'ws=s'\max(ws,w)=s'\phi_{s,1}(w)\\
 &\phi_{s,0}(s'w)=\min(s'ws,s'w)=s'w=s'\min(ws,w)=s'\phi_{s,0}(w)
\end{align*}
 Since $s'ws>_Lws$, we have $\phi_{s,1}(s'w)>_\si\phi_{s,1}(w)$, and since $s'w>_Lw$, we have $\phi_{s,0}(s'w)>_\si\phi_{s,0}(w)$.\par
 
    The case $s'w<w$ follows from the result above for $w'=s'w$, and the case $ws<w$ from the fact that $\phi_{s,b}(w)=\phi_{s,b}(ws)$.
 \end{proof}

 \begin{corollary}
 \label{cor:phiL}
For every $\si=(s_1,\ldots, s_k)$, $s'\in S$ and binary word $\bi$, we have that
\[
\phi_{\si,\bi}(s'w)=\begin{cases}
\phi_{\si,\bi}(w) & \text{if }s'\phi_{\si|_{i-1},\bi|_{i-1}}(w)=\phi_{\si|_{i-1},\bi|_{i-1}}(w)s_i\text{ for some i}\\
s'\phi_{\si,\bi}(w)& \text{otherwise.}\\
\end{cases}
\]
Moreover, $\phi_{\si,\bi}\colon W\to W$ is a weak-left order preserving map of posets.
 \end{corollary}
\begin{proof}
This follows from successive applications of Lemma \ref{lem:phis}.
\end{proof}

\begin{lemma}
\label{Lem:phiWJw}
  Consider $J\subset S$,  $\overline{w}\in {}^JW$, $\si=(s_1,\ldots, s_n)$ and $\bi\in \{0,1\}^n$. Then
  \[
  \phi_{\si,\bi}(W_J\overline{w})=W_J\phi_{J,\si,\bi}(\overline{w}).
  \]
  In particular, if $\phi_{J,\si,\bi}(\overline{w})=\overline{w}$ we have that
  \[
  \phi_{\si,\bi}(W_J\overline{w})=W_J\overline{w}.
  \]
\end{lemma}
\begin{proof}
Let $w\in W_J\overline{w}$. That means that there exists $u\in W_J$ such that $w=u\overline{w}$ with $u\in W_J$ and $\ell(w)=\ell(u)+\ell(\overline{w})$. Let $s\in S$, then we have two cases: Either
$\overline{w}s\in {}^JW$ or $W_J\overline{w}=W_J\overline{w}s$, so that either $ws=u \overline{w}s$ or $ws=u'\overline{w}$ for $u'\in W_J$. In the first case we have
\[
\phi_{s,b}(w)=u\phi_{J,s,b}(\overline{w}),
\]
while in the second case we have $\phi_{J,s,b}(w)=w$, hence
\[
\phi_{s,b}(w)=\begin{cases}
u\phi_{J,s,b}(\overline{w})\\
u'\phi_{J,s,b}(\overline{w}).
\end{cases}
\]
This proves that
\[
\phi_{s,b}(W_J\overline{w})=W_J\phi_{J,s,b}(w).
\]
With successive applications we get the result of the Lemma.
\end{proof}
For each $J\subset S$, consider the sets
\begin{align*}
    F_1:=&\{(w,\bi); \bi \in \Good(L,J,w)\}\subset W\times \{0,1\}^n,\\
    F_2:=&\{(w,\bi); \phi_{J,\si,\bi}(w)=w\}\subset {}^JW\times \{0,1\}^n,
\end{align*}
and define the function 
\begin{align*}
\kappa_J\col F_1&\to F_2\\
         (w,\bi)&\mapsto (\overline{w},\bi).
\end{align*}
where $\overline{w}$ is the minimal representative of $W$ in $W_Jw$.

\begin{Exa}
  Keep the notation in Examples \ref{exa:S4}, \ref{exa:S4dot}, \ref{exa:cJs}, and \ref{exa:S4good}. Then
  \[
  F_1=\left\{\begin{array}{l}
  (1234,(0,0)),(1324,(0,1)), (1342,(0,0)), (1432,(0,1)), (2134,(1,0)),\\
  (3124,(1,0)),   (3142, (1,0)), (3214,(1,1)), (3412,(0,1)), (4312,(1,1)) \end{array} \right\}
  \]
  and 
  \[
  F_2=\left\{\begin{array}{l}
  (1234,(0,0)),(1234,(1,0)), (1324,(0,1)), (1342,(0,0)), (1342,(0,1)),\\
  (3124,(1,0)), (3124,(1,1)), (3142, (1,0)), (3412,(0,1)), (3412,(0,1))
  \end{array}\right\}.
  \]
  Moreover, we can see that
  \begin{align*}
  1234,2134&\in W_J[1234],& 1324&\in W_J[1324] & 1342, 1432&\in W_J[1342]\\
  3124, 3214&\in W_J[3124]& 3142& \in W_J[3142]& 3412, 4312&\in W_J[3412].
  \end{align*}
  We can see that $\kappa_J$ is a bijection. Proposition \ref{prop:bijection} below proves that this is always the case.
\end{Exa}

\begin{proposition}
\label{prop:bijection}
The function $\kappa_J$ is a bijection. 
\end{proposition}
\begin{proof}
Fix $\overline{w}\in {}^JW$ and $\bi\in \{0,1\}^n$ with $\phi_{J,\si,\bi}(\overline{w})=\overline{w}$. Let us prove that exists exactly one $w\in W_J\overline{w}$ such that $\bi \in Good(L,J,w)$. Consider $\phi_{\si,\bi}\colon W\to W$. Since $\phi_{J,\si,\bi}(\overline{w})=\overline{w}$, we have by Lemma \ref{Lem:phiWJw} that $\phi_{\si,\bi}(W_J\overline{w})=W_J\overline{w}$. Note that $\phi_{\si,\bi}|_{W_J\overline{w}}$  still is weak-left order preserving on $W_J\overline{w}$. By Tarski's fixed point theorem (\cite{Tarski}), the set of fixed points of $\phi:=\phi_{\si,\bi}|_{W_j\overline{w}}$ is non-empty and has a minimal element $w$. We will prove that this $w$ satisfies $\bi \in \Good(L,J,w)$ and is the unique element of $W_J\overline{w}$ with this property.\par

    Choosing $s'\in D_{\si,J}(w)$, we have $s'w<_\si w$ and $s'w\in W_J\overline{w}$. Since $w$ is minimal, we have $\phi(s'w)\neq s'w$. On the other hand, by Corollary \ref{cor:phiL}, we have that $\phi(s'w)=\phi(w)=w$ or $\phi(s'w)=s'\phi(w)=s'w$, and therefore $\phi(s'w)=\phi(w)$. Again by Corollary \ref{cor:phiL}, there exists $i$ such that $s'\phi_{\si|_{i-1}},\bi|_{i-1}(w)=\phi_{\si|_{i-1}},\bi|_{i-1}(w)s_i$.  This proves that $\bi\in \Good(L,J,w)$.\par
    
    Choose a fixed point $\widehat{w}\in W_J\overline{w}$ that it is not minimal. By Lemma \ref{lem:chain} below there exists $s'\in D_{\si,J}(\widehat{w})$ such that $s'\widehat{w}$ is a fixed point of $\phi$. 
    
    Repeating the argument above, we have that, in order for $s'\widehat{w}$ be a fixed point, it must hold that $s'\phi_{\si|_{i-1}},\bi|_{i-1}(w)\neq\phi_{\si|_{i-1}},\bi|_{i-1}(w)s_i$ for every $i$. This means that $\bi\notin \Good(L,J,\widehat{w})$.
\end{proof}

\begin{lemma}
\label{lem:chain}
Let $w$ be the minimal fixed point and let $\widehat{w}$ be a fixed point. Choose a maximal chain $w<_Lw_1<_\si\ldots <_\si w_k<_\si\widehat{w}$ in $W_J\overline{w}$. Then $w_i$ is a fixed point for every $i$. In particular,  there exists $s'\in D_{\si,J}(\widehat{w})$ such that $s'\widehat{w}$ is a fixed point.
\end{lemma}
\begin{proof}
Since $\phi$ is order preserving we have that $w\leq_\si\phi(w_1)\leq\ldots\leq\phi(w_k)\leq\widehat{w}$. Since $w_{i+1}=s'_iw_i$, by Lemma \ref{lem:phis} we have that $\ell(\phi(w_{i+1}))=\phi(w_i)$ or $\ell(\phi(w_{i+1}))=\phi(w_i)+1$. Since $\ell(\phi(\widehat{w}))=\ell(\widehat{w})=\ell(w)+k+1=\ell(\phi(w))+k+1$, we have that $\ell(\phi(w_{i+1}))=\phi(w_i)+1$ holds for all $i=1,\ldots, k$. In particular $\phi(w_{i+1})=s'_i\phi(w_i)$, which proves that $\phi(w_i)=w_i$ for all $i$. Moreover, $w_k=s'\widehat{w}$ for some $s'\in D_{\si,J}(w)$.
\end{proof}
This proves Theorem \ref{thm:TrTr}.

 To finish this section we restrict ourselves to the case where $W$ is a Weyl group.  In this case, we have bijections between $S$ and $\Delta$ (as seen in Section \ref{sec:alggroups}) and between $\Tr$ and $\Phi^+$. For a transposition $t\in \Tr$ we write $\gamma_t$ for the induced positive root, and conversely if $\gamma\in \Phi^+$ is a positive root, we write $t_{\gamma}$ for the induced transposition. This bijection between $\Tr$ and $\Phi^+$ induces a bijection between $\Phi(w)$ and $\Tr_\si(w)$ for every $w\in W$.\par

 \begin{Def}
 \label{def:goodseq}
  Let $\underline \delta:=(\delta_i)_{i=1,\ldots,m}$ be a sequence of positive roots. We say that $\underline \delta$ is a \emph{good sequence} for $w$ if for every $\gamma\in \Phi(w)$ at least one of the following holds:
 \begin{enumerate}
     \item There exists $i$ such that $\gamma=\delta_i$.
     \item For all $\alpha\in \Delta$ such that $\gamma-\alpha\in \Phi\setminus\Phi(w)$, there exists $i$ such that $\gamma-\alpha=\delta_i$ and $\delta_j\neq\alpha$ for every $j<i$.
 \end{enumerate}
 \end{Def}

\begin{Def}
 \label{def:goodseq1}
  Let $\si=(s_1,\ldots, s_n)$ be a sequence of simple reflections, $J\subset S$, $w\in W$ and $\bi\in \Good(L,J,w)$. Write $w=u\overline{w}$ with $u\in W_J$ and $\overline{w}\in {}^JW$, and define $w_i:=\phi_{\si_i,\bi|_{i}}(w)$.  Let $j_i$ be the unique integer such that $b_{j_i}=1$ and $\sum_{l=1}^{j_i}{b_l}=i$. Let $t_i\in \Tr$ be the transposition such that $t_iw_{j_i-1}=w_{j_i-1}s_{j_i}$.  Define the sequence $\underline{\delta}=\underline{\delta}(L,J,w,\bi)=(\delta_1,\ldots, \delta_{|\bi|})$ by $\delta_i=\gamma_{t_i}$.
\end{Def}

\begin{Exa}
\label{exa:goodseq} Let $W$, $w$, $\si$ and $\bi$ as in Example \ref{exa:good}. Then the sequence $\underline{\delta}$ defined in Proposition \ref{prop:goodseq} is $\gamma_{1,2},\gamma_{1,3},\gamma_{3,4},\gamma_{2,4}$ corresponding to the transpositions $(1,2),(1,3),(3,4),(2,4)$.
\end{Exa}

\begin{proposition}
\label{prop:goodseq}
Let $\si=(s_1,\ldots, s_n)$ be a sequence of simple reflections (roots), $J\subset S$, $w\in W$ and $\bi\in \Good(L,J,w)$. Write $w=u\overline{w}$ with $\overline{w}\in {}^JW$, then the sequence $\underline{\delta}=\underline{\delta}(L,J,w,\bi)$ is a good sequence for $u$ in $W_J$.
\end{proposition}
\begin{proof}
As in Definition \ref{def:goodseq1}, write $w_i=\phi_{\si|_i,\bi|_i}(w)$. Choose $\gamma\in \Phi(u)\subset \Phi_J$ such that $\gamma\neq \delta_i$ for every $i=1,\ldots, |\bi|$. We claim that $\gamma\in \Phi(w_i)$ for every $i$. Indeed, we have that $\gamma\in \Phi(w)$ because
\[
\ell(t_\gamma w)=\ell(t_\gamma u\overline{w})\leq \ell(t_\gamma u)+\ell(\overline{w})<\ell(w)+\ell(\overline{w})=\ell(w).
\]
Assume by contradiction that there exists $i$ such that $\gamma\notin \Phi(w_i)$ and choose $k:=\max\{i;\gamma\notin \Phi(w_i)\}$. Clearly, $k\neq n$ because $w_n=w$. Moreover, since $\gamma\in \Phi(w_{k+1})\setminus \Phi(w_k)$, we have $w_{k+1}=w_ks_{k+1}$. By Lemma \ref{lem:wt} below, $w_{k+1}>w_k$ and $t_\gamma w_k=w_{k+1}s_{k+1}$. This means $b_{k+1}=1$ (because $w_{k+1}>w_k$) and hence $\gamma=\delta_{k+1}$, a contradiction. This proves that $\gamma\in \Phi(w_i)$ for every $i$.\par

Let $s\in J$ be such that $\gamma-\alpha_s\in \Phi$ (and hence $\gamma-\alpha_s\in \Phi_J$). Since both $\Phi(w)$ and $\Phi^+\setminus \Phi(w)$ are closed by sums, and since $\gamma\in \Phi(w)$, we have that if $\alpha_s\notin \Phi(w)$, then $\gamma-\alpha_s\in \Phi(w)$.\par

Assume that $\alpha_s\in \Phi(w)$ (equivalently, $s\in \Tr_\si(w)$). By Lemma \ref{lem:wt}, either $s\in \Tr_\si(w_k)$ for every $k$ or there exists $k$ such that $s\in \Tr_\si(w_{k+1})\setminus \Tr_\si(w_k)$.\par

  Let us assume first that $s\in \Tr_\si(w_k)$ for every $k$. Since $\bi$ is a $(L,J)$-good word for $w$, we have that there exists $k$ such that $s w_k=w_k s_{k+1}$. Moreover, since $s\in \Tr_\si(w_k)$, we have that $\ell(w_ks_{k+1})=\ell(sw_k)<\ell(w_k)$, so by Lemma \ref{lem:wt} and the fact that $s\in \Tr_\si(w_{k+1})$, we must have $b_{k+1}=1$.\par

In the second case, there exists $k$  such that $\Phi(w_{k+1})=\Phi(w_k)\setminus\{\alpha_s\}$. Since $w_n=w$, we must have that there also exists some $k$ such that $\Phi(w_{k+1})=\Phi(w_k)\cup\{\alpha_s\}$. In that case, we have that $w_{k+1}=\alpha w_k=w_ks_{k+1}$, $b_{k+1}=1$ and $w_{k+1}>w_k$.\par

   Let $k$ be the minimum $i$ such that $b_{i+1}=1$ and $sw_i=w_{i+1}s_{i+1}$, which exists by the discussion above. In particular, we have that $\delta_j\neq \alpha_s$ for every $j<k+1$. 
   
   Let us proceed by contradiction and assume that $\gamma-\alpha_s\notin \Phi(w)$ and that there does not exist $j\leq k$ with $\gamma-\alpha_s=\delta_j$. By Lemma \ref{lem:wt}, this means that $\gamma-\alpha_s\notin \Phi(w_j)$ for every $j\leq k$. Since $\gamma\in \Phi(w_j)$ for every $j$, we must have that $\alpha_s\in \Phi(w_j)$ for every $j\leq k$. In particular, we have that $\alpha_s\in \Phi(w_k)$, $\gamma-\alpha_s\notin \Phi(w_k)$, $\gamma\in \Phi(w_k)$ and that $w_k^{-1}sw_k=s_{k+1}$. This contradicts Lemma \ref{lem:alphabeta}
\end{proof}

\begin{lemma}
\label{lem:wt}
Let $w\in W$ and $s\in S$, and assume that $ws>w$. Then $\Tr_\si(ws)=\Tr_\si(w)\cup \{t\}$, where $t\in \Tr$ is such that $tw=ws$. Conversely, if $ws<w$ then $\Tr_\si(ws)=\Tr_\si(w)\setminus\{t\}$.
\end{lemma}
\begin{proof}
We first prove that $t\in \Tr_\si(ws)$. Since $t(ws)=w<ws$, then $t\in \Tr_\si(ws)$. Now let $t'\in \Tr_\si(w)$, so that $\ell(t'w)<\ell(w)$, which means
\[
\ell(t'ws)\leq \ell(t'w)+1\leq \ell(w) < \ell(ws).
\]
Then $t'\in \Tr_\si(ws)$, hence $\Tr_\si(w)\cup\{t\}\subset \Tr_\si(ws)$. But these sets both have cardinality $\ell(w)+1$, so $\Tr_\si(ws)=\Tr_\si(w)\cup\{t\}$ as desired. The second statement follows directly from the first.
\end{proof}

\begin{lemma}
\label{lem:alphabeta}
  Let $w\in W$, $\alpha\in \Phi(w)$ and $\beta\notin \Phi(w)$ be such that $\alpha+\beta\in \Phi(w)$. Then $w^{-1}(\alpha)\notin -\Delta$, or, equivalently, $w^{-1}s_\alpha w$ is not a simple transposition.
\end{lemma}
\begin{proof}
By the hypothesis we have that
\[
w^{-1}(\alpha)\in \Phi^{-}\text{, }w^{-1}(\beta)\in \Phi^+\text{ and }w^{-1}(\alpha+\beta)\in \Phi^{-}.
\]
Since $w^{-1}(\alpha)=w^{-1}(\alpha+\beta)-w^{-1}(\beta)$ we have that $w^{-1}(\alpha)\notin -\Delta$.
\end{proof}
 
 \section{Paving by affines the Bott-Samelson variety}
 \label{sec:pavingbs}

   In this section we will describe an affine paving of the Bott-Samelson variety. Usually, these affine pavings are described in terms of Tits \emph{buildings} and \emph{galleries}, see \cite{Titsresume}, \cite[Section 3]{Gaussent} and \cite[Chapter 14]{CC}. We will avoid defining such objects and give instead an explicit construction.\par
   
      Our construction will be done inductively. Let $\si=(s_1,\ldots, s_n)$ be a word in $S$. When $n=0$, then $\bs_\si=\flag$ and the paving will simply be the paving given by Schubert stratification, that is, $\Omega_{w,e}^\circ=B\dw B/B \subset \flag$. We have the following isomorphim (see Section \ref{sec:alggroups})
   \begin{align*}
    U^w&\to \Omega_{w,e}^\circ\\
    u&\mapsto u\dw B.
   \end{align*}
  
\begin{Exa}
\label{exa:Uwomega}
    Consider $w=4213\in S_4$. We saw in example \ref{exa:Uw} that 
    \[
    U^w=\left\{\left(\begin{array}{cccc}
       1 & * &0 & *\\
       0 & 1 &0 &*\\
       0 & 0 & 1 & *\\
       0 & 0 & 0 &1
        \end{array}\right)\right \}.
    \]
    This means that $\Omega_{w,e}^\circ$ is the set of flags associated to matrices of the form
    \[
    \left(\begin{array}{cccc}
       1 & x_1 &0 & x_2\\
       0 & 1 &0 &x_3\\
       0 & 0 & 1 & x_4\\
       0 & 0 & 0 &1
        \end{array}\right)\cdot \left( \begin{array}{cccc}
       0 & 0 &1 & 0\\
       0 & 1 &0 &0\\
       0 & 0 & 0 &1\\
       1 & 0 & 0 &0
        \end{array}\right) = \left(\begin{array}{cccc}
       x_2 & x_1 &1 & 0\\
       x_3 & 1 &0 &0\\
       x_4 & 0 & 0 & 1\\
       1 & 0 & 0 &0
        \end{array}\right)=:(v_1,v_2,v_3,v_4).
    \]
    The flag $V_\bullet$ can be recovered from the above matrix by setting $V_1=\langle v_1\rangle$, $V_2=\langle v_1,v_2\rangle$, $V_3=\langle v_1,v_2,v_3 \rangle$ and $V_4=\mathbb{C}^4$ (recall Example \ref{exa:gVbullet}). We will usually abuse notation and write $V_\bullet=(v_1,v_2,v_3,v_4)$.
\end{Exa}

  We generalize the above construction to any $n>0$. Let $\si=(s_1,\ldots, s_n)$ be a word in $S$, $w$ any element of $W$ (not necessarily the one given by $\si$), and $\bi=(b_1,\ldots, b_{n})$ a binary word of length $n$. Recall the definition of the function $\phi_{\si,\bi}\col W\to W$ in Equation \ref{eq:PhiJLb}. We write $w_i=\phi_{\si|_i,\bi|_i}(w)$ and, as usual, $\dw_i$ for the chosen representative of $w_i$ in $G$. Moreover, let $\gamma_i$ be the positive root such that $t_{\gamma_i}w_{i-1}=w_{i-1}s_i$.\par 
  
  For $\gamma\in \phi^+$ and $b=0,1$ define 
  \[
  U_{\gamma,b}:=\begin{cases}
  U_\gamma&\text{ if }b=1\\
  \{1\}&\text{ if }b=0,
  \end{cases}
  \]
    and consider the map $g_{\si,\bi}^{w}$ defined by
\begin{equation}
    \label{eq:gLbw}
\begin{aligned}
 U^w\times U_{\gamma_1,b_1}\times\cdots \times U_{\gamma_n,b_n}&\to \bs_{\si}   \\
 (u_0,u_1,\ldots, u_n)&\mapsto (u_0\dw B, u_0u_1\dw_1B,\ldots, u_0u_1\ldots u_n\dw_nB).
 \end{aligned}
 \end{equation}
 To motivate this definition we give an example.
 \begin{Exa}
 \label{exa:BSL}
     As in example \ref{exa:Uwomega}, consider $w=4213\in S_4$ and let $\si=(s_1)$, with $s_1=(2,3)$. Then 
     \[
     \bs_\si=\{(V_\bullet^1,V_\bullet^2); V_j^{1}=V_j^{2}\text{ for every }j=1,3,4\}.
     \]
     Assuming that $V_\bullet^1 \in \Omega_{w,e}^\circ$, we have  $V_1^1=\langle v_1\rangle$, $V_2^1=\langle v_1,v_2\rangle$, $V_3^1=\langle v_1,v_2,v_3 \rangle$ and $V_4^1=\mathbb{C}^4$ for $v_1,v_2,v_3,v_4$ as in Example \ref{exa:Uwomega}.\par
       If we fix $V_\bullet^1$, we have a $\mathbb{P}^1$ of choices for $V_\bullet^2$, in fact, we just have to choose $V_2^2$ satisfying 
       \[
       <v_1>=V_1^1=V_1^2\subset V_2^2 \subset V_3^2=V_3^1=\langle v_1,v_2,v_3\rangle.
       \]
       Which means that $V_2^2=\langle v_1, \lambda_1v_2+\lambda_2 v_3\rangle$. In particular if 
       \[
       \bs_{\si}^{w}:=\{(V_\bullet^1,V_\bullet^2)\in\bs_\si; V_\bullet^1\in \Omega_{w,e}^\circ\},
       \]
       then $\bs_{\si}^{w}=\Omega_{w,1}^\circ\times\mathbb{P}^1$. \par
       
       The idea is to stratify $\mathbb{P}^1$ as $\mathbb{P}^1=\mathbb{A}^1\cup \mathbb{A}^0$.  There are two natural choices for accomplishing this: 
       \begin{align}
           \mathbb{P}^1&=\{(1:\lambda_2)\} \cup \{(0:1)\},\label{eq:P1strat1}\\
           \mathbb{P}^1&=\{(\lambda_1:1)\} \cup \{(1:0)\}.\label{eq:P1strat2}
        \end{align}
      Let us take a closer look at each of these options, defining in both cases $\bs_{\si,0}^w:=\Omega_{w,1}^\circ\times\mathbb{A}^0$ and $\bs_{\si,1}^w:=\Omega_{w,e}^\circ\times \mathbb{A}^1$.\par

      Let us begin with the first possible stratification \eqref{eq:P1strat1}. If we have $(V_\bullet^1,V_\bullet^2)\in \bs_{\si,1}^w$, then $V_2^2=\langle v_1, v_2+\lambda_2v_3\rangle$, $V_3^2=\langle v_1, v_2+\lambda_2v_3, v_3 \rangle $, and we can write the matrix associated to $V_\bullet^2$ as
     \begin{equation}
     \label{eq:V2bullet1}   
     V_\bullet^2=\left(\begin{array}{cccc}
       x_2 & x_1+\lambda_2 &1 & 0\\
       x_3 & 1 &0 &0\\
       x_4 & 0 & 0 & 1\\
       1 & 0 & 0 &0
        \end{array}\right).
     \end{equation}
     If $(V_\bullet^1,V_\bullet^2)\in \bs_{\si,0}^w$, then $V_2^2=\langle v_1,v_3\rangle$ and $V_3^2=\langle v_1,v_3,v_2\rangle$, which means
     \begin{equation}
         \label{eq:V2bullet2}
     V_\bullet^2=\left(\begin{array}{cccc}
       x_2 & 1 &x_1 & 0\\
       x_3 & 0 &1 &0\\
       x_4 & 0 & 0 & 1\\
       1 & 0 & 0 &0
        \end{array}\right).
     \end{equation}

     On the other hand, if we choose stratification \eqref{eq:P1strat2} and consider $(V_\bullet^1,V_\bullet^2)\in \bs_{\si,1}^w$ then, $V_2^2=\langle v_1, \lambda_1v_2+v_3\rangle$ and $V_3^2=\langle v_1, \lambda_1v_2+v_3,v_2\rangle$, so we can write 
     \[
     V_\bullet^2=\left(\begin{array}{cccc}
       x_2 & \lambda_1x_1+1 &1 & 0\\
       x_3 & \lambda_1 &0 &0\\
       x_4 & 0 & 0 & 1\\
       1 & 0 & 0 &0
        \end{array}\right).
     \]
     If we have  $(V_\bullet^1,V_\bullet^2)\in \bs_{\si,0}^w$, then $V_2^2=\langle v_1, v_2\rangle$ and $V_3^2=\langle v_1,v_2, v_3\rangle$, thus
     \[
     V_\bullet^2= \left(\begin{array}{cccc}
       x_2 & x_1 &1 & 0\\
       x_3 & 1 &0 &0\\
       x_4 & 0 & 0 & 1\\
       1 & 0 & 0 &0
        \end{array}\right).
     \]
     
     Among these two choices, note that only stratification \eqref{eq:P1strat1} has the property that the matrix associated to $V_\bullet^2$ has a permutation of columns which turns it into an upper triangular matrix with all diagonal entries equal to $1$. That is to say, this matrix is of the form $u\dw_1$ for some $u\in U$ and $w_1\in W$. (In this case $w_1=w$, but that is not so in general.) For this reason, this is the stratification we will choose. \par

      More precisely, the matrices in Equations \eqref{eq:V2bullet1} and \eqref{eq:V2bullet2} can be rewritten as
     \begin{equation}
     \label{eq:matrixbraid}
     \left(\begin{array}{cccc}
       x_2 & x_1+\lambda_2 &1 & 0\\
       x_3 & 1 &0 &0\\
       x_4 & 0 & 0 & 1\\
       1 & 0 & 0 &0
        \end{array}\right)=
        \left(\begin{array}{cccc}
       1 & x_1 &0 & x_2\\
       0 & 1 &0 &x_3\\
       0 & 0& 1 & x_4\\
       0 & 0 & 0 &1
        \end{array}\right)\cdot
        \left(\begin{array}{cccc}
       1 & \lambda_2  & 0 & 0\\
       0 & 1 &0 &0\\
       0 & 0& 1 &0\\
       0 & 0 & 0 &1
        \end{array}\right)\cdot
        \left(\begin{array}{cccc}
       0 & 0 &1 & 0\\
       0 & 1 &0 &0\\
       0 & 0& 0 & 1\\
       1 & 0 & 0 &0
        \end{array}\right)
     \end{equation}
     and 
     \[
   \left(\begin{array}{cccc}
       x_2 & 1 &x_1 & 0\\
       x_3 & 0 &1 &0\\
       x_4 & 0 & 0 & 1\\
       1 & 0 & 0 &0
        \end{array}\right)=
        \left(\begin{array}{cccc}
       1 & x_1 &0 & x_2\\
       0 & 1 &0 &x_3\\
       0 & 0& 1 & x_4\\
       0 & 0 & 0 &1
        \end{array}\right)\cdot
        \left(\begin{array}{cccc}
       1 & 0  & 0 & 0\\
       0 & 1 &0 & 0\\
       0 & 0& 1 &0\\
       0 & 0 & 0 &1
        \end{array}\right)\cdot
        \left(\begin{array}{cccc}
       0 & 1 &0 & 0\\
       0 & 0 &1 &0\\
       0 & 0& 0 & 1\\
       1 & 0 & 0 &0
        \end{array}\right).
     \]
     Note that the permutation matrices appearing above are, respectively, the matrices of  $4213=w$ and $4123=ws_1=:w_1$ (see also the first item in Example \ref{exa:good}). Moreover, we have that $t_{\gamma_1}:=ws_1w^{-1}=(1,2)$, so
     \[
     U_{\gamma_1,1}=\left\{
      \left(\begin{array}{cccc}
       1 & \lambda_2  & 0 & 0\\
       0 & 1 &0 &0 \\
       0 & 0& 1 &0\\
       0 & 0 & 0 &1
        \end{array}\right)\right\}
        \quad\text{ and }\quad 
        U_{\gamma_1,0}=\left\{
      \left(\begin{array}{cccc}
       1 & 0  & 0 & 0\\
       0 & 1 &0 & 0\\
       0 & 0& 1 &0\\
       0 & 0 & 0 &1
        \end{array}\right)\right\}.
     \]
     
     This means that an element of $\bs_{\si,1}^w$ can be written as $(u_0\dw, u_0u_1\dw)$ for $u_0\in U^w$ and $u_1\in U_{\gamma_1}$, while an element of $\bs_{\si,0}^w$ is of the form $(u_0\dw, u_0\dw_1)$. It is also worth noting that $w_1<w$, hence $w_1=\phi_{\si,0}(w)$ and $w=\phi_{\si,1}(w)$. Finally, we have isomorphisms $U^w\times U_{\gamma_1,1}\to \bs_{\si,1}^w$ and $U^w\times U_{\gamma_1,0}\to \bs_{\si,0}^w$.

            If $\si=(s_1,\ldots, s_n)$ is a longer word, we can repeat the process above. Assuming that we have  found matrices for $V_\bullet^1,\ldots, V_\bullet^j$, we start with the matrix associated to $V_\bullet^{j}$ and find the possible matrices associated with $V_{\bullet}^{j+1}$ and choose the stratification such that this matrix associated with $V_{\bullet}^{n+1}$ can be made upper-triangular after a permutation of the columns.
 \end{Exa}
 
 \begin{Rem}
 It is also possible to write the map $g_{\si, \bi}^w$ is in terms of \emph{braid matrices}, that is, matrices of the form
 \[
 Br_i(\lambda)=\left(\begin{array}{cccccc}
      1&\dots  & && \dots &0\\
       \vdots &\ddots &&&\vdots\\
       0& \dots &0&1&\dots &0\\
       0&\dots &1&\lambda & \dots& 0\\
       \vdots &\ddots &&& \vdots\\
       0&\dots  & && \dots &1\\
 \end{array}\right).
 \]
 For instance, Equation \eqref{eq:matrixbraid} can be written as
 \[
 \left(\begin{array}{cccc}
       x_2 & x_1+\lambda_2 &1 & 0\\
       x_3 & 1 &0 &0\\
       x_4 & 0 & 0 & 1\\
       1 & 0 & 0 &0
        \end{array}\right)=
        \left(\begin{array}{cccc}
       x_2 & x_1 &1 & 0\\
       x_3 & 1 &0 &0\\
       x_4 & 0 & 0 & 1\\
       1 & 0 & 0 &0
        \end{array}\right)\cdot
        Br_2(\lambda_2)\cdot
        \left(\begin{array}{cccc}
       1 & 0 &0 & 0\\
       0 & 0 &1 &0\\
       0 & 1& 0 & 0\\
       0 & 0 & 0 &1
        \end{array}\right).
 \]
 \end{Rem}

\begin{lemma}
The map $g_{\si,\bi}^{w}$ is well defined, that is, $g_{\si,\bi}^{w}(u_0,\ldots, u_n)\in \bs_\si$ for every $(u_0,\ldots, u_n)\in U^w\times U_{\gamma_1,b_1}\times\cdots \times U_{\gamma_n,b_n}$.
\end{lemma}
\begin{proof}
It is enough to check that 
\[
(u_0u_1\ldots u_i\dw_i)^{-1} (u_0u_1\ldots u_{i+1}\dw_{i+1})\in P_{s_{i+1}}.
\]
However, we can simplify the expression to
\[
\dw_i^{-1}u_{i+1}\dw_{i+1}\in P_{s_{i+1}}.
\]

We have that either $w_{i+1}=w_i$ or $w_{i+1}=w_is_{i+1}$. Since $u_{i+1}\in U_{\gamma_{i+1}}$ and $t_{\gamma_{i+1}}=w_is_iw_i^{-1}$, we have  $U_{\gamma_{i+1}}=w_i U_{\pm \alpha_{s_{i+1}}}w_i^{-1}$, where $\alpha_{s_{i+1}}$ is the simple root associated to $s_{i+1}$. Since $w_i U_{\pm \alpha_{s_{i+1}}}w_i^{-1}\subset w_iP_{s_{i+1}}w_i^{-1}$ and $\ds_{i+1}\in P_{s_{i+1}}$, then
   \[
   w_i^{-1}u_{i+1}w_{i+1}\in w_i^{-1}w_iP_{s_{i+1}}w_i^{-1}w_{i+1}=P_{s_{i+1}},
   \]
   which finishes the proof.
   \end{proof}

We will now prove that $g_{\si,\bi}^w$ is an immersion. Let $\bs_{\si,\bi}^{w}$ be the image of the map $g_{\si,\bi}^{w}$.
\begin{proposition}
The maps $g_{\si,\bi}^{w}$ are isomorphisms onto their images, and the variety $\bs_\si$ is paved by the collection $\bs_{\si,\bi}^{w}$ as $\bi$ varies in $\{0,1\}^n$ and $w$ in $W$.
\end{proposition}
\begin{proof}
We proceed by induction on $n$. The case $n=0$ follows from the Bruhat decomposition. Assume that $n\geq1$, let  $\si':=(s_1,\ldots, s_{n-1})$ be a truncation of $\si$, and suppose that $g_{\si',\bi'}^{w}$ is an isomorphim onto its image for each $w\in W$ and $\bi'\in \{0,1\}^{n-1}$. Moreover assume that $(X_{\si',\bi'}^{w})_{w,\bi'}$ pave $\bs_{\si'}$. Fix $w$ and $\bi$, and define as before $w_i=\phi_{\si|_i,\bi|_i}$ and $\gamma_iw_{i-1}=w_{i-1}s_i$. Also, let $\bi'=\bi|_{n-1}$. Let $p\col \bs_\si\to \bs_{\si'}$  and $q\col U^w\times U_{\gamma_1,b_1}\times\ldots\times U_{\gamma_n,b_n} \to U^w\times U_{\gamma_1,b_1}\times \ldots\times U_{\gamma_{n-1},b_{n-1}}$ be the projections.  It is clear from the definition that $p\circ g_{\si,\bi}^{w}=g_{\si',\bi'}^{w}\circ q$,
  \[
  \begin{tikzcd}
  U^w\times U_{\gamma_1,b_1}\times\ldots\times U_{\gamma_n,b_n} \ar[d,"q"] \ar[r,"{g_{\si,\bi}^w}"] & \bs_\si \ar[d,"p"]\\
    U^w\times U_{\gamma_1,b_1}\times \ldots\times U_{\gamma_{n-1},b_{n-1}}\ar[r,"{g_{\si',\bi'}^w}"] & \bs_{\si'}.
   \end{tikzcd}
   \]
   
   Let $\bi_0:=(\bi',0)$ and $\bi_1:=(\bi',1)$. We have that  $\bs_{\si,\bi_0}^{w}, \bs_{\si,\bi_1}^{w}\subset  p^{-1}(\bs_{\si',\bi'}^{w})$. Define $w':=\phi_{\si',\bi'}(w)$ and let $x\in \bs_{\si',\bi'}^{w}$ be a point with last coordinate $u_x\dw'B$ for some $u_x\in U$. Then 
  \[
  p^{-1}(x)=\{(x,gB); g\in u_x\dw'P_{s_{n}}\}.
  \]
   Recall that, if $\alpha_{s_n}$ is the simple root associated to $s_n$, then (see Equation \ref{eq:PsB})
  \[
  P_{s_n}=U_{\alpha_{s_n}}\ds_nB\cup B\quad\text{ and }  P_{s_n}=\ds_nU_{\alpha_{s_n}}\ds_nB\cup \ds_nB.
  \]
  Now we consider two cases. The first is when $w's_{n}>w'$. In this case we have that $\phi_{\si,\bi_1}(w)=w's_n$ and $\phi_{\si,\bi_0}(w)=w'$. Moreover $\dw' U_{\alpha_{s_n}} \dw'^{-1}=U_{\gamma_n}$, thus
  \begin{align*}
  u_x\dw'P_{s_{n}}=&u_x\dw' (U_{\alpha}\ds_nB\cup B)\\
                =&u_x\dw'U_{\alpha}\dw'^{-1} \dw'\ds_nB\cup u_x \dw'B\\
                =&u_xU_{\gamma_n} \dw'\ds_nB\cup u_x \dw'B\\
                =&u_xU_{\gamma_n} \dw'\ds_nB\cup u_x \dw'B
  \end{align*}
  The second case is when $w's_n<w'$. In this case we have $\phi_{\si,\bi_1}(w)=w'$ and $\phi_{\si.\bi_0}(w)=w's_n$. Moreover $\dw'\ds_n U_{\alpha_{s_n}} \ds_n\dw'^{-1}=U_{\gamma_n}$, so that
\begin{align*}
  u_x\dw'P_{s_{n}}=&u_x\dw' (\ds_nU_{\alpha}\ds_nB\cup \ds_nB)\\
                =&u_x\dw'\ds_nU_{s_n}\ds_n\dw'^{-1} \dw'B\cup u_x \dw'\ds_nB\\
                =&u_xU_{\gamma_n} \dw'B\cup u_x \dw'\ds_nB.
  \end{align*}
  The result follows.
\end{proof}

\section{Basic properties of Lusztig varieties}
\label{sec:twisted}
 In this section we give basic results about the varieties $\h_w(X)$ and $\hbs_{\si}(X)$ defined in Definition \ref{def:twisted}. We keep the notation of Section \ref{sec:alggroups} and consider the variety $\hb\subset G\times\flag\times\flag$ given by
  \[
  \hb :=\{(X, g_1B, g_2B); g_1^{-1}Xg_2\in B\}=\{(X, g_1B, g_2B); X\in g_1Bg_2^{-1}\}.
  \]
  Note that the projection  $\hb \to \flag\times\flag$ is a $B$-bundle (i.e., a smooth morphism with fibers isomorphic to $B$). Moreover, $\hb$ is isomorphic to $G\times \B$ via the map
  \begin{align*}
  G\times \B&\to \hb\\
  (X, gB) &\mapsto (X,XgB,gB).
  \end{align*}
  Also we have that $G$ acts on $Y$ by conjugation on the first coordinate and left multiplication on the other coordinates:
  \begin{align*}
      G\times Y&\to Y\\
      (g, (X,g_1B,g_2B))&\mapsto (gXg^{-1}, gg_1B, gg_2B).
  \end{align*}
  
 Let $w\in W$ and $\si=(s_1,\ldots, s_{\ell(w)})$ a reduced word for $w$. We define the variety $\h_w$ (and its associated strata $\h_w^{\circ})$ and the variety $\hbs_\si$ as
\begin{align*}
    \h_w&:=\hb\times_{\flag\times\flag}\Omega_w & \h_w^\circ&:=\hb\times_{\flag\times\flag}\Omega_w^\circ\\
    \hbs_\si&:=\hb\times_{\flag\times\flag}\bs_\si.
\end{align*}
More explicitly we have
\begin{align*}
    \h_w&:=\overline{\{(X,gB); g^{-1}Xg\in BwB\}}\subset G\times \B,\\
    \h_w^\circ&:=\{(X,gB); g^{-1}Xg\in BwB\}\subset G\times \B,\\
    \hbs_\si&:=\{(X,(g_0B,g_1B,\ldots, g_{\ell(w)}B)); g_0^{-1}Xg_{\ell(w)}\in B\}\subset G\times \bs_\si.
\end{align*}
As before, $G$ acts on these $3$ varieties by conjugation on the first coordinate and left multiplication on the other coordinates.\par
We have the following commutative diagram
\[
\begin{tikzcd}
& \hbs_\si \ar[rr] \ar[ld] \ar[ldd]& &\bs_\si \ar[ld] \ar[ldd]\\
\h_w \ar[d] \ar[rr]&&\Omega_w\ar[d] &\\
\hb \ar[rr]&& \flag\times\flag.\\
\end{tikzcd}
\]
Via the first projection $\hb\to G$, we have induced maps $\h_w\to G$ and $\hbs_\si\to G$. These maps are equivariant when we consider the adjoint action of $G$ on itself. For $X\in G$, we denote by $\h_w(X)$ and $\hbs_\si(X)$ the fibers of $\h_w\to G$ and $\hbs_\si\to G$ over $X$. Set theoretically, these fibers can be described as (see Definition \ref{def:twisted})
\begin{equation}
    \label{eq:fiberh}
\begin{aligned}
\h_w(X)&=\{gB; g^{-1}Xg\in B\}\subset \flag\\
\hbs_\si(X)&=\{(g_0B,g_1B,\ldots, g_{\ell(w)}B)\in \bs_\si; g_0^{-1}Xg_{\ell(w)}\in B\}.
\end{aligned}
\end{equation}
\begin{Rem}
 We can define $\hbs_\si$ and $\hbs_\si(X)$ for every word $\si$, not just the reduced ones. 
\end{Rem}

 Recall that $G^{rs}$ denotes the open subset of $G$ consisting of the regular semisimple elements, and denote by $\h_w^{rs}$, $\h_w^{rs,\circ}$, and $\hbs_\si^{rs}$ the restrictions of $\h_w$, $\h_w^{\circ}$ and $\hbs_\si$ to $G^{rs}$.
  
\begin{Exa}
 Let $G=GL(\mathbb{C},4)$ and take
 \[
 X_1:=\left(\begin{array}{cccc}
    1 & 0 & 0 &0\\
    0 & 2 & 0& 0\\
    0 & 0 & 3&0\\
    0 & 0& 0& 4
 \end{array}\right)
 \quad\text{ and }\quad
 X_2:=\left(\begin{array}{cccc}
    1 & 1 & 0 & 0\\
    0 & 1 & 1 & 0\\
    0 & 0 & 1 & 1\\
    0 & 0 & 0 & 1
 \end{array}\right).
 \]
 Also consider $w=2413$. Then 
 \[
 \h_w(X)=\{V_\bullet; XV_1\subset V_2, V_2\subset XV_3\}.
 \]
 We try to give a more geometric description of $\h_w(X_1)$ and $\h_w(X_2)$. First, we consider the forgetful map
 \[
 \h_w(X)\to \Gr(2,4).
 \]
 Then it is clear that its image is given by $\{V_2\subset \mathbb{C}^4; \dim(XV_2\cap V_2)\geq 1\}$. \par
 Recall the Pl\"ucker coordinates of $\Gr(2,4)$: choose a basis $v_1,v_2$ of $V_2$, and let $M_{V_2}=(v_1, v_2)$ be the $4\times 2$ matrix with columns $v_1$ and $v_2$. Then the $2\times 2$ minors $p_{ij}$ given by lines $i$ and $j$ of $M_{V_2}$ define an embedding
 \begin{align*}
 Gr(2,4)&\to \mathbb{P}^5\\
  V_2&\mapsto(p_{12}:p_{13}:p_{14}:p_{23}:p_{24}:p_{34}).
 \end{align*}
 The image of $\Gr(2,4)$ has equation $p_{12}p_{34}-p_{13}p_{24}+p_{14}p_{23}=0$. The condition $\dim(XV_2\cap V_2)\geq 1$ is equivalent to the vanishing of the determinant $\det(v_1,v_2,Xv_1,Xv_2)$. It is not hard to write this determinant in terms of $p_{ij}$. For $X=X_1$, the ideal of the image of $\h_w(X)$ in $\mathbb{P}^5$ is
 \[
 I_1=(p_{13}p_{24}-4p_{12}p_{34},p_{14}p_{23}-3p_{12}p_{34}),
 \]
 while for $X=X_2$ it is 
 \[
 I_2= (p_{24}^2-p_{14}p_{34}-p_{23}p_{34}, p_{14}p_{23}- p_{13}p_{24} + p_{12}p_{34}).
 \]

 It is not hard to see that $V(I_1)$ is singular precisely at the 6 points $(1:0:\ldots:0),\ldots, (0:\ldots:0:1)$. These points correspond to the vector spaces $V_2$ such that $X_1V_2=V_2$.  For example at $(0:1:0:0:0:0)$, the affine equation for $V(I_1)$ becomes (in the affine chart $p_{13}=1$)
 \[
 (p_{24}-4p_{12}p_{34}, p_{14}p_{23}-3p_{12}p_{34}).
 \]
 This means that $V(I_1)\cap \{p_{13}=1\}$ is isomorphic to the cone $\spec(\frac{\mathbb{C}[x,y,z,w]}{(xy-zw)})$. \par
  We claim that $\h_w(X_1)$ is $V(I_1)$ blown up at the $6$ singular points. Indeed, the map $\h_w(X_1)\to V(I_1)$ is one-to-one outside the singular points, while if $XV_2\neq V_2$ the only possible choices for $V_1$ and $V_3$ are $V_1=V_2\cap X_1^{-1}V_2$ and $V_3=V_2+X_1^{-1}V_2$. Over the points such that $X_1V_2=V_2$, the fiber of $\h_w(X_1)\to V(I)$ is $\mathbb{P}^1\times\mathbb{P}^1$, since we are free to choose $V_1\subseteq V_2$ and $V_3\supseteq V_2$. The preimages of the singular points are Cartier divisors, and the blow-up of a cone on its singular point has a $\mathbb{P}^1\times\mathbb{P}^1$ as its exceptional divisor, so we have proved our claim.\par
    Upon a rescaling of the variables $p_{ij}$, we could assume that 
    \[
    I=(p_{13}p_{24}-p_{12}p_{34},p_{14}p_{23}-p_{12}p_{34}),
    \]
    and now $V(I)$ carries an action of $S_4$, where $wp_{ij}=p_{w(i)w(j)}$. Since the singular points are invariant points in this action, we have that this action of $S_4$ lifts to an action of $S_4$ on  $\h_w(X_1)$.\par
    The unipotent case $X=X_2$ is more complicated (for instance, it is known that regular nilpotent Hessenberg varieties are not even normal \cite{}). We have that only the point $(1:0:0:0:0:0)$ (which corresponds to $V_2=\langle e_1,e_2\rangle$) corresponds to a point where $X_2V_2=V_2$. However, the map $\h_w(X_2)\to V(I_2)$ is not the blowup of this special point, but the blowup of a more sophisticated subscheme.\par
    
\end{Exa}

\section{Paving by affines the variety $\hbs_{\si}(X)$}
\label{sec:hbspaving}

In this section we will construct an affine paving of $\hbs_\si(X)$ for a regular element $X\in G^r$. The idea is to prove that for each $\bs_{\si,\bi}^w\subset \bs_\si$ constructed in Section \ref{sec:pavingbs}, we have that $\hbs_\si(X)\cap \bs_{\si,\bi}^w$ is either empty or isomorphic to an affine space. In order to do that, we will need a technical result.

  Let $A$ be a $\mathbb{C}$-algebra and consider $A[x_1,\ldots, x_N]$  the polynomial ring in $N$ variables over $A$. Assume that the variable $x_i$ has weight $k_i\in \mathbb{Z}$. We consider a finite set of polynomials $\mathcal{F}=\{F_{k,j}\}$ where $k\in \mathbb{Z}$ and $j\in\{1,\ldots, m_k\}$ (note that $m_k=0$ except for a finite number of values $k$).

\begin{proposition}
\label{prop:affine}
Assume that we can write $F_{k,j}=G_{k,j}+H_{k,j}$ where $G_{k,j}$ is a linear polynomial with complex coefficients in variables of weight $k$ and $H_{k,j}$ is a polynomial in variables of weight less than $k$. Moreover, assume that for each $k$ the polynomials $\{G_{k,j}\}_{j\in\{1,\ldots, m_k\}}$ are linearly independent. In that case, we have that $V(F)=\spec(\frac{A[x_1,\ldots,x_N]}{\langle F\rangle})$ is isomorphic to an affine space over $A$ of expected dimension $N-\sum_{k\in \mathbb{Z}} m_k$ (over $spec(A)$).
\end{proposition}
\begin{proof}
The proof is done by induction. Consider $X_k=\{x_i;\wt(x_i)\leq k\}$ and let $A_k=A[X_k]$ and $\mathcal{F}_k:=\{F_{k',j}\}_{k'\leq k}$. Clearly $\mathcal{F}_k$ is a set of polynomials in $A_k$. For $k$ sufficiently small, we have $X_k=\emptyset$ and $\mathcal{F}_k=\emptyset$, and in this case the result holds trivially. Assume, by the induction hypothesis, that $V(\mathcal{F}_k)$ is isomorphic to an affine space of dimension $|X_k|-\sum_{k'\leq k}m_{k'}$. Since $G_{k+1,1},\ldots, G_{k+1,m_{k+1}}$ are linearly independent, upon change the variables of weight $k$ with suitable linear combinations, we can assume that $G_{k+1,j}=x_j$. Hence, there is an isomorphism $V(\mathcal{F}_k)\times \mathbb{A}^{\alpha_k-m_{k+1}}\to V(F_{k+1})$ given by
 \begin{align*}
   \mathbb{A}^{X_k}\times \mathbb{A}^{\alpha_k-m_{k+1}} &\to \mathbb{A}^{X_k}\times \mathbb{A}^{\alpha_k}\\    
   (y,(x_{m_k+1},\ldots, x_{\alpha_k}))&\mapsto (y, (-H_{k+1,1}(y), -H_{k+1,2}(y),\ldots, -H_{k+1,m_{k+1}}(y), x_{m_k+1},\ldots, x_{\alpha_k})
 \end{align*}
 where $\alpha_k=|X_{k+1}\setminus X_k|$. By the induction hypothesis, we have our result.
\end{proof}

 \par

 Let $\exp\col \mathfrak{g}\to G$ be the exponential map. As we saw in Section \ref{sec:alggroups}, the restriction $\exp:\mathfrak{u}\to U$ is an isomorphism of schemes (see also Equation \eqref{eq:BCH}). Moreover, given $w\in W$, we have $\exp(\mathfrak{u}^w)=U^w$ and $\exp(\mathfrak{u}_w)=U_w$. To fix notation, for each $\gamma\in \Phi^+$ we will choose $E_\gamma$ a generator of $\ul_{\gamma}$. Moreover, for $\beta,\gamma\in \Phi^+$, we define constants $\lambda_{\beta,\gamma}$ by the formula
 \begin{equation}
     \label{eq:lambda}
     [E_\beta,E_\gamma]=\lambda_{\beta, \gamma}E_{\beta+\gamma}.
 \end{equation}
 The equation above only makes sense if $\beta+\gamma$ is a root, otherwise we define $\lambda_{\beta,\gamma}:=0$. It is worth noting that $\lambda_{\beta,\gamma}=-\lambda_{\gamma,\beta}$.
 \par

Let $X\in B$ a regular element. Recall that we can pave $\bs_\si$ by affines isomorphic to $U^w\times U_{\delta_1,b_1}\times \ldots \times U_{\delta_{n},b_n}$ via the map $g_{\si,\bi}^{w}$ defined in Equation \eqref{eq:gLbw}. Also recall Equation \eqref{eq:fiberh}. Then we have that  
\begin{equation}
    \label{eq:u0X}
    (g_{\si,\bi}^w)^{-1}(\hbs_{\si}(X))=\{(u_0,u_1,\ldots, u_n);\dw^{-1}u_0^{-1}Xu_0u_1\ldots u_n\phi_{\si,\bi}(\dw)\in B\}.
\end{equation}
Or, equivalently,
\[
(g_{\si,\bi}^w)^{-1}(\hbs_{\si}(X))=\{(u_0,u_1,\ldots, u_n);u_0^{-1}Xu_0u_1\ldots u_n\phi_{\si,\bi}(\dw)\in \dw B\}.
\]
Since $u_0^{-1}Xu_0u_1\ldots u_n\in B$ for every $(u_0,\ldots, u_n)\in U^w\times U_{\delta_1,b_1}\times\ldots\times U_{\delta_n,b_n}$, we have that if $(g_{\si,\bi}^w)^{-1}(\hbs_{\si}(X))$ is non-empty, then $B\phi_{\si,\bi}(w)B\cap BwB\neq \emptyset$. This implies that $\phi_{\si,\bi}(w)=w$. In this case, we can rephrase Equation \eqref{eq:u0X} as
\begin{equation}
    \label{eq:gLu0}
    (g_{\si,\bi}^w)^{-1}(\hbs_{\si}(X))=\{(u_0,u_1,\ldots, u_n);u_0^{-1}Xu_0u_1\ldots u_n\in B\cap \dw B\dw^{-1}\}.
\end{equation}
For simplicity, we will write $\hbs_{\si,\bi}^w(X):=\hbs_\si(X)\cap \bs_{\si,\bi}^w$.\par

\begin{theorem}
\label{thm:hbspaving}
Let $X\in B$ be a regular element such that $X=X_sX_u$ and $Z_G(X_s)=G_J$ for some $J\subset S$.  Then $\hbs_{\si}(X)$ is paved by the affines $\hbs_{\si,\bi}^{w}(X)$ where $w\in W$ and $\bi \in \Good(L,J,w)$, and $\dim \hbs_{\si,\bi}^{w}(X)=|\bi|$.
\end{theorem}

\begin{corollary}
\label{cor:poinchi}
If $\si=(s_1,\ldots, s_n)$ is a word in $S$, we have the equality
\[
Poin(H^*(\hbs_\si(X_J)))=\chi^{\ind_{W_J}^W}((1+T_{s_1})(1+T_{s_2})\ldots(1+T_{s_n})).
\]
\end{corollary}
\begin{proof}
Since $\hbs_\si(X_J)$ is paved by affines, we have
\begin{align*}
Poin(H^*(\hbs_\si(X_J)))&=\sum_{w\in W}\sum_{\bi\in \Good(L,J,w)}q^{\dim(\hbs_{\si,\bi}^w(X_J))}\\
  &=\sum_{w\in W}\sum_{\bi\in \Good(L,J,w)}q^{|\bi|}\\
  &=\chi^{\ind_{W_J}^W}((1+T_{s_1})(1+T_{s_2})\ldots(1+T_{s_n})).
\end{align*}
The first and second equalities follow from Theorem \ref{thm:hbspaving}, the third equality from Theorem \ref{thm:TrTr}.
\end{proof}

\begin{corollary}
\label{cor:bettipalin}
The Betti numbers of $\hbs_\si(X)$ are palindromic for every regular element $X$.
\end{corollary}
\begin{proof}
This follows from Corollary \ref{cor:poinchi} and the fact that $\chi^{\ind_{W_J}^W}((1+T_{s_1})(1+T_{s_2})\ldots(1+T_{s_n}))$ is palindromic.
\end{proof}

\begin{corollary}
\label{cor:hbschlambda}
Consider $G=SL_n(\mathbb{C})$ and $W=S_n$. Let $\si=(s_1,\ldots, s_n)$ be a word in $S$, then we have the following equality
\[
\ch((1+T_{s_1})(1+T_{s_2})\ldots(1+T_{s_n}))=\sum_{\lambda\vdash n}Poin(H^*(\hbs_\si(X_{\lambda})))m_{\lambda},
\]
where each $X_{\lambda}$ is a regular matrix with Jordan decomposition $\lambda$.
\end{corollary}
\begin{proof}
This follows directly from Corollary \ref{cor:poinchi} and Equation \eqref{eq:cha_cJ}.
\end{proof}

\begin{corollary}
\label{cor:irreducible}
 Let $\si$ be a sequence such that every $s\in S$ appears in $\si$. Then the variety $\hbs_{\si}(X)$ is irreducible for every regular element $X$. In particular, $\h_w(X)$ is irreducible for every $w\in W$.
\end{corollary}
\begin{proof}
  The proof is the same as \cite[Corollary 14]{Precup18}. By proposition \ref{prop:cJmonic} and Theorem \ref{thm:hbspaving} we have that $\hbs_{\si}(X)$ has only one stratum of top dimension. By the Theorem of dimension of fibers applied to $\hbs_{\si}^{r}\to G^r$, the dimension of each irreducible component of $\hbs_{\si}(X)$ is at least $n$, in particular $\hbs_{\si}(X)$ is equidimensional and hence irreducible (since it has only one stratum of top dimension).
\end{proof}

The proof of Theorem \ref{thm:hbspaving} will take the rest of the section. We begin with an example to fix the main ideas.
\begin{Exa}
\label{exa:paving}
 As in Examples \ref{exa:good}, \ref{exa:goodseq}, \ref{exa:Uwomega} and \ref{exa:BSL}, let $G=GL(4)$ then $W=S_4$. If $w=4231$ then $\Tr(w)=\{(1,2),(1,4),(2,4),(3,4)\}$ and $D_\si(w)=\{(1,2),(3,4)\}$. Let $\bi=(1,0,1,1,1,0)$, and $\si=((2,3),(1,2),(3,4),(2,3),(1,2),(3,4))$. We have that the sequence $\underline{\delta}$ (defined in \ref{def:goodseq1}) is given by $\delta_1=(1,2)$, $\delta_2=(1,3)$, $\delta_3=(3,4)$ and $\delta_4=(2,4)$. This means that
\begin{align*}
    u_0&=\left(\begin{array}{cccc}
       1 & x_{12} &0 & x_{14}\\
       0 & 1 &0 &x_{24}\\
       0 & 0 & 1 & x_{34}\\
       0 & 0 & 0 &1
        \end{array}\right), & 
    u_1&=\left(\begin{array}{cccc}
       1 & y_1 &0 &0 \\
       0 & 1 &0 &0\\
       0 & 0 & 1 &0 \\
       0 & 0 & 0 &1
        \end{array}\right), &
    u_2&=\left(\begin{array}{cccc}
       1 & 0 &y_2 & 0\\
       0 & 1 &0 &0\\
       0 & 0 & 1 & 0\\
       0 & 0 & 0 &1
        \end{array}\right),\\
    u_3&=\left(\begin{array}{cccc}
       1 & 0 &0 & 0\\
       0 & 1 &0 &0\\
       0 & 0 & 1 & y_3\\
       0 & 0 & 0 &1
        \end{array}\right),    &
        u_4&=\left(\begin{array}{cccc}
       1 & 0 &0 & 0\\
       0 & 1 &0 &y_4\\
       0 & 0 & 1 &0 \\
       0 & 0 & 0 &1
        \end{array}\right),& X&:=\left(\begin{array}{cccc}
       1 & 1 &0 & 0\\
       0 & 1 &1 & 0\\
       0 & 0 & 1 &1 \\
       0 & 0 & 0 &1
        \end{array}\right).
\end{align*}
Then
\[
u_0^{-1}Xu_0=\left(\begin{array}{cccc}
       1 & 1 &-x_{12} & -x_{12}x_{34}+x_{24}\\
       0 & 1 &1 & x_{34}\\
       0 & 0 & 1 &1 \\
       0 & 0 & 0 &1
        \end{array}\right).
\]
 Moreover, we have that
\begin{align*}
M:=&u_0^{-1}Xu_0u_1u_2u_3u_4\\
=&\left(\begin{array}{cccc}
       1 & y_1+1 &-x_{12}+y_2 & -x_{12}x_{34}-x_{12}y_3+y_2y_3+y_1y_4+x_{24}+y_4\\
       0 & 1 &1 & x_{34}+y_3+y_4\\
       0 & 0 & 1 &y_3+1 \\
       0 & 0 & 0 &1
        \end{array}\right).
\end{align*}

 The condition that $M\in \dw B\dw^{-1}$ says that $M_{12}=0$, $M_{14}=0$, $M_{24}=0$ and $M_{34}=0$ (these are the same entries that are non-zero in $u_0\in U^w$). Therefore, we have that the subscheme
 \[
 (g_{\si,\bi}^w)^{-1}(\hbs_\si(X))\subset U^w\times U_{\delta_1}\times U_{\delta_2}\times U_{\delta_3}\times U_{\delta_4}=\spec(\mathbb{C}[x_{12},x_{14},x_{24},x_{34},y_1,y_2,y_3,y_4])
 \]
 is given by the following equations
 \begin{align*}
     0&=y_1+1,\\
     0&=-x_{12}x_{34}-x_{12}y_3+y_2y_3+y_1y_4+x_{24}+y_4,\\
     0&=x_{34}+y_3+y_4,\\
     0&=y_3+1.
 \end{align*}
It is not hard to see that $(g_{\si,\bi}^w)^{-1}(\hbs_\si(X))\cong \mathbb{A}^4_{\mathbb{C}}$.\par
  We note that if we had omitted $u_1$ in the product defining $M$, then $M_{12}$ would be just $1$, in which case $(g_{\si,\bi}^w)^{-1}(\hbs_\si(X))$ would be empty. The same holds for $\delta_3$. This is the main motivation behind the definition of a good sequence, to exclude that the matrix $M$ could have an entry $M_{i,i+1}=1$ for some $(i,i+1)\in D_\si(w)$.
\end{Exa}

Fix $\si=(s_1,\ldots, s_n)$, $w\in W$ and $\bi$.  We begin, restricting ourselves to the case that $X\in U$ is a regular unipotent element. In this case, Equation \ref{eq:gLu0} becomes
\[
    (g_{\si,\bi}^w)^{-1}(\hbs_{\si}(X))=\{(u_0,u_1,\ldots, u_n);u_0^{-1}Xu_0u_1\ldots u_n\in U_w\},
\]
hence we have the fiber diagram
\begin{equation}
    \label{eq:UUw}
\begin{tikzcd}
(g_{\si,\bi}^w)^{-1}(\hbs_{\si}(X)) \ar[r]\ar[d]& U_w \ar[d]\\
U^w\times U_{\delta_1,b_1}\times\ldots\times U_{\delta_{n},b_n} \ar[r] & U,
\end{tikzcd}
\end{equation}
where the bottom map is given by 
\[
(u_0,u_1,\ldots, u_n)\to u_0^{-1}Xu_0u_1\ldots u_n.
\]
Since $U_{\delta,0}=\{1\}$, we consider only the roots $\delta_i$ such that $b_i=1$. Renaming then, we write $\delta_1,\delta_2,\ldots, \delta_{n'}$ for such roots. Then
\[
U^w\times U_{\delta_1,b_1}\times\ldots\times U_{\delta_{n},b_n}
\]
becomes 
\[
U^w\times U_{\delta_1}\times\ldots\times U_{\delta_{n'}}.
\]
Of course, we have that $n'=|\bi|=\sum_{1\leq i\leq n} b_i$. Alternatively, via BCH formula, the bottom map in Equation \eqref{eq:UUw} can be viewed as the map
\begin{align*}
\mathfrak{u}^w\times \mathfrak{u}_{\delta_1}\times\ldots\times \mathfrak{u}_{\delta_{n'}}&\to \mathfrak{u}\\
(u_0,\ldots, u_{n'})&\mapsto (-u_0)\otimes X\otimes u_0\otimes u_1\otimes\ldots\otimes u_{n'}.
\end{align*}
Via the (inverse of the) exponential map, we can assume that $X\in \ul$, and since $X$ is regular, we can also assume that $X=\sum_{\alpha\in \Delta} E_\alpha$ (up to conjugation). \par  
\begin{Exa}
\label{exa:pavingBCH}
We will repeat Example \ref{exa:paving} using BCH formular. Everything stays the same except that
\begin{align*}
    u_0&=\left(\begin{array}{cccc}
       0 & x_{12} &0 & x_{14}\\
       0 & 0 &0 &x_{24}\\
       0 & 0 & 0 & x_{34}\\
       0 & 0 & 0 &0
        \end{array}\right), & 
    u_1&=\left(\begin{array}{cccc}
       0 & y_1 &0 &0 \\
       0 & 0 &0 &0\\
       0 & 0 & 0 &0 \\
       0 & 0 & 0 &0
        \end{array}\right), &
    u_2&=\left(\begin{array}{cccc}
       0 & 0 &y_2 & 0\\
       0 & 0 &0 &0\\
       0 & 0 & 0 & 0\\
       0 & 0 & 0 &0
        \end{array}\right),\\
    u_3&=\left(\begin{array}{cccc}
       0 & 0 &0 & 0\\
       0 & 0 &0 &0\\
       0 & 0 & 0 & y_3\\
       0 & 0 & 0 &0
        \end{array}\right),    &
        u_4&=\left(\begin{array}{cccc}
       0 & 0 &0 & 0\\
       0 & 0 &0 &y_4\\
       0 & 0 & 0 &0 \\
       0 & 0 & 0 &0
        \end{array}\right),& X&:=\left(\begin{array}{cccc}
       0 & 1 &0 & 0\\
       0 & 0 &1 & 0\\
       0 & 0 & 0 &1 \\
       0 & 0 & 0 &0
        \end{array}\right).
\end{align*}
Then
\[
(-u_0)\otimes X \otimes u_0=\left(\begin{array}{cccc}
       0 & 1 &-x_{12} & -x_{12}x_{34}+x_{24}\\
       0 & 0 &1 & x_{34}\\
       0 & 0 & 0 &1 \\
       0 & 0 & 0 &0
        \end{array}\right),
\]
see Lemma \ref{lem:u0Xu0} below. Moreover, we have that
\begin{align*}
M:=&(-u_0)\otimes X\otimes u_0\otimes u_1\otimes u_2\otimes u_3\otimes u_4=\left(\begin{array}{cccc}
       0 & y_1+1 &-x_{12}-\frac{1}{2}y_1+y_2 & M_{14}\\
       0 & 0 &1 & x_{34}+\frac{1}{2}y_3+y_4\\
       0 & 0 & 0 &y_3+1 \\
       0 & 0 & 0 &0
        \end{array}\right),
\end{align*}
where
\[
M_{14}=-x_{12}x_{34}-\frac{1}{2}x_{34}y_1-\frac{1}{2}x_{12}y_3-\frac{1}{6}y_1y_3+\frac{1}{2}y_2y_3+\frac{1}{2}y_1y_4+x_{24}+\frac{1}{12}y_1-\frac{1}{2}y_2+\frac{1}{12}y_3+\frac{1}{2}y_4
\]
see Corollaries \ref{cor:f} and \ref{cor:ggamma} below.\par
\end{Exa}

In view of Proposition \ref{prop:affine}, it is convenient to give weights to the variables we are considering. Assume that we have variables $x_{\gamma,i}$ for each $\gamma\in \Phi^{+}$ and each integer $i$.  Recall that we have a partial ordering $\leq $ in $\Phi^+$. For $\beta,\gamma\in \Phi^+$, we write $\beta\lessdot \gamma$ if there exists $\alpha\in \Delta$ with $\gamma=\beta+\alpha$ and write $\beta<<\gamma$ if $\beta\leq \gamma$ and $\hta(\beta)\leq \hta(\gamma)-2$.\par
  Given a property $P$ on the set $\Phi^+$ we will write $g(x_{\beta,i},P(\beta))$ to mean a polynomial in the variables $x_{\beta,i}$ such that $\beta\in \Phi^+$ and $\beta$ satisfies $P$. Consider the elements of $\mathbb{C}[x_{\gamma,i}]\otimes_{\mathbb{C}} \ul$ given by
\begin{equation}
    \label{eq:u<gamma}
Y=\sum_{\gamma\in\Phi^+} E_{\gamma}(f_{\gamma}(x_{\gamma,i})+g_{\gamma}(x_{\beta,i},(\beta\lessdot\gamma) \vee (\beta\in\Delta\cap \Tr_{\gamma}))+h_{\gamma}(x_{\beta,i},\beta<<\gamma)),
\end{equation}
where $g$ and $h$ have no constant terms. In particular, if $\gamma\in \Delta$, then $g_\gamma=h_\gamma=0$.
\begin{proposition}
\label{prop:bchy1y2}
If $Y^1$, $Y^2$ and $Y^3$ are elements as above with
\[
Y^i=\sum_{\gamma\in\Phi^+} E_{\gamma}(f^i_{\gamma}(x_{\gamma,i})+g^i_{\gamma}(x_{\beta,i},(\beta\lessdot\gamma) \vee (\beta\in\Delta\cap \Tr_{\gamma}))+h^i_{\gamma}(x_{\beta,i},\beta<<\gamma))
\]
and $Y^3=Y^1\otimes Y^2$, then
\begin{align*}
    f^3_{\gamma}&=f^1_{\gamma}+f^2_{\gamma},\\
    g^3_{\gamma}&=g^1_{\gamma}+g^2_\gamma+\frac{1}{2}\sum_{\alpha\in\Delta}\lambda_{\alpha,\gamma-\alpha}(f^1_\alpha f^2_{\gamma-\alpha}-f^2_\alpha f^1_{\gamma-\alpha})&\textnormal{if $\hta(\gamma)\geq 3$},\\
    g^3_{\gamma}&=g^1_{\gamma}+g^2_\gamma+\frac{1}{2}\sum_{\alpha\in\Delta}\lambda_{\alpha,\gamma-\alpha}(f^1_\alpha f^2_{\gamma-\alpha})&\textnormal{if $\hta(\gamma)= 2$},
\end{align*}
where $\lambda_{\alpha,\gamma-\alpha}$ are the structure constants defined in Equation \eqref{eq:lambda}.
\end{proposition}
\begin{proof}
We have that
\begin{align*}
     Y^1\otimes Y^2 = Y^1+Y^2+\frac{1}{2}[Y_1,Y_2]+[\text{higher commutators}].
\end{align*}
Since all variables in the coefficient of $E_{\gamma}$ are of the form $x_{\beta,i}$, with $\beta\leq \gamma$, the higher commutators term will belong entirely to $h^3_{\gamma}$ and can therefore be ignored. Continuing with the computation, we have
\begin{align}
\label{eq:Y1Y2}
     Y^1+ Y^2&=\sum_{\gamma\in \phi^+}E_{\gamma}(f^1_\gamma+f^2_\gamma+g^1_\gamma+g^2_\gamma+h^1_\gamma+h^2_\gamma).
\end{align}
and 
\begin{align}
\label{eq:[Y1Y2]}
     [Y^1,Y^2]=\sum_{\gamma} E_{\gamma}\bigg(\underset{\beta_1+\beta_2=\gamma}{\sum_{\beta_1,\beta_2\in\phi^+}}\lambda_{\beta_1,\beta_2}(f^1_{\beta_1}+g^1_{\beta_1}+h^1_{\beta_1})(f^2_{\beta_1}+g^2_{\beta_2}+h^2_{\beta_2})\bigg).
\end{align}

If neither $\beta_1$ nor $\beta_2$ is a simple root, then the coefficient 
\[
\lambda_{\beta_1,\beta_2}(f^1_{\beta_1}+g^1_{\beta_1}+h^1_{\beta_1})(f^2_{\beta_2}+g^2_{\beta_2}+h^2_{\beta_2})
\]
also belongs in $h^3_\gamma$, so we can consider only the summands where either $\beta_1\in \Delta$ or $\beta_2\in\Delta$. \par
 Now we are ready to prove the statements in the Proposition. The first follows directly from Equations \eqref{eq:Y1Y2} and \eqref{eq:[Y1Y2]}.  Let us prove the second statement, considering first the case where $\hta(\gamma)\geq 3$. In this case, precisely one of the $\beta_i$ is in $\Delta$, and we assume without loss of generality that $\beta_1\in \Delta$. In particular, we see that $\beta_1<<\gamma$. Moreover $g^1_{\beta_1}=h^1_{\beta_1}=0$, and the coefficient above becomes
 \[
 \lambda_{\beta_1,\beta_2}(f^1_{\beta_1}f^2_{\beta_2}+f^1_{\beta_1}g^2_{\beta_2}+f^1_{\beta_1}h^2_{\beta_2}).
 \]
 It is easy to see that $f^1_{\beta_1}g^2_{\beta_2}+f^1_{\beta_1}h^2_{\beta_2}$ belongs in $h^3_\gamma$ and $f^1_{\beta_1}f^2_{\beta_2}$ belongs in $g^3_{\gamma}$. Since $\lambda_{\beta_1,\beta_2}=-\lambda_{\beta_2,\beta_1}$, the case $\beta_2\in\Delta$ is analogous and we get the second statement.\par

If $\hta(\gamma)=2$, then both $\beta_1,\beta_2\in\Delta$, so  $g^i_{\beta_i}=h^i_{\beta_i}=0$ for $i=1,2$ and the coefficient above becomes
\[
\lambda_{\beta_1,\beta_2}f^1_{\beta_1}f^2_{\beta_2}.
\]
The result follows.
\end{proof}

Each of the spaces $\ul^w$ and $\ul_{\delta_i}$ are affine, and we write $u_0=\sum_{\beta \in \Tr(w)}x_\beta E_\beta$ and $u_i=y_{i}E_{\delta_i}$. As before, we will have a root associated to each variable. From now on, we will write the elements of $\mathbb{C}[x_\beta,y_{i}]\otimes_\mathbb{C}\ul$ as in Equation \ref{eq:u<gamma}, abusing notation, omit the variables and write simply $f_\gamma,g_\gamma, h_\gamma$.
Recall that $[E_\alpha, E_\beta]=\lambda_{\alpha,\beta}E_{\alpha+\beta}$, where $\lambda_{\alpha,\beta}\neq 0$ if and only if $\alpha+\beta\in \Phi^+$. 
\begin{lemma}
\label{lem:u0Xu0}
We can write
\[
(-u_0)\otimes X\otimes (u_0)=\sum_{\gamma\in \Phi^+} E_\gamma(f^0_{\gamma}+g^0_{\gamma}+h^0_{\gamma})
\]
where 
\[
f^0_{\gamma}=\begin{cases}
1 & \text{ if }\gamma\in \Delta\\
0 & \text{otherwise},
\end{cases}
\]
and
\begin{align*}
    g^0_{\gamma}&=\underset{\gamma-\alpha\in\phi_w}{\sum_{\alpha\in\Delta}}\lambda_{\alpha,\gamma-\alpha}x_{\gamma-\alpha}
\end{align*}
for every $\gamma\in\phi^+\setminus\Delta$.
\end{lemma}
\begin{proof}
We could use Proposition \ref{prop:bchy1y2}, but it is more informative to compute directly:
\begin{align*}
    (-u_0)\otimes X\otimes u_0=&(-u_0)\otimes (X+u_0+\frac{1}{2}[X,u_0]+[\text{higher commutators}])\\
    =&(-u_0)+X+u_0+\frac{1}{2}[X,u_0]+[\text{higher commutators}]\\
    &  +\frac{1}{2}\big[-u_0,X+u_0+\frac{1}{2}[X,u_0]+[\text{higher commutators}]\big]+[\text{higher commutators}]\\
    =&X+\frac{1}{2}[X,u_0]+\frac{1}{2}[-u_0,X]+[\text{higher commutators}]\\
    =&X+[X,u_0]+[\text{higher commutators}]\\
    =&\sum_{\alpha\in \Delta}E_\alpha+\sum_{\gamma\notin\Delta}E_{\gamma}\bigg(\underset{\alpha+\beta=\gamma}{\sum_{\alpha\in\Delta,\beta\in\Tr(w)}}\lambda_{\alpha,\beta}x_\beta+h^0_{\gamma}\bigg).
\end{align*}
This finishes the proof.
\end{proof}
\begin{Exa}
\label{exa:ggamma0}
In the notation of Example \ref{exa:pavingBCH}, we have  $g^0_{13}=-x_{12}$, $g_{24}^0=x_{34}$, $g^0_{14}=x_{24}$, and $h_{24}^0=0$, $h_{13}^0=0$, $h_{14}^0=-x_{12}x_{34}$.
\end{Exa}

In the proposition we have defined polynomials $f_\gamma^0$, $g_\gamma^0$, $h_\gamma^0$ in the variables $x_{\beta}$ for $\beta\in \Tr_w$.
Define inductively $f_\gamma^i, g_\gamma^i, h_\gamma^i$ by the formula (recall Equation \eqref{eq:u<gamma}) 
\[
\sum_{\gamma\in \Phi^+} E_\gamma(f_\gamma^{i+1}+g_\gamma^{i+1}+h_\gamma^{i+1})=\bigg(\sum_{\gamma\in \Phi^+} E_\gamma(f_\gamma^i+g_\gamma^i+h_\gamma^i)\bigg)\otimes y_{i+1}E_{\delta_{i+1}}.
\]
Moreover, we set $f_\gamma:=f_{\gamma}^{n'}$, $g_{\gamma}:=g_{\gamma}^{n'}$ and $h_{\gamma}:=h_{\gamma}^{n'}$ for the final polynomials.
The next result follows directly from Proposition \ref{prop:bchy1y2}.

\begin{proposition}
\label{prop:recurf}
We have the following relations.
\begin{enumerate}
    \item If $\gamma\neq \delta_{i+1}$, $\delta_{i+1}\notin \Delta$ and $\gamma-\delta_{i+1}\notin\Delta$, then
    \begin{align*}
    f_{\gamma}^{i+1}&=f_{\gamma}^i,\\
    g_{\gamma}^{i+1}&=g_{\gamma}^i.
    \end{align*}
    \item If $\gamma\neq \delta_{i+1}$, $\delta_{i+1}\in \Delta$ or $\gamma-\delta_{i+1}\in\Delta$, then
    \begin{align*}
    f_{\gamma}^{i+1}&=f_{\gamma}^{i},\\
    g_{\gamma}^{i+1}&=g_{\gamma}^{i}+\frac{1}{2}\lambda_{\gamma-\delta_{i+1},\delta_{i+1}}y_{i+1}f_{\gamma-\delta_{i+1}}^{i}.
    \end{align*}
    \item If $\gamma=\delta_{i+1}$, then
    \begin{align*}
    f_{\gamma}^{i+1}&=f_{\gamma}^{i}+y_{i+1},\\
    g_{\gamma}^{i+1}&=g_{\gamma}^{i}.
    \end{align*}
\end{enumerate}
\end{proposition}
\begin{proof}
Write $y_{i+1}E_{\delta_{i+1}}=\sum_{\gamma\in \Phi^+} E_{\gamma}(\widetilde{f}_\gamma)$. By Proposition \ref{prop:bchy1y2}, we have that
\begin{align*}
    f_{\gamma}^{i+1}=&f_{\gamma}^i+\widetilde{f}_{\gamma}\\
    g_{\gamma}^{i+1}=&g_{\gamma}^i+\frac{1}{2}\sum_{\alpha\in\Delta}\lambda_{\alpha,\gamma-\alpha}(f_{\alpha}^{i}\widetilde{f}_{\gamma-\alpha}-\widetilde{f}_{\alpha}f_{\gamma-\alpha}^{i}).&\textnormal{if $\hta(\gamma)\geq 3$}\\
    g^{i+1}_{\gamma}=&g^i_{\gamma}+\frac{1}{2}\sum_{\alpha\in\Delta}\lambda_{\alpha,\gamma-\alpha}(f^i_\alpha \widetilde{f}_{\gamma-\alpha}).&\textnormal{if $\hta(\gamma)= 2$}
\end{align*}
Since $\widetilde{f}_{\gamma}=0$, unless $\gamma=\delta_{i+1}$, we get our results.
\end{proof}
\begin{corollary}
\label{cor:f}
For every $\alpha\in \Delta$, we have 
\begin{align*}
f_{\alpha}^{i}=&1+\underset{\delta_i=\alpha}{\sum_{j\in\{1,\ldots, i\}}} y_j,
\end{align*}
while for every $\gamma\notin\Delta$,
\begin{align*}
f_{\gamma}^{i}=&\underset{\delta_i=\gamma}{\sum_{j\in\{1,\ldots, i\}}} y_j.
\end{align*}
\end{corollary}
\begin{proof}
This follows directly from Proposition \ref{prop:recurf}.
\end{proof}

\begin{corollary}
\label{cor:ggamma}
It holds that
\begin{equation}
\label{eq:ggamma}    
g_{\gamma}=g_{\gamma}^{n'}=g_{\gamma}^0+\frac{1}{2}\underset{\gamma-\delta_j\in \Delta}{\sum_{1\leq j\leq n'}}\lambda_{\gamma-\delta_j,\delta_j}y_j+\frac{1}{2}\underset{(\delta_i\in\Delta)\vee(\gamma-\delta_i\in\Delta)}{\sum_{1\leq i\leq n'-1}}\lambda_{\delta_i,\gamma-\delta_i}y_i(\underset{\delta_j=\gamma-\delta_i}{\sum_{i+1\leq j\leq n}} y_j)
\end{equation}
\end{corollary}

\begin{proof}
By Proposition \ref{prop:recurf}, we see that
\begin{align*}
g_{\gamma}^{n'}= &g_{\gamma}^{0}+\frac{1}{2}\underset{(\gamma-\delta_j\in\Delta)\vee (\delta_j\in\Delta)}{\sum_{1\leq j\leq n'}} \lambda_{\gamma-\delta_j,\delta_j}y_j f_{\gamma-\delta_j}^{j-1}\\
             =&g_{\gamma}^{0}+\frac{1}{2}\underset{(\gamma-\delta_j\in\Delta)\vee(\delta_j\in\Delta)}{\sum_{1\leq j\leq n'}} \lambda_{\gamma-\delta_j,\delta_j}y_j \big([\gamma-\delta_j\in\Delta]+\underset{\delta_i=\gamma-\delta_j}{\sum_{1\leq i\leq j-1}}y_i\big)\\
             =&g_{\gamma}^0+\frac{1}{2}\underset{\gamma-\delta_j\in \Delta}{\sum_{1\leq j\leq n'}}\lambda_{\gamma-\delta_j,\delta_j}y_j+\frac{1}{2}\underset{(\delta_i\in\Delta)\vee( \delta_j\in\Delta)}{\underset{\delta_i+\delta_j=\gamma}{\sum_{1\leq i<j\leq n'}}}\lambda_{\delta_i,\delta_j}y_jy_i\\
             =&g_{\gamma}^0+\frac{1}{2}\underset{\gamma-\delta_j\in \Delta}{\sum_{1\leq j\leq n'}}\lambda_{\gamma-\delta_j,\delta_j}y_j+\frac{1}{2}\underset{(\delta_i\in\Delta)\vee(\gamma-\delta_i\in\Delta)}{\sum_{1\leq i\leq n'-1}}\lambda_{\delta_i,\gamma-\delta_i}y_i(\underset{\delta_j=\gamma-\delta_i}{\sum_{i+1\leq j\leq n}} y_j)
\end{align*}
\end{proof}
\begin{Exa}
\label{exa:ggamman}
Using the notation of Example \ref{exa:paving} and restricting attention to $\gamma\in \Tr(w)=\{(1,2),(1,4),(2,3),(2,4)\}$, we have
\begin{align*}
    f_{12}&=y_1+1\\
    f_{13}&=y_2\\
    f_{34}&=y_3+1\\
    f_{24}&=y_4\\
    f_{14}&=0\\
    g_{24}&=x_{34}+\frac{1}{2}y_3\\
    g_{13}&=-x_{12}-\frac{1}{2}y_1\\
    g_{14}&=x_{24}-\frac{1}{2}y_2+\frac{1}{2}y_4+\frac{1}{2}y_1y_4+\frac{1}{2}y_2y_3\\
    h_{14}&=x_{12}x_{34}-\frac{1}{2}x_{34}y_1-\frac{1}{2}x_{12}y_3-\frac{1}{6}y_1y_3+\frac{1}{12}y_1+\frac{1}{12}y_3
\end{align*}
in agreement with Corollaries \ref{cor:f} and \ref{cor:ggamma}. Indeed, let us compute $g_{14}$ via Corollary \ref{cor:ggamma}: We have $\gamma=(1,4)$, $\delta_1=(1,2)$, $\delta_2=(1,3)$, $\delta_3=(3,4)$, and $\delta_4=(2,4)$, so $\gamma-\delta_i\in \Delta$ only for $i=2,4$. Moreover, only for the pairs $(i,j)=(1,4),(2,3)$ do we have that $\delta_i+\delta_j=\gamma$ and either $\delta_i\in \Delta$ or $\delta_j\in \Delta$, so
\[
g_{14}=x_{24}+\frac{1}{2}(\lambda_{(3,4),(1,3)}y_2+\lambda_{(1,2),(2,4)}y_4)+\frac{1}{2}(\lambda_{(1,2),(2,4)}y_1y_4+\lambda_{(1,3),(3,4)}y_2y_3).
\]
Since $\lambda_{(i,j),(j,k)}=1$ for $i<j<k$, we have
\[
g_{14}=x_{24}-\frac{1}{2}y_2+\frac{1}{2}y_4+\frac{1}{2}y_1y_4+\frac{1}{2}y_2y_3
\]
as expected.
\end{Exa}

We are interested in the space $(g_{\si,\bi}^w)^{-1}(\hbs_{\si}(X)):=V(f_{\gamma}+g_{\gamma}+h_{\gamma})_{\gamma\in \phi(w)}$. First note that if there exists $\alpha\in\Delta\cap \phi_w$ such that $\delta_j\neq \alpha$ for every $j$, then $\hbs_{\si,X}^{w,\bi}=\emptyset$. This is because
\[
f_\alpha+g_{\alpha}+h_{\alpha}=1
\]
by Corollary \ref{cor:f} and the fact that $g_{\alpha}=h_{\alpha}=0$. In particular, if $\bi$ is not a good word for $w$, then $\hbs_{\si,X}^{w,\bi}=\emptyset$. \par
We now assume that $\bi$ is a good word for $w$. In particular we can apply Proposition \ref{prop:goodseq}, and see that for every $\gamma\in \Tr_\si(w)$, one of the following conditions hold
\begin{enumerate}
    \item There exists $i$ such that $\gamma=\delta_i$.
    \item For all $\alpha\in \Delta$ such that $\gamma-\alpha\in \phi\setminus\Tr(w)$, there exists $i$ such that $\gamma-\alpha=\delta_i$ and $\delta_j\neq\alpha$ for every $j<i$.
\end{enumerate}

For every $\beta\in \{\delta_1,\ldots, \delta_{n'}\}$, let $y_\beta:=y_{k_\beta}$, where $k_{\beta}=\min\{i, \delta_i=\beta\}$. Let us give weights to the variables according to $\wt(x_\beta):=\hta(\beta)+1$ and $\wt(y_\beta):=\hta(\beta)+\frac{1}{2}$ if $\beta\in \Tr(w)$, or $\wt(y_\beta)=\hta(\beta)+1$ otherwise. For every other $y_i$ we define $\wt(y_i):=0$.
\begin{proposition}
\label{prop:gf}
Let $\gamma\in\Tr(w)$ be a root such that $\gamma\neq \delta_i$ for every $i$. Then
\begin{align*}
g_{\gamma}=&\sum_{\alpha\in\Delta,\gamma-\alpha\in\Tr(w)}\lambda_{\alpha,\gamma-\alpha}x_{\gamma-\alpha}+\sum_{\alpha\in\Delta,\gamma-\alpha\notin\Tr(w)}\lambda_{\alpha,\gamma-\alpha}y_{\gamma-\alpha}\frac{1}{2}(2-f_{\alpha})+h_{\gamma}'
\end{align*}
where $h_{\gamma}'$ has only variables of weight less than $\hta(\gamma)$.
\end{proposition}
\begin{proof}
Since $\gamma\neq \delta_i$ for every $i$, condition (2) above holds. Fix $\alpha\in\Delta$ such that $\gamma-\alpha\in\Phi^+$.   If $\gamma-\alpha\notin\Tr(w)$, then $\delta_j\neq\alpha$ for every $j<k_{\gamma-\alpha}$. In particular the variable $y_{\gamma-\alpha}=y_{k_{\gamma-\alpha}}$ only appears in the following summand in Equation \ref{eq:ggamma}
\[
\frac{1}{2}\lambda_{\alpha,\gamma-\alpha}y_{\gamma-\alpha}+\frac{1}{2}\lambda_{\gamma-\alpha,\alpha}y_{\gamma-\alpha}(\underset{\delta_j=\alpha}{\sum_{k_\beta+1\leq j\leq n'}} y_j)
\]

Now we check that every other variable $y_i$ appearing in $g_{\gamma}$ has weight less than $\hta(\gamma)$. If  $\gamma-\alpha\in\Tr(w)$, then all variables $y_i$ with $\delta_i=\gamma-\alpha$ either have weight $0$, $\hta(\gamma)-1$, or $\hta(\gamma)-\frac{1}{2}$. If $\alpha\in \Tr(w)$, then all variables $y_i$ with $\delta_i=\alpha$ have weight either $0$, $1$, or $\frac{3}{2}$, and if $\alpha\notin\Tr(w)$, then either $\alpha=\gamma-\alpha'$ for some $\alpha'\in\Delta$ and we are in the case of the previous paragraph, or $\hta(\gamma)\geq3$ and all variables $y_i$ with $\delta_i=\alpha$ have weight either $0$ or $2$.\par
 We finish by noting that 
 \[
 \underset{\delta_j=\alpha}{\sum_{i+1\leq j\leq n'}}y_j=f_{\alpha}-1,
 \]
 since there is no $j<k_\beta$ with $\delta_j=\alpha$. Since $\lambda_{\alpha,\gamma-\alpha}=-\lambda_{\gamma-\alpha,\alpha}$ the result follows.
\end{proof}
\begin{Exa}
In Example \ref{exa:ggamman} we saw that 
\[
f_{12}=1+y_1\quad,f_{34}=y_3+1\text{ and } g_{14}=x_{24}-\frac{1}{2}y_2+\frac{1}{2}y_4+\frac{1}{2}y_1y_4+\frac{1}{2}y_2y_3.
\]
 This means that
\begin{align*}
g_{14}=&x_{24}+(-1)y_2\frac{1}{2}(2-(y_3+1))+y_4+y_1y_4\\
           =&x_{24}+\lambda_{(3,4),(1,3)}y_2\frac{1}{2}(2-(y_3+1))+(y_4+y_1y_4)
\end{align*}
which agrees with Proposition \ref{prop:gf}, since $(1,4)-(3,4)=(1,3)=\delta_2\notin \Tr(w)$ and $(1,4)-(1,2)=(2,4)\in \Tr(w)$.
\end{Exa}

We are finally in a position to use Proposition \ref{prop:affine} and prove Theorem \ref{thm:hbspaving}. 

\begin{proof}[Proof of Theorem \ref{thm:hbspaving}] Assume first that $X$ is regular unipotent, so that $J=S$.

We have three cases
\begin{enumerate}
    \item If $\alpha\in\Delta\cap \Tr(w)$, then $g_\alpha=h_\alpha=0$ and the coefficient of $E_\alpha$ is
    \[
    f_{\alpha}=1+\sum_{\delta_j=\gamma} y_j.
    \]
    \item If $\gamma\in\Tr(w)\setminus\Delta$ and there exists $j$ such that $\gamma=\delta_j$, then the coefficient of $E_{\gamma}$ is
    \[
    f_{\gamma}+g_{\gamma}+h_{\gamma}=\sum_{\delta_j=\gamma} y_j+g_{\gamma}+h_{\gamma}.
    \]
    \item If $\gamma\in\Tr(w)\setminus\Delta$ and there does not exists $j$ such that $\gamma=\delta_j$, then the coefficient of $E_{\gamma}$ is
    \[
    g_{\gamma}+h_{\gamma}=\sum_{\alpha\in\Delta,\gamma-\alpha\in\Tr(w)}\lambda_{\alpha,\gamma-\alpha}x_{\gamma-\alpha}+\sum_{\alpha\in\Delta,\gamma-\alpha\notin\Tr(w)}\lambda_{\alpha,\gamma-\alpha}y_{\gamma-\alpha}\frac{1}{2}(2-f_{\alpha})+h_{\gamma}'+h_{\gamma}.
    \]
\end{enumerate}
The variety $\hbs_{\si,\bi}^{w}(X)$ is isomorphic, via $g_{\si,\bi}^w$ to $(g_{\si,\bi}^w)^{-1}(\hbs_\si(X))$, the latter being the locus where all the polynomials in items (1), (2), and (3) above vanish. It is straightforward to see that we can change the polynomials in case (3) to
\begin{enumerate}
    \item[(3')] If $\gamma\in\Tr(w)\setminus\Delta$ we consider the polynomial
    \[
    \sum_{\alpha\in\Delta,\gamma-\alpha\in\phi(w)}\lambda_{\alpha,\gamma-\alpha}x_{\gamma-\alpha}+\sum_{\alpha\in\Delta,\gamma-\alpha\notin\phi(w)}\lambda_{\alpha,\gamma-\alpha}y_{\gamma-\alpha}+h_{\gamma}'+h_{\gamma}.
    \]
\end{enumerate}
This is because $\alpha+(\gamma-\alpha)\in \Tr(w)$ and $\gamma-\alpha\notin\Tr(w)$, hence $\alpha\in \Tr(w)$, which will force $f_{\alpha}$ to vanish.\par
The only thing left is to see that the linear polynomials
\[
\sum_{\alpha\in\Delta,\gamma-\alpha\in\phi(w)}\lambda_{\alpha,\gamma-\alpha}x_{\gamma-\alpha}+\sum_{\alpha\in\Delta,\gamma-\alpha\notin\phi(w)}\lambda_{\alpha,\gamma-\alpha}y_{\gamma-\alpha}
\]
for $\gamma\in \Tr(w)\setminus(\Delta\cup\{\delta_1,\ldots,\delta_{n'}\})$ are linearly independent, but this follows from Corollary \ref{cor:LI}.
\end{proof}

Since we need a slightly more general result for the non-unipotent case, we state the following Proposition, which is what we actually prove above.

\begin{proposition}
\label{prop:locusdelta}
Let $X\in U$ be a regular unipotent element, $w\in W$, and $\underline{\delta}=(\delta_1,\ldots, \delta_{n'})$ a good sequence for $w$ of positive roots. Then the locus
\[
U_{\underline{\delta}}^w(X):=\{(u_0,u_1,\ldots, u_{n'})\in U^w\times U_{\delta_1}\times\ldots\times U_{\delta_{n'}}; u_0^{-1}X u_0 u_1\ldots u_{n'}\in U_w\}
\]
is isomorphic to an affine space of dimension $n'$.
\end{proposition}

\begin{proof}[Proof of Theorem \ref{thm:hbspaving}, the general case]

Let $X=X_sX_u$, where $X_s$ is semisimple and $X_u$ is unipotent. We know (by \cite{Stein65}) that $X$ is regular if and only if $X_u$ is regular in $Z_G(X_s)$. Up to conjugation, we can assume that there exists $J\subset \Delta$ such that $Z_G(X_s)=G_J$ and $X_u\in U_j$ is regular in $G_J$. \par
  The main idea is to reduce to the regular unipotent case and use Propostion \ref{prop:affine}.  As before, for each $\bi\in \{0,1\}^n$ and $w\in W$ we define $\hbs_{\si,\bi}^{w}(X)=\bs_{\si,\bi}^w\cap \hbs_\si(X)$. Via the isomorphism $g_{\si,\bi}^w$, we have that $\hbs_{\si,\bi}^w(X)$ is isomorphic to $(g_{\si,\bi}^w)^{-1}(\hbs_\si(X))\subset U^w\times U_{\delta_1}\times\ldots\times U_{\delta_{n'}}$. Moreover, recalling Equation \eqref{eq:gLu0} we have
  \[
  (g_{\si,\bi}^w)^{-1}(\hbs_\si(X))=\{(u_0,u_1,\ldots, u_{n'}); u_0^{-1}Xu_0u_1\ldots u_{n'}\in  \dw B\dw^{-1}\}.
  \]
Since $X_s\in T\subset \dw B\dw^{-1}$, we have that
\[
u_0^{-1}Xu_0u_1\ldots u_m\in \dw B\dw^{-1}\text{ if and only if } X_s^{-1}u_0^{-1}Xu_0u_1\ldots u_m \in \dw B\dw^{-1}.
\]
By Lemma \ref{lem:UJnormal} the projection $p_J\col U\to U_J$ is a group homomorphism. Define 
\[
U_{\delta_i,J}:=U_{\delta_i}\cap U_J=\begin{cases}
  \{1\}&\text{ if }\delta_i\notin \Phi_J\\
  U_{\delta_i}&\text{ if }\delta_i\in \Phi_J,
  \end{cases}
\]
and let $w_J\in W_J$ and ${}^Jw\in {}^JW$ be such that $w=w_J\cdot {}^Jw$. We let $q_J$ be the projection map
\begin{align*}
q_J\col U^w\times U_{\delta_1}\times\ldots\times U_{\delta_{n'}}&\to U^{w_J}\times U_{\delta_1,J}\times\ldots\times U_{\delta_{n'},J}\\
 (u_0,u_1,\ldots, u_{n'})&\mapsto (p_J(u_0),p_J(u_1),\ldots, p_J(u_{n'})).
\end{align*}
Then by Lemma \ref{lem:UJnormal}, we have that $p_J$ is a homomorphism, $p_J(tut^{-1})=tp_J(u)t^{-1}$, and $p_J(u)\in U_J\subset Z_G(X_s)$ for every $u\in U$. Thus
\begin{align*}
p_J(X_s^{-1}u_0^{-1}X_sX_uu_0u_1\ldots u_n) = & p_J(X_s^{-1}u_0^{-1}X_s)X_up_J(u_0)p_J(u_1)\ldots p_J(u_n)\\
                                             =  & X_s^{-1}p_J(u_0)^{-1}X_sX_up_J(u_0)p_J(u_1)\ldots p_J(u_n)\\
                                           = & p_J(u_0)^{-1}X_up_J(u_0)p_J(u_1)\ldots p_J(u_n),
\end{align*}
and in particular we have
\[
q_J((g_{\si,\bi}^w)^{-1}(\hbs_{\si}(X))=U^{w_J}_{\underline{\delta}_J}(X_u),
\]
where $\underline{\delta}_J$ is the sequence of positive roots obtained from $\underline{\delta}$ by removing the entries $\delta_i$ such that $\delta_i\not\in \Phi_J$. By Propositions \ref{prop:locusdelta} and \ref{prop:goodseq} we have that $U^{w_J}_{\underline{\delta}_J}(X_U)$ is isomorphic to an affine space of dimension $|\underline{\delta}_J|$. Moreover,
\[
(g_{\si,\bi}^w)^{-1}(\hbs_\si(X))\subset U_{\underline{\delta}_J}^{w_J}(X_u)\times_{(U^{w_J}\times U_{\delta_1,J}\times\ldots U_{\delta_{n'},J})} (U^w\times U_{\delta_1}\times\ldots \times U_{\delta_{n'},J}),
\]
and the right-hand side is isomorphic to 
\[
U_{\underline{\delta}_J}^{w_J}(X_u)\times \prod_{\gamma\in \Tr(w)\setminus\Phi_J} U_{\gamma}\times \prod_{i, \delta_i\notin\Phi_J} U_{\delta_i}.
\]
Let $A$ be the $\mathbb{C}$-algebra such that $U_{\underline{\delta}_J}^{w_J}(X_u)=\spec(A)$. Then
\[
U_{\underline{\delta}_J}^{w_J}(X_u)\times_{(U^{w_J}\times U_{\delta_1,J}\times\ldots U_{\delta_{n'},J})} (U^w\times U_{\delta_1}\times\ldots \times U_{\delta_{n'},J})=\spec(A[x_\gamma,y_i]_{\gamma\in \Tr(w)\setminus\Phi_J,\delta_i\notin\Phi_J}).
\]
Via the exponential map, we can preform the computation in $(\ul,\otimes)$, since $X_s^{-1}u_0^{-1}X_s\in U$, $X_u\in U$, and $\log(X_s^{-1}u_0^{-1}X_s)=\ad(X_s)(-\log(u_0))$. Therefore we have
 \[
  (g_{\si,\bi}^w)^{-1}(\hbs_\si(X))=\{(u_0,u_1,\ldots, u_{n'}); (\ad(X_s)(-u_0))\otimes X_n\otimes u_0\otimes u_1\otimes\ldots \otimes u_{n'}\in \ul_w.\}
  \]

As before, write $u_{0}=\sum_{\gamma\in \Tr(w)}x_\gamma E_\gamma$ and $u_i=y_iE_{\delta_i}$. However, if $\gamma,\delta_i\in \phi_J$, we consider $x_\gamma$ and $y_i$ as elements of $A$. Then only $x_{\gamma}$ for $\gamma\in\Tr(w)\setminus\Phi_J$ and $y_i$ for $\delta_i\notin\phi_J$ are variables in $A[x_\gamma,y_i]_{\gamma\in \Tr(w)\setminus\Phi_J,\delta_i\notin\Phi_J}$. If $\gamma\notin \Phi_J$ then $\gamma+\beta\notin\Phi_J$ for every $\beta$, so by Proposition \ref{prop:recurf} we have that the coefficient of $E_\gamma$ in 
\[
\ad(X_s)(-u_0))\otimes X_n\otimes u_0\otimes u_1\otimes\ldots \otimes u_{n'}
\]
is constant for $\gamma\in \Phi_J$. Thus we can restrict attention to coefficients of $E_\gamma$ for $\gamma \notin \Phi_J$. As in Equation \ref{eq:u<gamma}, we write some elements of $A[x_\gamma,y_i]_{\gamma\in \Tr(w)\setminus\Phi_J,\delta_i\notin\Phi_J}\otimes_\mathbb{C}\ul$ as
\[
Y=\sum_{\gamma\in\Phi^+\setminus \Phi_J} E_{\gamma}(f_{\gamma}(x_{\gamma,i})+g_{\gamma}(x_{\beta},\beta<\gamma,y_i) ).
\]
By Lemma \ref{lem:adu0X} below  and Proposition \ref{prop:recurf}, we can write the coefficient of each $E_\gamma$ for $\gamma\in\Tr(w)\setminus\phi_J$ in $(\ad(X_s)(-u_0))\otimes X_n\otimes u_0\otimes\ldots \otimes u_n$ as
\[
f_\gamma+g_\gamma=(1-\gamma(X_s))x_\gamma+\sum_{\delta_j=\gamma}y_j+g_\gamma^0
\]
where $g_{\gamma}$ is a polynomial in variables $x_\beta, y_{\beta,i}$ with $\beta<\gamma$. Since $\gamma\notin \Phi_J$, we have $(1-\gamma(X_s))\neq0$, so applying Proposition \ref{prop:affine}, we see that $\hbs_{\si,\bi}^w$ is an affine space over $U_{\underline{\delta}_J}^{w_J}(X_u)$ of the expected dimension, and hence it is an affine space of the expected dimension $n'=|\bi|$.
\end{proof}

\begin{lemma}
\label{lem:adu0X}
We can write
\[
(\ad(X_s)(-u_0))\otimes X_u\otimes (u_0)=\sum_{\gamma\in\Tr(w)\setminus\Phi_J} E_\gamma(f^0_{\gamma}+g^0_{\gamma})+\sum_{\gamma\notin(\Tr(w)\setminus\Phi_J)}E_\gamma(g_\gamma^0)
\]
where $f^0_{\gamma}=(1-\gamma(X_s))x_\gamma$ for every $\gamma\in\Tr(w)\setminus \Phi_J$ and $g_\gamma^0$ is a polynomial in variables $x_\beta$ with $\beta<\gamma$.
\end{lemma}
\begin{proof}
Recall that $\ad(X_s)u_0=\sum_{\beta\in\phi(w)} \beta(X_s)x_\beta E_\beta$. Moreover, we have that if $\gamma\in\Phi_J$ and $\alpha+\beta=\gamma$, then $\alpha,\beta\in \Phi_J$ which means that $\alpha(X_s)=\beta(X_s)=1$.
\begin{align*}
    (\ad(X_s)(-u_0))\otimes X_u\otimes u_0=&(\ad(X_s)(-u_0))\otimes (X_n+u_0+[\text{higher commutators}])\\
    =&(\ad(X_s)(-u_0))+X_u+u_0+[\text{higher commutators}]\\
    =&(I-\ad(X_s))(u_0)+X_u+[\text{higher commutators}]\\
    =&\sum_{\gamma\in\Tr(w)\setminus\Phi_J}E_\gamma((1-\gamma(X_s))x_\gamma+g_\gamma^0) +\sum_{\gamma\notin (\Tr(w)\setminus\phi_J) }E_\gamma(g_\gamma^0)
\end{align*}
\end{proof}

\section{The monodromy action on $\h_w$ and $\hbs_\si$} 

\label{sec:fibrations}

  The goal of this section is twofold. First, we use using Ehresmann's lemma and Thom's first isotopy lemma to prove that both $\h_w^{rs}\to G^{rs}$ and $\hbs_\si^{rs}\to G^{rs}$ are topological fibrations. Second, we study the monodromy action of $\h_w$ and $\hbs_{\si}$. We begin with the following proposition.

\begin{proposition}
\label{prop:hGsubmersion}
The map $\h_w^{rs,\circ}\to G^{rs}$ is a submersion.
\end{proposition}
\begin{proof}
First, consider $\widetilde{\h}_w^{\circ}=\{(X,g);g^{-1}Xg\in BwB\}\subset G\times G$. We have a map
\begin{align*}
\widetilde{\h}_w^\circ&\to \h_w^\circ\\
(X,g) &\mapsto (X,XgB, gB),
\end{align*}
which is smooth with fibers isomorphic to $B$. Since we are interested in the projection onto the first factor, it is enough to prove that 
\begin{equation}
\label{eq:THsur}
    \tang_{(X_0,g_0)}\widetilde{\h}_w^\circ\to \tang_{X_0}G
\end{equation} 
is surjective for every pair $(X_0,g_0)\in G^{rs}\times G$. Since $\widetilde{\h}_w^\circ\to G$ is $G$-equivariant, we can assume that $g_0=e$. Consider the composition
\[
\begin{tikzcd}[row sep=0cm]
G\times G\ar[r]&\ar[r] G\times G&\ar[r] G& G\\
(g',g)\ar[r, maps to]& (X_0g',g)\ar[r,maps to]&\ar[r,maps to]g^{-1}X_0g'g& X_0^{-1}g^{-1}X_0g'g.
\end{tikzcd}
\]
It is clear that $(e,e)$ is taken to $(e,e)$ and that the differential map is given by
\begin{align*}
    \g\times \g &\to \g\times \g\\
    (Y,Z) &\mapsto \ad(X_0)^{-1}(-Z)+Y+Z.
\end{align*}
Hence
\[
\tang_{(X_0,1)}\widetilde{H}_w=\{(Y,Z); \ad(X_0)^{-1}(-Z)+Y+Z\in \tang_{X_0}BwB\}.
\]
Thus the surjectivity of map \eqref{eq:THsur} is equivalent to the fact that for every $Y\in \g$, there exists $Z\in \g$ such that $\ad(X_0)^{-1}(-Z)+Y+Z\in \tang_{X_0}BwB$. In turn, this is equivalent to the surjectivity of the map
\begin{align*}
\g&\to \frac{\g}{\tang_{X_0}BwB}\\
Z&\mapsto (1-\ad(X_0^{-1}))(Z)+\tang_{x_0}BwB,
\end{align*}
which we will now establish.\par

Since $X_0$ is regular semisimple, it lies inside a maximal torus $T'\subset G$ (with associated Cartan subalgebra $\hl'$). If $\Phi'$ is the set of roots associated to $T'$, then $\g=\hl'\oplus \bigoplus_{\gamma\in \Phi'} \g_\gamma$, and if $E_\gamma$ is a generator of $\g_\gamma$, and since $X_0$ is regular, we have $(1-\ad(X_0^{-1}))(E_\gamma)=(1-\gamma(X_0^{-1}))E_\gamma\neq 0$ (see \cite[Proposition 2.11]{Stein65}). Hence, the image of the map 
\begin{align*}
f\col\g&\to \g\\
Z&\mapsto (1-\ad(X_0)^{-1}))(Z)
\end{align*}
is $\g_{\hl'}:=\bigoplus_{\gamma\in \Phi'} \g_\gamma$. \par
Now, each map below is surjective
\[
\g\xrightarrow{f}\g_{\hl'}\to \frac{\g}{\bl}\to \frac{\g}{\tang_{x_0}BwB},
\]
the second by Lemma \ref{lem:bh'} and the third by Lemma \ref{lem:TXB}. This finishes the proof.
\end{proof}

\begin{proposition}
\label{prop:fibrationH}
The map $\h_w^{rs}\to G^{rs}$ is a topological fibration.
\end{proposition}
\begin{proof}
For $z\leq w$, $\h_{z}^{rs,\circ}$  is a Whitney stratification of $\h_w^{rs}$. Indeed for $z\leq w$, $\Omega_z$ is a Whitney stratification of $\Omega_w$, and since $\h_w\to \Omega_w$ is a smooth morphism, we have that $\h_z^{\circ}$ is a Whitney stratification of $\h_w$. Since $\h_w^{rs}$ is an open subset of $\h_w$, $\h_{z}^{rs,\circ}$ is a Whitney stratification of $\h_w^{rs}$. By Proposition \ref{prop:hGsubmersion}, the restriction $\h_u^{rs,\circ}\to G$ is a submersion for every $u$, so the result follows by Thom's first isotopy lemma.
\end{proof}

\begin{proposition}
\label{prop:fibrationHBS}
Let $w\in W$, and let $\si$ be a reduced word for $w$. Then the map $\hbs_{\si}^{rs}\to G^{rs}$ is a topological fibration.
\end{proposition}
\begin{proof}
Since $\hbs_\si$ is smooth, it is sufficient, by Ehresmann's lemma, to prove that $\hbs^{rs}_\si\to G^{rs}$ is a submersion. Take $Q$ a closed point in $\hbs_\si^{rs}$ and denote by $\beta_\si\col \hbs_\si\to \h_w$ the natural map. Then $Q$ belongs in some strata $\beta_\si^{-1}(\h_u^\circ)$. Moreover, $\beta_\si^{-1}(\h_u^\circ)\to \h_u^\circ$ is a topological fibration since $\alpha_\si^{-1}(\Omega_u)\to \Omega_u$ is a topological fibration, hence $\tang_Q\beta_\si^{-1}(\h_u^\circ)\to \tang_{\beta_\si(Q)}\h_u^\circ$ is surjective. However, we already know that $\tang_{\beta_\si(Q)}\h_u^\circ\to \tang_{f(\beta_\si(Q))}G$ is surjective, so the map $\tang_Q\beta_\si^{-1}(\h_u^\circ)\to \tang_{f(g(Q))}G$ is surjective, which proves that $\tang_Q\hbs_\si\to \tang_{f(\beta_\si(Q))}G$ is surjective.
\[
\begin{tikzcd}
& \hbs_\si \ar[rr] \ar[ld] \ar[ldd]& &\bs_\si \ar[ld] \ar[ldd]\\
\h_w \ar[d] \ar[rr]&&\Omega_w\ar[d] &\\
\hb \ar[rr]&& \flag\times\flag\\
\end{tikzcd}
\]
\end{proof}

The remainder of this section is devoted to proving the following result.

\begin{proposition}
\label{prop:pi1actW}
Let $X\in T^r$ be a regular element of $T$, $w$ an element of $W$, and $\si=(s_1,\ldots, s_{\ell(w)})$ a reduced word for $w$. Then
\begin{enumerate}
    \item The action of $\pi_1(G^{rs},X)$ on $H^*(\hbs_{\si}(X))$ factors through an action of $W$ on $H^*(\hbs_\si(X))$.
    \item The action of $\pi_1(G^{rs},X)$ on $IH^*(\h_u(X))$ factors through an action of $W$ on $IH^*(\h_u(X))$ for every $u\in W$.
    \item  The actions of $W$ are compatible with the isomorphism
\[
H^*(\hbs_\si(X))\cong IH^*(\h_w(X))\oplus \bigoplus IH^*(\h_u(X))\otimes L_u[-\ell(w)+\ell(u)]
\]
given by the decomposition theorem, where $W$ acts on $L_u$ trivially.
\end{enumerate}
\end{proposition}

As in \cite{BrosnanChow}, the idea is to prove the results for the equivariant cohomology and then restrict to the usual cohomology. We begin with a few definitions.\par

Let $\si=(s_1,\ldots, s_{\ell(w)})$ be a reduced word of an element $w\in W$. Recall the torus fiber bundle $\T$ defined in Section \ref{sec:torusbundle}.  There is an action of $\T$ on $G^{rs}\times \flag$, given by
 \begin{align*}
 m\col \T\times_{G^{rs}} (G^{rs}\times \flag)&\to G^{rs}\times\flag\\
  ((X,g'), (X,gB)) & \to (X, g'gB).
 \end{align*}
Since $g'\in Z_G(X)$ we have that $m$ preserves each stratum $\h_w^{rs,\circ}\subset G^{rs}\times \flag$. Hence $m$ will also be considered as an action of $\T$ on $\h_w^{rs}$ for every $w\in W$.\par

Since $\bs_\si=\flag\times_{\mathcal{P}_{s_1}}\flag\times\ldots\times \times_{\mathcal{P}_{s_n}}\flag$, we have an action of $\T$ in $G^{rs}\times \bs_\si$ given by the action of $m$ on each of the factors of $\bs_\si$, that is, 
\begin{align*}
m' \col \T\times_{G^{rs}} (G^{rs}\times \bs_\si)&\to G^{rs}\times \bs_\si\\
((X,g'), (X,g_0B,g_1B,\ldots, g_nB))&\mapsto (X,(g'g_0B,g'g_1B,\ldots, g'g_nB)).
\end{align*}
Again, since $g'\in Z_G(X)$, we have that this action stabilizes $\hbs_\si$.  Clearly the map $\hbs_\si^{rs}\to \h_w^{rs}$ is $\T$-equivariant.\par

  We now turn our attention to the fixed locus of $m$.  A point $(X,gB)\in G^{rs}\times\flag$ is a fixed point of $m$ if and only if $gB=g'gB$ for every $g'\in Z_G(X)$. The latter occurs if and only if $g^{-1}Z_G(X)g\subset B$, which is the same as $Z_G(g^{-1}Xg)\subset B$. Since $X$ is regular semisimple, then so is $g^{-1}Xg$. In particular, we have that $Z_G(g^{-1}Xg)\subset B$ if and only $g^{-1}Xg\in B$. This proves that the fixed locus of the action $m$ is precisely $\h_{e}^{rs}$ (defined in Definition \ref{def:twisted}, Section \ref{sec:twisted} and Equation \eqref{eq:h1}).  \par

  The fixed locus of $m'$ is little more involved. Since the action $m'$ is simply the action $m$ on the factors of $\bs_\si$, we have that fixed locus of $m'$ is
    \begin{equation}
        \label{eq:GrsBSLT}
        (G^{rs}\times \bs_{\si})^\T=\h_{e}^{rs}\times_{G^{rs}\times P_{s_1}}\h_{e}^{rs}\times_{G^{rs}\times P_{s_2}}\ldots\times_{G^{rs}\times P_{s_n}} \h_{e}^{rs}.
    \end{equation}

   It is easier to see this fixed locus on a fiber over $X\in T^{r}$. We know that $\h_{e}^{rs}(X)=\{\dw B; w \in W\}$, in particular, we have 
    \[
    \bs_\si^{\T_X}=\{(\dw_0B, \dw_1B,\ldots, \dw_nB); w_i^{-1}w_{i+1}\in <s_i>\text{ for every }i=0,\ldots, n-1\}.
    \]

This gives a bijection between the set of fixed points $\bs_{\si}^{\T_X}$ and the set $W\times \{0,1\}^{\ell(w)}$. Explicitly, this bijection is given by
\begin{align*}
    W\times\{0,1\}^{\ell(w)}&\to \bs_{\si}^{\T_X}\\
    (w_0,\bi) &\mapsto (\dw_0 B, \dw_1B, \ldots, \dw_{\ell(w)} B),
\end{align*}
where $w_i=\phi_{\si|_{i},\bi|_{i}}(w_0)$. We will abuse notation and say that a pair $(w_0,\bi)$ is a fixed point in  $\bs_\si^{\T_X}$. Moreover, we have that a fixed point $(w_0,\bi)\in \bs_{\si}^{\T_X}$ is in $\hbs_\si(X)$ if and only if $\phi_{\si,\bi}(w_0)=w_0$. Meanwhile, the Weyl group $W$ acts on $\bs_\si^{\T_X}$ by 
  \begin{align*}
  W\times \bs_\si^{\T_X}&\to \bs_\si^{\T_X}\\
  (w',(\dw_0B,\ldots, \dw_{\ell(w)}B))&\mapsto (\dw'\dw_0B,\dw'\dw_1B,\ldots, \dw'\dw_{\ell(w)}B)
  \end{align*}
  which can be restricted to $\hbs_\si(X)^{\T_X}$. Since $\T\to G^{rs}$ and $\hbs_\si^{rs}\to G^{rs}$ are fibrations, we have that $\pi_1(G^{rs},X)$ acts on  $H_{T}^*(\hbs_\si(X))$ and $H_T^*(\hbs_\si(X)^{\T_X})$. By localization we obtain an injective map
  \begin{equation}
      \label{eq:hbsinje}
      H_T^*(\hbs_\si(X))\hookrightarrow H_{T}^*(\hbs_\si(X)^{\T_X})
  \end{equation}
  that is $\pi_1(G^{rs},X)$-equivariant.
  
  \begin{lemma}
  \label{lem:pi1hb}
  The action of $\pi_1(G^{rs},X)$ on $H_T^*(\hbs_\si(X)^{\T_X})$ factors through an action of $W$ on $H_T^*(\hbs_\si(X)^{\T_X})$.
  \end{lemma}
  \begin{proof}
  We have that $\hbs_\si(X)^{\T_X}$ is a union of points, and the the action of $\pi_1(G^{rs},X)$ on these points factors through $W$. This follows from Equation \eqref{eq:GrsBSLT} and the fact that $\h_e^{rs}\to G^{rs}$ is a $W$-Galois cover.\par
   On the other hand, we have that the action of $\pi_1(G^{rs},X)$ on the Lie algebra $\hl$ of $T=\T_X$ factors through $W$, which is the content of Corollary \ref{cor:pi1T}. Since 
   \[
   H_T^*(\hbs_\si(X)^{\T_X})=\bigoplus_{pt\in\hbs_\si(X)^{\T_X}} H_T^*(pt)= \bigoplus_{pt\in\hbs_\si(X)^{\T_X}} \mathbb{C}[\hl]
   \]
  we have our result.
  \end{proof}
  
  \begin{corollary}
  \label{cor:pi1actW}
  The action of $\pi_1(G^{rs},X)$ on $H_T^*(\hbs_\si(X))$ factors through an action of $W$ on $H_T^*(\hbs_\si(X))$.
  \end{corollary}
  \begin{proof}
  This is a direct consequence of Lemma \ref{lem:pi1hb} and Equation \eqref{eq:hbsinje}.
  \end{proof}
  
  \begin{proof}[Proof of Proposition \ref{prop:pi1actW}]
  We begin by proving item (1). The result follows from  Corollary \ref{cor:pi1actW} and the fact that $H^*(\hbs_\si(X))$ is a quotient of $H_T^*(\hbs_\si(X))$. Since the fiber diagram
     \[
     \begin{tikzcd}
     \hbs_\si^{rs}\ar[r]\ar[d, "{\beta_\si}"] & \bs_\si\ar[d, "{\alpha_{\si}}"]\\
     \h_w^{rs}\ar[r]   & \Omega_w
     \end{tikzcd}     
     \]
     is a smooth base change, we have that the decomposition theorem for $\hbs_\si^{rs}\to \h_w^{rs}$ reads
     \[
     (\beta_\si)_*(IC_{\hbs_\si(X)})=IC_{\h_w}\oplus\bigoplus_{z\leq w} IC_{\h_z}\otimes L_z,
     \]
     where  $L_u$ are the same graded vector spaces appearing in Equation \eqref{eq:Schubertdecomp}.\par
     Recalling that $IH^i(\mathcal{X})=\mathcal{H}^{i-\dim_\mathbb{C}(\mathcal{X})}(IC_{\mathcal{X}})$,  $\dim(\hbs_\si)=\dim(G)+ \ell(w)$, and $\dim(\h_z)=\dim(G)+\ell(z)$, we have that (see also Proposition \ref{prop:fibrations})
     \begin{equation}
         \label{eq:HTw}
         H^*(\hbs_\si(X))\cong IH^*(\h_w(X))\oplus \bigoplus_{z\leq w}  IH^*(\h_z(X))\otimes L_z[-\ell(w)+\ell(z)].
     \end{equation}
   
     By Propositions \ref{prop:fibrations},   \ref{prop:fibrationH} and \ref{prop:fibrationHBS}, the isomorphism in Equation \eqref{eq:HTw} is compatible with the action of $\pi_1(G^{rs},X)$. Since the action of $\pi_1(G^{rs},X)$ on the left-hand side factors through $W$, we have that the action on the right-hand side will also factor through $W$. This proves items (2) and (3).
  \end{proof}

\section{Proof of the main theorems}
\label{sec:proofs}

 Now we have all the ingredients to prove  Theorems \ref{thm:mainsimply} and \ref{thm:hbsall}. 

\begin{proof}[Proof of Theorem \ref{thm:hbsall}]
Items (1) and (2) follow from Theorem \ref{thm:hbspaving} and Corollary \ref{cor:poinchi}.\par
Item (3) follows from items (1) and (2), Theorems \ref{thm:invmap} and \ref{thm:invmappalin}, and Propositions \ref{prop:conjugate} and  \ref{prop:pi1actW}. 

  Now we prove item (3). By Proposition \ref{prop:conjugate}, Theorems \ref{thm:invmap} and \ref{thm:invmappalin} and Corollary \ref{cor:bettipalin} we have that
 \[
 H^*(\hbs_\si(X))^{W_J}\cong H^*(\hbs_\si(X_J)),
 \]
 for some $X_J=X_sX_u$ with $Z_G(X_s)=G_J$. In particular
 \begin{equation}
     \label{eq:poinXJ}
  \Poin(H^*(\hbs_\si(X))^{W_J})=\Poin(H^*(\hbs_\si(X_J))).
 \end{equation}
  By Corollary \ref{cor:poinchi},
 \begin{equation}
 \label{eq:poinchi}
 \Poin(H^*(\hbs_\si(X_J)))=\chi^{\ind_{W_J}^W}(\prod_{i=1}^{n}(1+T_{s_i}))),
 \end{equation}
proving the first equality in item (3). To finish the proof, we just notice that, in the case of $G=SL_n(\mathbb{C})$ and $W=S_n$, we have
\begin{align*}
    \ch(\prod_{i=1}^{n}(1+T_{s_i})&=\sum_{\lambda\vdash n} \chi^{\ind_{S_{\lambda}}^{S_n}}(\prod_{i=1}^{n}(1+T_{s_i}))) m_{\lambda},\\
     \ch(H^*(\hbs_\si(X))^{S_{\lambda}})&=\sum_{\lambda\vdash n}\Poin(H^*(\hbs_\si(X))^{S_{\lambda}}) m_{\lambda}
\end{align*}
\end{proof}

\begin{proof}[Proof of Theorem \ref{thm:mainsimply}]

 Let $\si$ be a reduced word for $w$. The first statement of the theorem follows directly from Proposition \ref{prop:pi1actW} item (2). By Proposition \ref{prop:pi1actW} item (3), we have that 
 \[
 H^*(\hbs_\si(X))^{W_J}\cong IH^*(\h_w(X))^{W_J}\oplus\bigoplus_{z< w}IH^*(\h_z(X))^{W_J}\otimes L_z[-\ell(w)+\ell(u)].
 \]
In particular,
\begin{equation}
    \label{eq:poinhb}
 \Poin(H^*(\hbs_\si(X))^{W_J})=\Poin( IH^*(\h_w(X))^{W_J})+\sum_{z\leq w}\Poin(IH^*(\h_z(X))^{W_J})q^{\frac{\ell(w)-\ell(z)}{2}}P_{\si,z}(q^{\frac{1}{2}}),
 \end{equation}
 where $P_{\si,z}(q^{\frac{1}{2}})$ are the polynomials in Proposition \ref{prop:springer}. By Proposition \ref{prop:springer},
 \begin{align}
 \prod_{i=1}^{\ell(w)}(1+T_{s_i})=&q^{\frac{\ell(w)}{2}}C'_w+\sum_{z<w}q^{\frac{\ell(w)}{2}}P_{\si,z}(q^{\frac{1}{2}})C'_z.\nonumber\\
   =&q^{\frac{\ell(w)}{2}}C'_w+\sum_{z<w}q^{\frac{\ell(w)-\ell(z)}{2}}P_{\si,z}(q^{\frac{1}{2}})q^{\frac{\ell(z)}{2}}C'_z. \label{eq:prod1+T}
 \end{align}
 Combining Equations \eqref{eq:poinhb}, \eqref{eq:poinXJ}, \eqref{eq:poinchi}, and \eqref{eq:prod1+T}, with an inductive argument yields
 \[
 \Poin( IH^*(\h_w(X))^{W_J})=\chi^{\ind_{W_J}^W}(q^{\frac{\ell(w)}{2}}C'_w),
 \]
 as desired. The last statement of the theorem follows from the observation that
  \begin{align*}
  \ch(IH^*(\h_w(X)))&=\sum_{\lambda\vdash n}\Poin( IH^*(\h_w(X))^{S_\lambda})m_{\lambda},\\
  \ch(q^{\frac{\ell(w)}{2}}C'_w)&=\sum_{\lambda\vdash n}\chi^{\ind_{S_\lambda}^{S_n}}(q^{\frac{\ell(w)}{2}}C'_w)m_{\lambda},
  \end{align*}
  which finishes the proof of Theorem \ref{thm:mainsimply}.
\end{proof}

\section{Further directions}
\label{sec:questions}
\subsection{Rationally smoothness}
One class of varieties that has palindromic Betti numbers and whose intersection cohomology is isomorphic to its singular cohomology are \emph{rationally smooth varieties}, that is, varieties $\X$ such that $IC_\X=\mathbb{\mathbb{C}}_\X[\dim(\X)]$.   In \cite{Precup18}, Precup asks if regular Hessenberg varieties are rationally smooth. We conjecture the same should hold for the varieties $\hbs_{\si}(X)$.

\begin{conjecture}
\label{conj:bott}
  When $X$ is regular the varieties $\hbs_{\si}(X)$ are rationally smooth.
\end{conjecture}
Notice that if $X$ is regular semisimple, then the varieties $\hbs_{\si}(X)$ are smooth and hence rationally smooth.

We note that $\h_w(X_{\lambda})$ is not rationally smooth in general, since already the Schubert variety $\Omega_w$ is not rationally smooth. Nevertheless it seems reasonable to expect that the intersection cohomology sheaf of $\h_w(X_{\lambda})$ is not so different from the intersection cohomology sheaf of the Schubert variety $\Omega_w$. More precisely, we conjecture that:

\begin{conjecture}
\label{conj:IC}
  For $X\in G$ a regular element and $w\in W$, we have that $IC_{\h_w(X)}=i^!(IC_{\Omega_w})$, where $i\col \h_w(X)\to \Omega_w$ is the natural map $i(gB)=(XgB,gB)$. 
\end{conjecture}

\subsection{Cohomology of Hessenberg varieties} Since the varieties $\h_w(X)$ are generalizations of Hessenberg varieties, we believe that  these varieties may help to understand the cohomology of Hessenberg varieties (see \cite{HaradaPrecup}, \cite{HHMPT} and \cite{HPT21} for results about the cohomology ring of Hessenberg varieties). More precisely,  since $\h_w(X)$ is a  subvariety of $\h_\m(X)$ for every $w\leq \m$, we have natural geometric classes $[\h_w(X)]\in H^*(\h_{\m(X)})$. In the case that $X$ is unipotent, we have that the number of codimension $i$ subvarieties that appear in this way agrees with the known dimension of $H^{2i}(\h_{\m(X)})$. We therefore conjecture:

\begin{conjecture}
Let $w\in S_n$ be a smooth permutation and $X\in SL(n)$ a regular element. Then the classes $[\h_u(X)]\in H^*(\h_w(X))$ (for $u\leq w$) are linearly independent. If $X$ is semisimple, then these classes are fixed by the action of $S_n$. If $X$ is regular unipotent, then the $[\h_u(X)]$ will form a basis of $H^*(\h_w(X))$.
\end{conjecture}

When $X$ is regular semisimple and $w\in S_n$ is a smooth permutation, then $\h_w(X)$ is a GKM-space and the equivariant cohomology ring $H_T^*(\h_w(X))$ has an explicit description in terms of the moment graph of $\h_w(X)$ (see \cite{Tym08}).

\begin{Question}
For each $u\in S_n, u\leq w$, compute the class $[\h_u(X)]$ in terms of the moment graph of $\h_w(X)$.
\end{Question}

\subsection{Relation with $\LLT$-polynomials}

  Given a Hessenberg function $\m$, besides the chromatic quasisymmetric function $\csf_q(G_{\m})$ there are other interesting symmetric functions associated to it. For example, the unicellular $\LLT$-polynomial\footnote{The $\LLT$-polynomials are defined in a more general setting, see \cite{LLT}. The definition we are giving was introduced in \cite{CarlssonMellit} and \cite{AlexPanova}.} associated to $\m$ is defined as
  \begin{equation}
  \label{eq:LLTm}
  \LLT(\m)=\sum_{\kappa} q^{\asc(\kappa)}x_{\kappa},
  \end{equation}
  where the sum runs through all vertex colorings of $G_\m$ (not necessarily proper). The chromatic quasisymmetric function and the $\LLT$-polynomial are related via a plethystic relation (see \cite[Section 5.1]{HHL}, \cite[Proposition 3.5]{CarlssonMellit} and \cite[Lemma 12]{Alexandersson_2020}):
  \[
  \csf_q(\m))=\frac{LLT(\m)[\x(q-1)]}{(q-1)^n}.
  \]
   Or, equivalently
   \[
   \LLT(\m)=(q-1)^n\csf_q(\m)\Big[\frac{\x}{q-1}\Big].
   \]
   We could ask:

\begin{Question}
 We can define $\LLT$-like polynomials for non-codominant permutations $w$ via
 \[
 \LLT(w):=(q-1)^n(\omega(\ch(q^{\frac{\ell(w)}{2}}C'_w)))\Big[\frac{\x}{q-1}\Big].
 \]
 Do these symmetric function have an interesting combinatorial definition?
 \end{Question}
 
 For example, we have that
 \[
 \LLT(3412)=(q^4+q^3)s_{1,1,1,1}+(3q^3+3q^2)s_{2,1,1}+(q^3+2q^2+q)s_{2,2}+(3q^2+3q)s_{3,1}+(q+1)s_4.
 \]
 Since for codominant permutations $w$ we have that $\LLT(w)$ is Schur-positive (see \cite{HaimanGronj} and \cite{GP}) and shifted $e$-positive (see \cite{dadderio}, \cite{AS2020} and \cite{ANtree}), the same properties would be enjoyed by $\LLT(w)$ for every permutations $w\in S_n$, if (some modification of) Conjecture \ref{conj:haimansingular} was true.
 \begin{conjecture}
 For every $w\in S_n$ we have that $\LLT(w)$ is Schur-positive and shifted $e$-positive.
 \end{conjecture}

In \cite{ANtree}, for each Hessenberg function $\m$, a weight function $\wt_{\m}\col\{u\in S_n;u\leq \m\}\to \mathbb{Z}_{\geq0}$ was defined such that
\[
\LLT(\m)=\sum_{u \leq \m}(q-1)^{n-\ell(\lambda(u))}q^{\wt_{\m}(u)}e_{\lambda(u)}.
\]
\begin{Question}
For each permutation $w\in S_n$, we may ask if there is a function $\wt_w\col \{ u \in S_{n} ; u \leq w \}\to \mathbb{Z}_{\geq0}$ such that
\begin{equation}
    \label{eq:LLTw}
\LLT(w)=\sum_{u\leq w}(q-1)^{n-\ell(\lambda(u))}q^{\wt_{w}(u)}P_{u,w}(q)e_{\lambda(u)},
\end{equation}
where $P_{u,w}$ is the Kazhdan-Lusztig polnomial.\par

The above equation is equivalent to 
\[
  \ch(q^{\frac{\ell(w)}{2}}C'_w)=\sum_{u\leq w}q^{\wt_w(u)}P_{u,w}(q)\rho_{\lambda(u)},
\]
where $\rho_{\lambda}$ is defined by $\rho_{i}:=q^{j-1}P_i(q^{-1})$, $\rho_{\lambda}:=\prod \rho_{\lambda_i}$ and $P_j$ is the Hall-Littlewood polynomial (see \cite[Proposition 2.1]{ANtree}).
\end{Question}

\subsection{Hopf algebra of Dyck paths}In \cite{GP} a proof of the Shareshian-Wachs conjecture was given by means of a Hopf algebra of Dyck paths. Since each Dyck path corresponds to a codominant permutation, we ask:

\begin{Question}
Does there exist a Hopf algebra $\mathcal{S}$ of permutations extending the Hopf algebra of Dyck paths considered in \cite{GP}? Does there exist a multiplicative map $\zeta\col S\to \mathbb{C}(q)$ inducing a Hopf algebra morphism $\Psi_\zeta\col \mathcal{S}\to \mathcal{Q}Sym$ such that $\ch(q^{\frac{\ell(w)}{2}}C'_w)=\Psi_\zeta(w)$ (or for some other power of $q$). Can we give a proof of Theorem \ref{thm:mainsimply} using this Hopf algebra as in \cite{GP}? Are the $\LLT$-polynomials $\LLT(w)$ induced from a different multiplicative map $\zeta'$ as is the case in \cite{GP}? 
\end{Question}

\subsection{Log-concavity of $h$-coefficients} Finally, motivated by computational results for the codominant case (see \cite[Section 5]{AN}), we make one last conjecture:
\begin{conjecture}
The $h$-coefficients of $\ch(q^{\frac{\ell(w)}{2}}C'_w)$ are log-concave polynomials in $q$ for every $w\in S_n$.
\end{conjecture}

\bibliographystyle{amsalpha}
\bibliography{bibli}

\end{document}